\documentclass[11pt, reqno]{amsart}

\usepackage{epsfig,amsfonts}
\usepackage{subfig}
\usepackage[sort]{cite}
\usepackage{amsmath}
\usepackage{amssymb}
\usepackage{amsthm}
\usepackage{mathrsfs}
\usepackage{graphicx}
\usepackage{colortbl}
\usepackage[mathscr]{euscript}
\usepackage{lscape}
\usepackage{multicol,caption}
\usepackage{lipsum}
\usepackage{float}

 \usepackage{setspace} 

\makeatletter

\newenvironment{figurehere}
  {\def\@captype{figure}}
  {}
\makeatother

\long\def\comment#1{}
\newtheorem{theorem}{Theorem}

\newtheorem{corollary}{Corollary}
\newtheorem{lemma}{Lemma}

\newtheorem{assumption}{Assumption}
\theoremstyle{definition}

\numberwithin{remark}{section}
\newtheorem{example}{Example}

\newcommand{\eps}{\varepsilon}

\newcommand{\be}{\begin{eqnarray}}
\newcommand{\ee}{\end{eqnarray}}

\newcommand{\EC}{\mathscr{E}}

\newcommand{\ba}{\begin{array}}
\newcommand{\ea}{\end{array}}
\newcommand{\bs}{\begin{align}\begin{split}\nonumber}
\newcommand{\bsnumber}{\begin{align}\begin{split}}
\newcommand{\es}{\end{split}\end{align}}

\newcommand{\n}{n}

\renewcommand{\(}{\left(}
\renewcommand{\)}{\right)}
\renewcommand{\[}{\left[}
\renewcommand{\]}{\right]}
\renewcommand{\hat}{\widehat}

\newcommand{\QED}{$\blacksquare$}

\newcommand{\F}{\mathscr{F}}

\newcommand{\N}{\mathrm{N}}
\newcommand{\RF}{\mathrm{RF}}

\newcommand{\Ep}{{\mathrm{E}}}

\renewcommand{\Pr}{{\mathrm{P}}}

\def\\upsilon  {{d}}

\renewcommand{\hat}{\widehat}
\renewcommand{\leq}{\leqslant}
\renewcommand{\geq}{\geqslant}

\newcommand{\convdist}{\overset{d}{\rightarrow} }
\newcommand{\convprob}{\overset{p}{\rightarrow} }

\newcommand{\MI}{ { \mathscr M} }
\newcommand{\PI}{\mathscr P}

\newcommand{\betahat}{\hat \beta}

\newcommand{\Jone}{ \(\log(KL)^{1/2}\)}
\newcommand{\Jtwo}{\( \log(L)^{1/2}\)}
\newcommand{\Jthree}{\(KK^{\alpha_{\rho} } L^{-\alpha_{\mathscr Z}} \)}
\newcommand{\Jfour}{\(n^{1/2} \zeta_0(K) L^{-\alpha_{\mathscr Z}}\)}
\newcommand{\Jfive}{\(   n^{-1/2}K^{\alpha_{\rho}} K^{\alpha_{\Phi}/2}s_0^{1/2} \log(L)^{1/2} +  L^{- \alpha_{\mathscr Z}}K^{\alpha_{\rho}}   \)}
\newcommand{\Jsix}{\(  n^{-1/2}K^{\alpha_{\rho}} K^{\alpha_{\Phi}/2}s_0^{1/2} \log(L)^{1/2} +  L^{- \alpha_{\mathscr Z}}K^{\alpha_{\rho}}    \)}
\newcommand{\Jseven}{\(  n^{-1/2}K^{\alpha_{\rho}} K^{\alpha_{\Phi}/2}s_0K^{\alpha_{I_\Phi}/2} \log(L)^{1/2} +  L^{- \alpha_{\mathscr Z}}K^{\alpha_{\rho}} \)}
\newcommand{\Jeight}{\Jseven}
\newcommand{\Jnine}{\( n^{-1/2} s_0^{1/2}K^{ \alpha_{I_\Phi}/2}  \log(KL)^{1/2} \)}
\newcommand{\Jten}{\(n^{-1/2} s_0^{1/2}K^{ \alpha_{I_\Phi}/2}  \log(L)^{1/2} \)}

\newcommand{\Jsixteen}{K^{1/2}\zeta_0(K)K^{\alpha_{\rho}}L^{-\alpha_{\mathscr Z}}}

\newcommand{\PhiK}{{\Phi_K}}

\author[Damian Kozbur]{  \textbf{ \lowercase{\uppercase{D}amian \uppercase{K}ozbur }}  \\  \textit{\lowercase{  \uppercase{U}niversity of \uppercase{Z}\"urich \\ \uppercase{D}epartment of \uppercase{E}conomics \\ \uppercase{S}ch\"onberggasse 1, 8001 \uppercase{Z}\"urich \\ email: \normalfont{ \texttt{damian.kozbur@econ.uzh.ch}}}}   }

\begin{document}

\title[\textsc{Additively Separable}]{\large \lowercase{\uppercase{I}nference in \uppercase{A}dditively \uppercase{S}eparable \uppercase{M}odels with a \uppercase{H}igh-\uppercase{D}imensional \uppercase{S}et of \uppercase{C}onditioning \uppercase{V}ariables} \\  }

\begin{abstract}
{\begin{onehalfspace} This paper studies nonparametric series estimation and inference for the effect of a single variable of interest $x$ on an outcome $y$ in the presence of potentially high-dimensional conditioning variables $z$. The context is an additively separable model $\Ep[y|x,z] = g_0(x) + h_0(z)$.  The model is high-dimensional in the sense that the series of approximating functions for $h_0(z)$ can have more terms than the sample size, thereby allowing $z$  potentially to have very many measured characteristics.  The model is required to be approximately sparse: $h_0(z)$ can be approximated using only a small subset of series terms whose identities are unknown. This paper proposes an estimation and inference method for $g_0(x)$ called \textit{Post-Nonparametric Double Selection}, which is a generalization of \textit{Post-Double Selection}. Rates of convergence and asymptotic normality for the estimator are derived and hold over a large class of sparse data generating processes.   A simulation study illustrates finite sample estimation properties of the proposed estimator and coverage properties of the corresponding confidence intervals.  Finally, an empirical application to college admissions policy demonstrates the practical implementation of the proposed method. 

\

 \end{onehalfspace}}
\end{abstract} 
\date{First version:  September 2013.  This version: \today. \\  \textit{ } \ \textit{  Correspondence}:  Sch\"onberggasse 1, 8001 Z\"urich, Department of Economics, University of Z\"urich, damian.kozbur@econ.uzh.ch.  \textit{Acknowledgements}:  I thank Christian Hansen, Tim Conley, Matt Taddy, Azeem Shaikh, Dan Nguyen, Dan Zou, Emily Oster, Martin Schonger, Eric Floyd, Kelly Reeve, and seminar participants at University of Western Ontario, University of Pennsylvania, Rutgers University, Monash University, and the Center for Law and Economics at ETH Zurich for helpful comments.  I thank Eric Bettinger for detailed discussion about data used in this paper.  I thank the Ohio Board of Higher Education and Eric Bettinger for providing de-identified data.  I thank Marc Biedermann for support with computations.  I gratefully acknowledge financial support from the ETH Postdoctoral Fellowship}
\maketitle

\pagebreak

\section{Introduction }

Nonparametric estimation in econometrics and statistics is useful for applications in which theory does not provide functional forms for relations between relevant observed variables. In many problems, primary quantities of interest can be computed from the conditional expectation function of an outcome variable $y$ given a regressor of interest $x$,  given by $\Ep[y|x] = f_0(x).$ 
\noindent In this case, nonparametric estimation is a flexible means for estimating unknown $f_0$ from data under minimal assumptions.

Many econometric problems also feature important conditioning information, given through variables denoted here by $z$.
Failing to properly control for $z$ will lead to incorrect estimates of the effects of $x$ on $y$.  Such situations require learning an entire family of conditional expectation functions,  indexed by $z$, given by $$\Ep[y|x,z]=f_{0,z}(x).$$

This paper studies series estimation and inference of $f_{0,z}$ in a particular case characterized by the following two main features.  
\begin{itemize}
\item[1.] $f_{0,z}$ is additively separable in $x$ and $z$, meaning that for some $g_0$ and $h_0$, $$f_{0,z}(x) = g_0(x) + h_0(z).$$
\item[2.] The conditioning variables $z$ are observable and high-dimensional.
\end{itemize} 

\vspace{1mm}

The primary targets of inference for this paper, $\theta_0$, arise from real-valued functionals $a$;  $$\theta_0 = a(g_0) \ \ \text{where} \ \ g \mapsto a(g) \in \mathbb R.$$  Leading examples of such functionals include the average derivative $a(g) = \Ep[g'(x)]$ or the difference of $a(g) = g(x_0^2)- g(x_0^1)$ for two distinct $x_0^1,x_0^2$ of interest. 
This paper constructs estimates $\hat g$ and $\hat \theta = a(\hat g)$, based on a proposed procedure, \textit{Post-Nonparametric Double Selection}, a generalization of \textit{Post-Double Selection} first defined in \cite{BCH-PLM}.  Under regularity conditions, the estimate $\hat \theta$, correctly normalized, is asymptotically Gaussian with suitably negligible bias.  This in turn enables construction of confidence intervals that cover $\theta_0$ to some pre-specified significance level.

 \subsection{Literature Review and Motivation for Additively Separable Models}
Additively separability, when it holds, is a convenient restriction in many econometric problems because any \textit{ceteris paribus} effect of changing $x$ to $x'$ is completely described by $g_0$.  Furthermore, a major statistical advantage when additive separability holds is that individual components $g_0, h_0$ can be estimated at faster minimax rates than a joint estimation of the family $f_{0,z}$ without imposing the additive separability restriction.\footnote{Results on faster minimax rates for separable models exist for both kernel methods (marginal integration and back-fitting methods) and series based estimators.  For a general review of these issues, see e.g. \cite{li:racine:book}.}  

Early study of additively separable models was initiated by  \cite{buja1989}  and \cite{hastie1986}, who describe backfitting techniques.  \cite{ChenHardleLintonSeveranceLossin}  propose marginal integration methods in the kernel context.  \cite{additive:derivatives} and \cite{YANG2003521} consider estimation of derivatives in components of additive models.  \cite{LocalPartitionedReg} develop local partitioned regression that can be applied more generally than the additive model.

Series estimation is convenient in additively separable models because additively separable series expansions are simple to construct.  This paper works with series expansions of the form $(p^K(x),q^L(z) )$ where $ p^K(x) = (p_{1K}(x),...,p_{KK}(x)) $ and $q^L(z) = (q_{1L}(z),...,q_{LL}(z))$
consisting of $K$ and $L$ different transformations of $x$ and $z$ such that their linear combinations can approximate $g_0(x)$ and $h_0(z)$.  References for general large-sample properties of series estimators as well as specializations to additively separable models have been derived by \cite{stone1985}, \cite{cox1988}, \cite{andrews:hang:additiveinteraction}, \cite{eastwoodgallant},  \cite{andrews1991} \cite{newey:series} \cite{victor:newseries} and many others.  Additional information about both kernel and series based estimation can also be found in the reference text \cite{li:racine:book}. 
 
Finally, \cite{huang2010} consider estimation of additively separable models with many additive components relative to the sample size.  The authors propose and analyze a series estimation approach with a Group-Lasso penalty to penalize different additive components.  This paper therefore studies a very similar setting to the one in \cite{huang2010}, but constructs a valid procedure for forming confidence intervals rather than focusing on parameter estimation error.  

\subsection{Literature Review and Motivation for High-Dimensional Estimation}A high-dimensional framework for $z$ allows researchers substantial flexibility in modeling conditioning information when the primary object of interest is $g_0$.  This framework is suited for analysis of particularly rich or \textit{big} datasets with a large number of conditioning variables.\footnote{In many cases, larger set of covariates can lend additional credibility to conditional exogeneity assumptions.  See the discussion in \cite{BCH-PLM}.}  In this paper, high-dimensionality of $z$ is formally defined by the total number of terms in a series expansion of $h_0(z)$.  
This paper considers expansions such that
$$h_0(z) = q_{1L}(z) \beta_{h_0, L,1} + ... + q_{LL}(z)\beta_{h_0, L,L} + o(1), \ \text{possibly} \ L>n$$

\noindent with $\beta_{h_0, L,j}$ denoting the $j$th component of an approximating vector $\beta_{h_0,L}$, the asymptotic $o(1)$ valid for $L \rightarrow \infty$, and $n$ is sample size.  In practice, high-dimensional $z$ can arise in multiple ways.   For example, $z$ can be high-dimensional itself, with many measured characteristics. 
Alternatively, $z$ itself can also have moderate dimension, but any sufficiently expressive series expansion of $h_0$ must have many terms as a consequence of the {curse of dimensionality}.  The theory that follows is agnostic to the way in which $L$ arises.  However, in this paper, discussion will always be cast referencing $q^L(z)$.

High-dimensional techniques are tools for handling estimation problems in which the number of parameters exceeds the sample size.\footnote{Statistical models which are extremely flexible, and thus overparameterized, are likely to overfit the data, leading to poor inference and out of sample performance.  Therefore, when many covariates are present, regularization is necessary.}  Most high-dimensional techniques require additional structure to be imposed on the problem at hand in order to ensure good performance.  One common structure for which reliable high-dimensional techniques exist is sparsity.  Sparsity designates that the number of nonzero parameters in a statistical model is small relative to the sample size.  In this setting, common techniques include $\ell_1$-regularized regression like Lasso and Post-Lasso\footnote{The Lasso is a shrinkage procedure which estimates regression coefficients by minimizing a quadratic loss function plus an $\ell_1$ penalty for the size of the coefficient.  The nature of the penalty gives Lasso favorable property that many parameter values are set identically to zero and thus Lasso can also be used as a model selection technique. Post-Lasso fits an ordinary least squares regression on variables with non-identically-zero estimated Lasso coefficients. For theoretical and simulation results about the performance of these two methods, see \cite{FF:1993} \cite{T1996}, \cite{elements:book} \cite{CandesTao2007} \cite{BaiNg2008}, \cite{BaiNg2009b}, \cite{BickelRitovTsybakov2009},  \cite{BuneaTsybakovWegkamp2006}, \cite{BuneaTsybakovWegkamp2007b} \cite{BuneaTsybakovWegkamp2007}, \cite{CandesTao2007}, \cite{horowitz:lasso}, \cite{knight:shrinkage}, \cite{Koltchinskii2009}, \cite{Lounici2008}, \cite{LouniciPontilTsybakovvandeGeer2009}, \cite{MY2007}, \cite{RosenbaumTsybakov2008}, \cite{T1996}, \cite{vdGeer}, \cite{Wainright2006}, \cite{ZhangHuang2006}, \cite{BC-PostLASSO},  \cite{BuhlmannGeer2011}, \cite{BC-PostLASSO}, among many more. }.  Other techniques include the Dantzig selector, Scad, and Stepwise regression.  To leverage favorable properties of high-dimensional techniques, this paper will also work under sparsity assumptions on series expansions in $q^L(z)$.

The main challenge in statistical inference or construction of confidence intervals after model selection is in attaining robustness to model selection errors. When coefficients are small relative to the sample size (i.e., statistically indistinguishable from zero), model selection mistakes are unavoidable.\footnote{Under some restrictive conditions --- for example, beta-min conditions that constrain nonzero coefficients to have large magnitudes --- perfect model selection can be attained (see \cite{zhao:yu:2006}).}  Such model selection mistakes can lead to distorted statistical inference in much the same way that pretesting procedures lead to distorted inference.
This intuition is formally developed in \cite{potscher} and \cite{leeb:potscher:pms}.  Nevertheless, given the practical value of dimension reduction and the increasing prevalence of high-dimensional datasets, studying robust high-dimensional inference techniques is an active area of current research.  Several recent papers offer solutions to this problem; see, ie., \cite{BellChernHans:Gauss}, \cite{BellChenChernHans:nonGauss}, \cite{ZhangZhang:CI}, \cite{BCH2011:InferenceGauss},  \cite{BCH-PLM}, \cite{vdGBRD:AsymptoticConfidenceSets}, \cite{JM:ConfidenceIntervals}, \cite{BCFH:Policy}, \cite{BCHK:Panel}.\footnote{Citations are ordered by date of first appearance on arXiv.} 

One solution, in \cite{BCH-PLM}, defines the Post-Double Selection procedure for partially linear models $\Ep[y|x,z] = \alpha_0x + h_0(z)$ with high-dimensional $z$.  The procedure reduces the dimension of $q^L(z)$  in two steps that involve Lasso estimation.  First, select all terms $q_{jL}(z)$ from $q^L(z)$ predictive of $x$ using Lasso.  Second, select all terms $q_{jL}(z)$ from $q^L(z)$ that are predictive of $y$ using Lasso.  $\alpha_0$ is then estimated using least squares regression of $y$ on $x$ and the union, $\tilde q(z)$, of selected terms.  This $\hat \alpha$ can be expressed through an orthogonal estimating equation. 
Under regularity conditions, it is sufficiently insensitive to model selection mistakes so that $n^{1/2}(\hat \alpha - \alpha_0)$ is asymptotically Gaussian.

\subsection{Description of Theoretical Contributions in this Paper}

The estimation procedure proposed in this paper is a generalization of Post-Double Selection.  Relative to Post-Double selection, this paper requires the new concept of a \textit{Selection Dictionary}, $\PhiK$, which consists of transformations of $x$.   The selection dictionary $\PhiK$ is used to define a first stage selection step in Post-Nonparametric Double Selection:  all terms $q_{jL}(z)$ from $q^L(z)$ predictive of $\varphi(x)$ for some $\varphi \in \PhiK$ using Lasso are selected and used in a post-model selection step.  As a result, estimation of $g_0$ and $h_0$ depends on three dictionaries $\PhiK$, $p^K(x)$ and $q^L(z)$. 
This paper suggests $\PhiK=\Phi_{K,\text{Span}} = \{ \varphi(x) \in \text{LinSpan}(p^K(x)): \text{sup}_x \| \varphi(x) \|_2 \leq 1 \}$.  The motivation in using $\Phi_{K,\text{Span}}$ is to construct final estimates of $g_0$ that are invariant to different orthogonalizations of $p^K(x)$.  However, the theory here is general enough to analyze the performance many different $\PhiK$, including $\Phi_{K,\text{Simple}} := \{p_{1K}(x),...,p_{KK}(x) \}$.  

Post-Nonparametric Double Selection can be viewed as related to a modified version of Post-Double Selection to general sequences of models, $\Ep[y|x_1,...,x_K,z] = \alpha_{0,1}x_1 + ... +\alpha_{0,K}x_K + h_0(z), \ K\rightarrow \infty$.  
  However, the current problem has added structure: the series expansion for approximating $g_0(x)$ is generated from transformations of $x$ belonging to a fixed finite dimensional space $\mathscr X \subseteq \mathbb R^r$.  The added structure on $\mathscr X$ is used to prove sharper estimation rates than for example, \cite{TU}, which does not have this structure.\footnote{\cite{TU} develop an inferential method, \textit{Targeted Undersmoothing}, applicable to general increasing sequences $x_1,...,x_K$ of variables of interest but never attains $\sqrt{n}$-confidence intervals of functionals $a(\alpha_{01},...,\alpha_{0,K})$.} 

This paper proves convergence rates and asymptotic normality for Post-Nonparametric Double Selection estimates of $g_0(x)$ and $\theta_0$ respectively.  The proofs use techniques in Newey's analysis of series estimators (see \cite{newey:series}) and ideas in Belloni, Chernozhukov, and Hansen's analysis of Post-Double Selection (see \cite{BCH-PLM}), along with careful tracking of a new notion of density of the set $\Phi_K$ relative to the linear span of $p^K(x)$.   When $\PhiK$ has higher density, estimates of $g_0(x)$ and $\theta_0$ become less sensitive to model selection mistakes.  

\section{Series estimation with a reduced dictionary}

This section establishes notation, reviews series estimation, and describes series estimation on a \textit{reduced dictionary}.  Consider data 

\vspace{-5mm}
$$\mathscr D_n = (y_i,x_i,z_i)_{i=1}^n$$ 

\vspace{-3mm} \noindent which are $ n \ i.i.d. \text{realizations of } (y,x,z),$ which satisfy $ \Ep[y|x,z]= g_0(x) + h_0(z). $ 
$y$ are scalar outcome variables, $x$ are explanatory variables of interest, $z$ are conditioning variables, and $g_0,h_0$ are unknown functions to be estimated.  Formal statements on $\mathscr D_n$ are in Section 4.

\subsection{Mechanical Description: Series Estimation and Series Estimation with a Reduced Dictionary} Series estimation of $(g_0,h_0)$ is carried out by least squares regression on series expansions in $x$ and $z$.  Define a dictionary of approximating functions by $$(p^K(x),q^L(z) )$$ 
\noindent where $ p^K(x) = (p_{1K}(x),...,p_{KK}(x)) $
and
$q^L(z) = (q_{1L}(z),...,q_{LL}(z))$
are each a series of $K$ and $L$ functions such that their linear combinations can approximate $g_0(x)$ and $h_0(z)$.  Construct the matrices $P = [p^K(x_1),...,p^K(x_n)]'$,  $ Q = [ q^L(z_1),..., q^L(z_n)]'$,$Y = (y_1, ...,y_n)'$, and let $\hat \beta_{y,(K,L)} = ([ P \ Q] '[P \  Q])^{-1}[ P \  Q]'Y$ be the least squares estimate\footnote{In all instances in this paper, when the required inverse does not exist, a pseudo-inverse may be used.} from $Y$ on $[P \ Q]$.  Let $[\hat \beta_{y,(K,L)}]_g$ be the components of $\hat \beta_{y,(K,L)}$ corresponding to $p^K$.
Then $\hat g(x)$ is defined by
$$\hat g(x) = {p^K(x)}' [\hat \beta_{y,(K,L)}]_g.$$

\vspace{-3mm}\noindent When $K+L>n$, dimension reduction or regularization is required.  Consider reductions  
$$(p^K(x), q^L(z)) \ \ \ \ \overset{\text{reduction}} \longrightarrow \ \ \ \ \  ( \tilde p(x), \tilde q(z)),$$ 

\vspace{-3mm}
\noindent comprised of a subset of the series terms in $(p^K(x), q^L(x))$.{ In this paper, because the primary objects of interest center around $g_0(x)$, it will be the convention to always take $\tilde p(x) = p^K(x)$.}  The \textit{post-model-selection} estimate of $g_0(x)$ is then defined analogously to the traditional series estimate.  Let $\hat \beta_{y, (\tilde p, \tilde q)} = ([ \tilde P \ \tilde Q] '[\tilde P \ \tilde Q])^{-1}[ \tilde P \ \tilde Q]'Y$ where $\tilde P = [\tilde p(x_1),...,\tilde p(x_n)]' = P$,  $ \tilde Q = [\tilde q(z_1),...,\tilde q(z_n)]'$ and as before, let $ [\hat \beta_{y,(\tilde p, \tilde q)}]_g$ be the components of $\hat \beta_{y,(\tilde p,\tilde q )}$ corresponding to $\tilde p$.  Then $\hat g$ is defined by
$$\hat g(x) = {\tilde p(x)}'  [\hat \beta_{y,(\tilde p, \tilde q)}]_g.$$

\vspace{-3mm} \noindent Finally, consider a functional $a(g) \in \mathbb R$ and set $\theta_0 = a(g_0)$.  Here $\theta_0$ is estimated with 

\

\vspace{-12mm}
$$\hat \theta = a(\hspace{.3mm} \hat g \hspace{.3mm}).$$  

\vspace{-2mm}
\noindent In order to use $\hat \theta$ for inference on $\theta_0$, an approximate expression for the variance $ var(\hat \theta)$ is necessary. The expression for the variance will be approximated using the delta method. 
Let $\hat A = \frac{\partial a( p^K(x) ' b) }{\partial b}([ \hat \beta_{y,(\tilde p, \tilde q)}]_g) $.  Let $\MI = \text{Id}_n -  \tilde Q(\tilde Q '\tilde Q)^{-1} \tilde Q'$ be the projection matrix onto the space orthogonal to the span of $\tilde Q$.  Finally, let $\hat  {\mathscr E}  = Y - [\tilde P, \tilde Q ]\hat \beta_{y,(\tilde p, \tilde q)}$.
Estimate $\hat V$ with:
\begin{align*}
&\hat V = \hat A \hat \Omega ^{-1} \hat \Sigma \hat \Omega^{-1} \hat A, \ \ \ \ \ \ 
\hat \Omega = n^{-1}\tilde P' \MI \tilde P, \ \ \ \ \  \ \
\hat \Sigma = n^{-1}\tilde P' \MI \text{diag}( \hat  {\mathscr E}   )^2 \MI \tilde P .
\end{align*} 

\vspace{-4mm}
\noindent The next sections give a dictionary reduction technique and regularity conditions, for which 
$$n^{1/2}\hat V^{-1/2}(\theta_0 - \hat \theta) \rightarrow_d N(0,1).$$ 

\vspace{-3mm}
The practical value of the results is that they formally justify approximate Gaussian inference for $\theta_0$.  A corollary is that for any significance level $\gamma \in (0,1)$, with $c_{1-\gamma/2}$ the $(1-\gamma/2)$-quantile of the standard Guassian distribution, it holds that
\vspace{-3mm}
$$\Pr( \theta_0 \in [ \hat \theta - c_{1-\gamma/2}n^{1/2} \hat V^{-1/2} , \hat \theta - c_{1-\gamma/2}n^{1/2}\hat V^{-1/2} ]  ) \rightarrow \gamma. $$

\subsection{Dictionary Reduction by  Post-Nonparametric Double Selection}

Coverage probabilities of the above confidence sets depend critically on how the dictionary reduction is performed (see \cite{potscher}, \cite{leeb:potscher:pms}).  
This section proposes a procedure for selecting $\tilde q(z)$ which, under regularity conditions, yields asymptotically valid confidence intervals.  

 The dictionary reduction procedure described below relies on Lasso-based model selection for concreteness.\footnote{Other model selection procedures like forward selection with similar properties to Lasso are also possible.}  
 For any random variable $v$ with observations $(v_1,...,v_n)$, the \textit{Lasso}  estimate $v$ on $q^L(z)$ with \textit{penalty parameter} $\lambda$ and \textit{loadings} $l_j$ is defined as a solution  
  $$\hat \beta_{v,L,\text{Lasso}} \in \text{arg} \min_{b \in \mathbb R^L} \sum_{i=1}^n (v_i - q^L(z_i)'b)^2 + \lambda \sum_{j=1}^L |l_j b_j|.$$ \vspace{-1mm}
 The corresponding \textit{selected set} $I_{v,L}$ is defined as $$I_{v,L} = \{ j : \hat \beta_{v,L,\text{Lasso},j}\neq 0 \}.$$  
 
 \vspace{-3mm}
 \noindent Finally, the corresponding  \textit{Post-Lasso} estimator is defined by $$\hat \beta_{v,L,\text{Post-Lasso}} \in \text{arg} \min_{b \in \mathbb R^L: b_j = 0 \ \text{for} \ j \notin I_{v,L}} \sum_{i=1}^n (v_i - q^L(z_i)'b)^2.$$

Lasso is chosen over other model selection possibilities for several reasons.  It is a computationally efficient procedure that has the mechanical ability to set coefficients identically equal to zero.   $|I_{v,L}|$ will be less than $n$ with high probability if a sufficiently high penalty level is chosen.  It is consistent with the previous literature:  \cite{BCH-PLM} used Lasso.    Appendix A in the Supplemental material gives implementation details for choices of $\lambda$ and $l_j$ in various Lasso steps required in this paper following \cite{BellChenChernHans:nonGauss}, who are motivated by allowing for heteroskedasticity.\footnote{The penalty parameter $\lambda$ controls the degree of regularization; $l_j$ are required for dealing with heteroskedasticity of unknown form, but can be directly calculated as in \cite{BellChenChernHans:nonGauss}.  } .






Next, Algorithm 1 defines Post-Nonparametric Double Selection. Let $\PhiK$ be a class of test functions with elements of the form $ \varphi(x)$.  $\PhiK$ is called the \textit{Selection Dictionary.}
Note, Post-Double Selection for a partially linear model $\Ep[y|x,z] = \alpha_0 x + h_0(z)$ in \cite{BCH-PLM} for scalar $x$ is the same as Algorithm 1 with singleton $\Phi_K = \{ x \}$.

\vspace{3mm}

\noindent{\textbf{Algorithm 1}. \textit{Post-Nonparametric Double Selection}}

\begin{itemize}

\item[\textbf{1.}] \textit{First Stage Model Selection Step}.  For each $\varphi \in \PhiK$, perform Lasso regression $\varphi(x)$ on $q^L(z)$ with penalty $\lambda_{\varphi}$ and loadings $l_{\varphi,j}$.  Let $I_{\varphi,L}$ be the selected terms.  Let $I_{\PhiK} = \cup_{\varphi\in\PhiK} I_{\varphi,L}$ be the union set of selected terms.

\item[\textbf{2.}] \textit{Reduced Form Model Selection Step}. Perform Lasso regression $y$ on $q^L(z)$ with penalty $\lambda_{\text{RF}}$ and loadings $l_{\text{RF},j}$.  Let $I_{\text{RF}}$ be the set of selected terms.

\item[\textbf{3.}] \textit{Post-Model Selection Estimation}. Set  $I_{\PhiK + \text{RF}} = I_{\PhiK }\cup I_{\text{RF}}$.  Estimate $\theta_0$ using $\hat \theta$ based on the reduced dictionary $(\tilde p(x),\tilde q(z)) = (p^K(x),  [q_{jL}(z)]_{j \in I_{\PhiK + \text{RF}}} ).$
\end{itemize}

\

There are several possibilities for $\Phi_K$.  This paper proposes the following. 
$$\Phi_{K,\text{ Span}} = \{ \varphi(x) \in \text{LinSpan}(p^K(x)) : \text{sup}_{x \in \mathscr X} \| \varphi(x)\|_2 \leq 1 \},$$
\noindent where $\mathscr X$ is the support of $x$.  This choice of $\PhiK$ is desirable because the resulting first stage model selection is invariant to different orthogonalizations of $p^K(x)$ up to rescaling.\footnote{Identities of covariates selected in the implementation in Appendix A in the Supplement are invariant to rescaling of the left-hand side variable.  Working instead with the condition $ \text{var}( \varphi(x)) \leq 1$, though infeasible because $\text{var}(\varphi(x))$ is unobserved, yields identical estimates $\hat \theta$.  Regularity conditions in Section 3 are stated in terms of $\text{var}(\varphi(x))$. }

The calculation of $I_{\Phi_{K,\text{Span}}}$ does not necessarily require the estimation of a continuum of Lassos.  Instead, it is sufficient to compute the following: 
$$\forall j \leq L: \ \ \max_{\varphi \in \PhiK} \left  | [ \hspace{1mm} \hat \beta_{\varphi (x),L, \text{Lasso}} \hspace{1mm} ]_j \right | .$$

\vspace{-2mm}
\noindent The above quantities correspond to a set of $L$ finite dimensional optimization problems whenever $\Phi_K$ has a finite dimensional parameterization (as does $\PhiK = \Phi_{K,\text{Span}}$).
The observation that $q_{jL}$ is selected into $I_{\PhiK}$ only when there is $\varphi \in \PhiK$ such that $j \in I_{\varphi,L}$ shows that knowing $\max_{\varphi \in \PhiK} \left  | [ \hspace{1mm} \hat \beta_{\varphi (x),L, \text{Lasso}} \hspace{1mm} ]_j \right | $ is sufficient for knowing $I_{\PhiK}$  i.e., for  estimating $\theta_0$, only the identity of selected terms is important (not their Lasso coefficients).  Note: the related reference \cite{BCFH:Policy} does require the full estimation of a continuum of Lasso coefficients, and gives suggestions for approximating this problem using a grid. The grid approach is computationally prohibitive even if $K$ is only moderately large, however.  

Though computing $I_{\Phi_{K,\text{Span}}}$ can be reduced to a finite set of finite-dimensional optimization problems, these problems will be in general nonconvex (though, for fixed $\varphi$, the induced Lasso objective function is convex).  
This paper uses a heuristic strategy in implementation.  For each $j \leq L$, a Lasso regression is run using exactly one test function, $\check \varphi_j \in \Phi_K$.  The choice of $\check \varphi_j $ is made based on being likely to select $q_{jL}$ relative to other $\varphi \in \Phi_K$.  Specifically, for each $j$, $\check \varphi_j$ is set to the linear combination of $p_{1K},...,p_{KK}$ with highest marginal correlation to $q_{jL}$.  Then the approximation to the first stage model selection step proceeds by using $\check I_\PhiK = \bigcup_{j\leq L} I_{ \check \varphi_{j}(x)}$ in place of $I_\PhiK$.  Appendix A in the Supplement contains further details.


\section{Formal Analysis}

This section describes formal conditions which guarantee convergence and asymptotic normality of the Post-Nonparametric Double Selection.  The assumptions are grouped by 4 sections: (3.1) Conditions for nonparametric estimation of $g_0$ and $\theta_0$.   (3.2) High-dimensional estimation conditions. (3.3) Density of $\PhiK$, and (3.4) Additional technical conditions.  

The following definitions are helpful.  Let $g$ be a function on $\mathscr X$.  Define the Sobolev norm $|g|_d = \sup_{x\in \mathscr X} \max_{|c| \leq d} |\partial ^{|c|} g / \partial x^c|$ where the inner maximum ranges over multi-indeces $c$. A vector $X$ is called $s$-sparse if $|\{j: X_j \neq 0 \}| \leq s$.  For a square integrable random variable $v$, $\pi_{q^L}v$ is defined by $\pi_{q^L}v(z) = q^L(z)'\beta_{v,L}$ for $\beta_{v,L}$ such that $\Ep [(v - q^L(z)'\beta_{v,L})^2]$ is minimized.  Equivalently, $\beta_{v,L} = \Ep[q^L(z){q^L(z)}']^{-1} \Ep[q^L(z) v ]$ when well-defined.  For functions $\varphi(x)$, write $\beta_{\varphi,L} = \beta_{\varphi(x),L}$.  Both $K$, $L$ depend on $n$.  All $O$ and $o$ asymptotics are with respect to $n\rightarrow \infty$.

\subsection{{{Conditions for Nonparametric Estimation of $g_0$ and $\theta_0$.}}}

The first group of assumptions defines the data generating process and gives conditions which ensure that if $h_0$ were known a priori, then nonparametric estimation of $g_0$ using $(p^K(x), h_0(z))$ would recover $g_0$ and $\theta_0$ according to the rates stated in Theorems 1 and 2.     Assumption 1-3 are near direct copies of Assumptions 1-3 in \cite{newey:series}, except for the presence of $z$.   Assumption 4 imposes regularity conditions on the functional $a(g)$ and is directly copied from of Assumptions 5--7 in \cite{newey:series}.    Assumption 5 ensures that $g_0$ and $h_0$ are separately identified.

\begin{assumption}[Data and Additive Separability]
The observed data, $\mathscr D_n$, is given by $n$ iid copies of random variables $(x,y,z) \in \mathscr X \times \mathscr Y \times \mathscr Z$ indexed by $1 \leq i \leq n$, so that $\mathscr D_n = (y_i,x_i,z_i)_{i=1}^n.$  $\mathscr Y \subseteq \mathbb R$ and $\mathscr X \subseteq \mathbb R^r$ for some integer $r>0$.  $\mathscr Z$ is a measurable space.  
There is a random variable $\varepsilon$ and functions $g_0$ and $h_0$ such that 
$y= g_0(x) + h_0(z) + \varepsilon, \ \ \Ep[\varepsilon|x,z]=0.$
\end{assumption}

\begin {assumption}[Regularity for $p^K$]
There are nonsingular matrices $B_K$ such that the eigenvalues of $\Omega_{B_Kp^K} =  \Ep\left [ B_K p^K(x) (B_K p^K(x))' \right ] $ are bounded uniformly away from zero in $K$. There are constants $\zeta_0(K)$ with 
$\sup_{ x \in \mathscr X } \| B_Kp^K(x)\|_2 \leq \zeta_0(K)$ 
and $n^{-1}\zeta_0(K)^2K \rightarrow 0$.
\end{assumption}

\begin{assumption}[Approximation of $g_0$] There is an integer $d \geq 0$, a real number $ \alpha_{g_0}>0,$ and  vectors $\beta_{g_0,K}$ such that $|g_0 - {p^K}'  \beta_{g_0,K}|_d = O(K^{-\alpha_{g_0}})$.  
\end{assumption}

\begin{assumption}[Differentiability and Regularity for $a$] The functional $a(g) \in \mathbb R$ is either linear
or the following conditions hold.  $ n^{-1} \zeta_d(K)^4 K^2 \rightarrow 0$ where $\zeta_d(K) = \max_{|c| \leq d } \sup_{ x \in \mathscr X} \| \partial^{|c|} B_K p^K(x)/\partial x^c \|_2$. There is function $D(g; \check g)$ that is linear in $g$.  For some constants $C, \nu>0$ and all $g,\check g, \check{\check g}$ with $| g - g_0|_d < \nu$, $|\check g - g_0|_d < \nu$, $|\check{\check g} - g_0|_d < \nu$, it holds that $|a(g) - a(\check g) - D(g-\check g;\check g) | \leq C(|g-\check g|_d)^2 $ and $|D(g;\check g) - D(g;\check{ \check g})| \leq C |g|_d |\check g - \check{\check g}|_d.$
Either 
{\em[1.]} There is a constant $\bar C>0$ such that  $|D(g)| \leq \bar C |g|_d$. There are $\bar \beta$ dependent on $K$ such that for $\bar g (x) = {p(x)^K}'\bar \beta$,   $\Ep[\bar g(x)^2] \rightarrow 0$ and $D(\bar g) \geq \bar C > 0.$   or
{\em[2.]} There is $\psi(x)$ with $0<\Ep[\psi(x)^2]<\infty$ such that $D(g) = \Ep[\psi(x)g(x)]$ and $D(p_{kK}) = \Ep[\psi(x)p_{kK}(x)]$ for every $k$. There is $\breve \beta$ such that $\Ep[(\psi(x) - {p(x)^K}'\breve \beta)^2] \rightarrow 0$.  Finally, $\bar V = \Ep[\psi(x)^2var (y|x,z)]$ is finite and nonzero.

\end{assumption}

\begin{assumption}[Identifiability\footnote{A stronger population condition can be substituted here.  There are square integrable classes of functions $\mathscr G \subseteq \textbf{ L}^2( \mathscr X  ), \mathscr H \subseteq \textbf{ L}^2( \mathscr Z) \subseteq \textbf{ L}^2( \mathscr X \times \mathscr Z)$ containing $g_0$ and $h_0$ respectively and that the analogue of Assumption 5 holds for all elements in $g_1,...,g_K \in \mathscr G,$ which poses a nonsingular matrix $B_K$ satisfying the condition of Assumption 2, and all projections $\pi_L$ onto the span of any collection $ h_1,...,h_L \in \mathscr H$, for all $K,L$.  }] For each $K$ and for $B_K$ as in Assumption 3, the matrix $ \Ep[ (B_Kp^K(x) - \pi_{q^L}B_Kp^K(z) ) (B_Kp^K(x) - \pi_{q^L}B_Kp^K(z))' ]$ has eigenvalues bounded uniformly away from zero in $K,L$.  In addition, $\sup_{z \in \mathscr Z} \| \pi_{q^L}B_Kp^K(z) \|_2 \leq O( \zeta_0(K))$. \end{assumption}

The objects $\zeta_0(K)$, $B_K$, $\alpha_{g_0}$, etc. characterize the regularity of the nonparametric estimation problem.  They all appear in \cite{newey:series} exactly as they do here, and the reader is referred there for more details.  
The matrix $B_K$ is a matrix that approximately orthogonalizes $p^K(x)$.  $\zeta_0(K)$ measures how large any approximate orthogonalization of $p^K(x)$ must be.   The constant $\alpha_{g_0}$ measures how well $p^K$ can approximate $g_0$ and depends on the smoothness of $g_0$. $D$ is a functional derivative of $a$.  $\psi$ is an influence function, which under regularity conditions, exists precisely when $\theta_0$ can be estimated at the $\sqrt n$ rate.  
Valid values for key quantities are stated for various cases in \cite{newey:series}.  

\begin{example}  Let $x$ be supported on $[-1,1]^r$ with density bounded away from zero. If  $p^K(x)$ are polynomials, then $\zeta_0(K) = O(K)$. If $p^K(x)$ are B-splines, then $\zeta_0(K) = O(K^{1/2})$.  For polynomials and B-splines, if $g_0$ is $m$-times continuously differentiable, $\alpha_{g_0} = m/r$ is valid. 
\end{example}


Assumption 5 imposes limitations on the dependence between $x$ and $z$.  For example, in the case that $\varphi(x) = x$ is an element of $p^K(x)$, this assumption states that the residual variation after a linear regression of $x$ on $z$ is nonvanishing.  More generally, the assumption requires that population residual variation after projecting $p_{kK}(x)$ away from $z$ is nonvanishing uniformly $K,L$.   In particular, constants cannot be freely added to both $g_0(x)$ and $h_0(z)$, requiring a user normalization condition like $g_0(0) = 0$ or $\Ep[g_0(x)] = 0.$

\subsection{High-Dimensional Estimation Conditions}

The next assumptions concern sparse high-dimensional estimation.  Assumption 7 imposes approximate sparsity for $h_0(z)$ and for projections $\varphi(x)$ onto $q^L(z)$. 
Assumption 8 is a high-level conditions on Lasso performance.  

\begin{assumption}[Sparsity] There is a constant $\alpha_{\mathscr Z}>0$, a sequence $s_0 = s_0(n) \geq 1$ and $s_0$-sparse vectors $\beta_{h_0,L,s_0}$ and $\{ \beta_{\varphi,L,s_0}\}_{\varphi \in \PhiK}$ with common support $S_0$ such that
\begin{itemize}
\item[1.]   $\sup_{z \in \mathscr Z} |h_0(z) - {q^L}(z)'\beta_{h_0,L,s_0}|  = O(L^{-\alpha_{\mathscr Z}})$.

\item[2.]   $\sup_{ \varphi \in \PhiK }  \sup_{z \in \mathscr Z} |\pi_{q^L} \varphi(z) - {q^L}(z)'\beta_{\varphi,L,s_0}| = O(L^{-\alpha_{\mathscr Z}}) .$ \end{itemize} 
\end{assumption}

\begin{assumption}[Model Selection]  There are constants $\alpha_{I_\Phi}$ and $\alpha_{\Phi}$ and bounds

\begin{itemize}
\item[1.]$\sup_{\varphi \in \PhiK }  \sum_{i=1}^n (q^L(z_i) '(\beta_{\varphi,L,s_0} - \hat \beta_{\varphi, L ,\text{\em Post-Lasso}} ))^2 $ $= O( K^{\alpha_{\Phi}} s_0\log(L))$
\item[2.]$\sum_{i=1}^n (q^L(z_i) '(\beta_{y,L,s_0} - \hat \beta_{y, L ,\text{\em Post-Lasso}} ))^2 $ $= O(  s_0\log(L))$\vspace{1mm}
\item[3.]$| I_{\RF} | = O(s_0)$\vspace{1mm} and $| I_{\PhiK } | \leq  K^{\alpha_{I_\Phi}}O(s_0) $ \vspace{1mm}
\end{itemize}
which hold with probability $1-o(1)$.  Moreover, if $\Phi_{K} \supseteq \Phi_{K,\text{Span}}$, then $\sup_{\varphi \in \PhiK}$ in \text{\em (2.)} can be weakened to $\sup_{\varphi \in \PhiK' }$ with $\PhiK'$ ranging over all nonrandom fixed $K$-element $\PhiK' \subseteq \PhiK$.
\end{assumption}

\begin{example} For finite choices $|\Phi_{K}|< \infty$, any $\alpha_{I_{\Phi}}$ satisfying $|\PhiK| = O(K^{\alpha_{I_\Phi}}) $ can be shown to satisfy Assumption 7 under regularity conditions standard in the literature,  because $\max_{\varphi \in \PhiK} |I_{\varphi,L}| \leq O(s_0)$ will hold.  Note that by similar reasoning, any $\alpha_{{\Phi}}$ satisfying $\log |\PhiK| = O(K^{\alpha_{\Phi}}) $ can be shown to satisfy Assumption 7 under regularity conditions standard in the literature.  (For example, see \cite{BCFH:Policy}, \cite{BellChenChernHans:nonGauss}). \end{example}
The main difficulty in verifying Assumption 7 with primitive conditions is in showing that $| I_{\PhiK}|$ stays suitably small.   \cite{BCFH:Policy} prove certain performance bounds for a continuum of Lasso estimates under the assumption that $\dim \PhiK$ is fixed and state that their argument would hold for certain sequences $\dim \PhiK \rightarrow \infty$.  Their bounds imply that $\alpha_{\Phi} = 1$ when $\text{dim}(\Phi_K) = K$.   \cite{BCFH:Policy} also proves that the size of the supports of the Lasso estimates, $| I_{\varphi,L} |$ stay bounded uniformly by a constant multiple of $s_0$ which does not depend on $n$ or $\varphi$.  They do not, however, prove that the size of the union $| \cup_{\varphi \in \PhiK} I_{\varphi,L} |$ remains similarly bounded.  Therefore, their results do not imply the existence of a finite value of $\alpha_{I_{\Phi}}$.    
This paper does not derive a bound for $|I_{\Phi_{K, \text{ Span}}} |$ as this would likely lie outside the scope of this project.  A valid alternative with verifiable bounds on the number of selected covariates is $$ \hat \Phi_K = \begin{cases} \Phi_{K,\text{ Simple}} \text{ on the event that } |I_{\Phi_{K,\text{ Span}}}| > t(n) \\ \Phi_{K,\text{ Span}} \text{ otherwise.} \end{cases}$$  for some increasing threshold function $t$ of $n$.    In the simulation study below, overselection is not a problem in the data generating processes considered.


 \subsection{{Conditions on Density of  $\PhiK$}}

The next assumption measures the density of $\PhiK$ in $\text{LinSpan}(p^K)$.  The density is given by the following definition.

\begin{assumption}[Density of $\PhiK$] Each $\varphi \in \PhiK$ satisfies $\text{\em var}(\varphi(x)) \leq 1$. Next, let 
$$ \rho (\PhiK) = \hspace{-2mm} \sup_{ \tiny \left \{ \hspace{-2mm} \begin{array}{c}{ g \in  \mathrm{LinSpan}(p^K)} \\ { \ \text{\em var}(g(x)) \leq 1} \end{array} \hspace{-2mm}  \right \} } \inf_{  \tiny \left \{  \hspace{-2mm}   \begin{array}{c}  \varphi = \eta_1 \varphi_1 + ... + \eta_{k_g} \varphi_{k_g} \\  k_g \geq 1\\ \eta = (\eta_1,...,\eta_{k_g} )\in \mathbb R^{k_g} \\  \varphi_1,...,\varphi_{k_g} \in \PhiK \end{array} \hspace{-2mm}\right  \} }\hspace{-2mm} \ \ \ \sup_{x \in \mathscr X} \ {{|g(x) - \varphi(x)|}}+{L^{-\alpha_{\mathscr Z}}}   \| \eta \|_1$$ 
where $\alpha_{\mathscr Z}$ is as in Assumption 6.  There is a constant $\alpha_{\rho} \geq 0$ with $\rho({\PhiK}) =O(K^{\alpha_{\rho}}L^{-\alpha_{\mathscr Z}}).$ 


\end{assumption}
This density measure has two components.  The first component measures the distance of $g$ to some linear combination $\varphi \in \Phi_K$ using the sup-norm on $\mathscr X$. The second measures the difficulty in expressing $g$ as a linear combination of elements of $\Phi_K$.  This is done using an $L^{-\alpha_{\mathscr Z}}$-normalized $\ell_1$ norm; note that $\varphi= \eta_1 \varphi_1 + ... + \eta_k \varphi_k$ is distant from $g$  if $L^{-\alpha_{\mathscr Z}}\| \eta \|_1$ is large.  The normalization $L^{-\alpha_{\mathscr Z}}$ aligns the scale of $\| \eta \|_1$ to match the sparse approximation error $L^{-\alpha_{\mathscr Z}}$. 
Note that $\rho(\PhiK)$ is as an optimal Lasso objective function value from an $\ell_\infty$-regression of $g$ onto all of $\Phi_K$.
The densest possible $\PhiK$ is generally attained when $\rho(\PhiK) = O(L^{-\alpha_{\mathscr Z}})$ because $\| \eta \|_1$ will generally be at least bounded below by a constant. 

The assumption that $\text{var}(\varphi(x)) \leq 1$ is an analytically convenient normalization.  For first stage model selection procedures which are invariant to rescalings of $\varphi(x)$, there is no loss of generality arising from the fact that $\text{var}(\varphi(x))$ is unobserved in practice.\footnote{Working instead with an assumption on $\sup_x \| \varphi(x) \|_2$ is equally possible, but would add additional notation to the argument because one would need to track the quantity $\sup_x \| \varphi(x) \|_2 / \text{var}(\varphi(x))$ for every $\varphi \in \Phi_K$.}

\begin{example} If for all $g$, there is $\varphi \in \PhiK$ with $\sup_{x\in \mathscr X} | g(x) - \varphi(x) | \leq L^{-\alpha_{\mathscr Z}}$, then $\rho(\PhiK) \leq L^{-\alpha_{\mathscr Z}}$.  In particular, $\rho(\Phi_{K,\text{Span}}) \leq L^{-\alpha_{\mathscr Z}}$.   On the other hand, $\rho(\Phi_{K,\text{Simple}})= O(L^{-\alpha_{\mathscr Z}} K^{1/2})$. 
\end{example}

The quantity $\rho(\PhiK)$ is useful in controlling the accumulation of sparse approximation error when $\pi_{q^L}g(z)$ is approximated by $\eta_1 q^L(z)\beta_{\varphi_1,L,s_0} + ... + \eta_k q^L(z)\beta_{\varphi_{k_g},L,s_0}$.  In univariate Post-Double Selection, such control is never necessary.  As a result, $\rho(\PhiK)$ is a newly introduced concept for this analysis.  An important additional note is \cite{BCFH:Policy} require $\beta_{v,L,s_0}$ to approximate the conditional expectation of $v$ given $z$ rather than just the linear projection when studying a continuum of Lasso regressions.\footnote{The arguments generating Lasso estimation performance bounds in \cite{BellChenChernHans:nonGauss}, in which the heteroskedastic Lasso described here was first proposed, are valid under a sparsity assumption on linear projection coefficients. }
 Their requirements would be much more demanding of $q^L(z)$ in the current context.  For example, $\Ep[x|z]$ having an $s$-sparse linear expansion in terms of $q^L(z)$ does not imply $\Ep[\varphi(x)|z]$ has an $s$-sparse linear expansion.  In fact,  $\Ep[x|z]$ having a general (nonsparse) linear approximation (up to $L^{-\alpha}$-approximation error) does not imply that $\Ep[\varphi(x)|z]$ has a general (nonsparse) linear approximation for arbitrary $\varphi$.  

\subsection{Additional Technical Regularity Conditions}




\begin{assumption}[Sparse Eigenvalues]
For $m\geq 1$ let $\kappa_{\min}(m)(M)$, $\kappa_{\max}(m)(M)$ be minimal and maximal $m$-sparse eigenvalues of the matrix $M$.\footnote{ More explicitly, $\kappa_{\min} (m)(M) := \min_{1 \leq \|\delta \|_0 \leq m} {\delta ' M \delta }/{\|\delta\|_2^2}, \  \kappa_{\max} (m)(M) := \max_{1 \leq \|\delta \|_0 \leq m} {\delta ' M \delta }/{\|\delta\|_2^2}.$} There is $s_{\kappa} = s_{\kappa}(n)$ with $s_0K^{\alpha_{I_\Phi}} = o(s_\kappa)$,   $\kappa_{\min}(s_\kappa)(n^{-1}Q' Q)^{-1}, \  \kappa_{\max}(s_{\kappa})(n^{-1}Q' Q ) =O(1)$ with probability $1-o(1)$.
\end{assumption}

\begin{assumption}[Moment Conditions]  The following moment conditions hold.
\begin{itemize}
\item[1.]  $\Ep[ q_{jL}(z)^2 [B_K(p^K(x) - \pi_{q^L}p^K(z) )]_k^2 ]$ is bounded away from zero uniformly in $K,L$
\item[2.]  $\Ep[|q_{jL}(z)|^3]$ $+(\Ep[q_{jL}(z)^2 \varepsilon^2])^{-1}$   $+\Ep[|q_{jL}(z)|^3| \varepsilon|^3]$    is bounded uniformly in $L$
\item[3.]  $ \Ep[\varepsilon^{4+\delta}|x,z] $ is bounded for some $\delta>0$.  $\text{\em var}(\varepsilon|x,z)$ is bounded away from zero.
\end{itemize}
\end{assumption}

\begin{assumption}[Rate Conditions]  The following rate conditions hold.
\begin{itemize}
\item[1.]$ \textcolor{black}{L^{-\alpha_{\mathscr Z}} \Big [ n^{1/2}  K^{1/2+\alpha_\rho} \zeta_0(K)^2 \log (L) \Big ] = o(1)} \phantom{\Bigg |}$
\item[2.] $\textcolor{black}{n^{-1/2}\[ (n^{ 2/(4+\delta)} + n^{1/6}) \hspace{.5mm} s_0 \hspace{.5mm}  K^{1+ 2\alpha_\rho + \alpha_{\Phi} + \alpha_{I_{\Phi}}/2} \zeta_0(K) \log(LK) \]= o(1)} $
\item[3.] $ n^{1/2}K^{-\alpha_{g_0}}   =o(1).\phantom{\Bigg |}$
\end{itemize}
\end{assumption}

Assumption 9 ensures that the sparse eigenvalues remain well-behaved for all $n$ with high probability.  Informally it states that no small subset of covariates in $q^L$ suffers a multicollinearity problem.  This assumption is commonly used in econometrics and can be shown to hold under more primitive conditions (see \cite{BC-PostLASSO},  \cite{ZhangHuang2006}).  
  Assumption 10 contains mild moment conditions.  Assumption 11 contains rate conditions.\footnote{Assumptions 11.1 and 11.2 are much simplified versions of the rate conditions actually needed in proving the theorems that follow.  Sharper, but more complicated rate conditions are stated in Appendix B in the Supplement. }  Assumption 11.1 is satisfied if $L^{-\alpha_{\mathscr Z}}$, the sparse approximation error, vanishes sufficiently quickly relative to $n$ (and other quantities).  Assumption 11.2 is satisfied when quantities involving $K$ and $s_0$ grow sufficiently slowly relative to $n$.   Assumption 11.3 is an undersmoothing condition, ensuring estimation variance for $\theta_0$ dominates smoothing bias for $\theta_0$, allowing a Gaussian limit as given in Theorem 2.   There is some tension between Assumptions 11.2 and 11.3 because the former requires $K$ to be suitably small relative to $n$ while the latter requires $K$ to be large.  Example 4 shows that possible rates for $K$ are nonempty provided $g_0$ is sufficiently smooth.  

\begin{example}
Suppose $x$ is supported on $[-1,1]$ with density bounded away from zero, $p^K(x)$ are B-splines, $g_0$ is $m$-times continuously differentiable, $ s_0$ is bounded by a polynomial in $\log( n)$, $L$ is bounded by a polynomial in $n$, and $\delta = 8$.  Then Assumptions 11.2 and 11.3 are simultaneously feasible (i.e., there exists a sequence $K$) provided  $m > \frac{3}{2}( \frac{3}{2} + 2\alpha_\rho + \alpha_{\Phi} + \frac{1}{2}\alpha_{I_\Phi})$.
\end{example}


\subsection{Convergence Theorems for Post-Nonparametric Double Selection Estimates}


The first theorem gives convergence rates for $\hat g$ that match those in \cite{newey:series}.   Recall that $\zeta_d(K) = \max_{|c| \leq d } \sup_{ x \in \mathscr X} \| \partial^{|c|} B_K p^K(x)/\partial x^c \|_2 .$  Theorem 1 can be interpreted as verifying that the nonparametric estimation rates derived in \cite{newey:series} continue to hold under a reduced dictionary when Post-Nonparametric Double Selection is used.  Importantly, the bounds below do not depend on $L$.

\begin{theorem}Under Assumptions 1--10 and 11.1--11.2, the Post-Nonparametric Double Selection estimate $\hat g$ for the function $g_0$ satisfies the following bounds.
$$\int (\hat g(x) - g_0(x))^2 dF_0(x) = O_p(n^{-1}K + K^{-2\alpha_{g_0}}).$$
$$|\hat g - g_0 |_d = O_p( n^{-1/2}\zeta_d(K) K^{1/2} + K^{-\alpha_{g_0}}).$$
\end{theorem}


Theorem 2 proves asymptotic normality results for $\hat \theta$.   This theorem justifies approximate Gaussian inference for functionals $\theta_0$ when the conditions in Assumptions 1--12, and particularly, the undersmoothing condition $n^{1/2} K^{-\alpha_{g_0}} = o(1)$, are satisfied.

\begin{theorem} Under Assumptions 1--3, 4.1, and  5--11, the Post-Nonparametric Double Selection estimate for the function $\theta_0$ satisfies  $\hat \theta = \theta_0 + O_p(n^{-1/2}\zeta_d(K)).$  In addition,
\begin{align*}
&n^{1/2} V^{-1/2} (\hat \theta - \theta_0) \overset{d} \rightarrow \N(0,1) \ \ \ \text{and}  \ \ \ n^{1/2} \hat V^{-1/2} (\hat \theta - \theta_0) \overset{d} \rightarrow \N(0,1).
 \end{align*}

\noindent Under Assumptions 1--3, 4.2, and 5--11,  \ 
$n^{1/2}(\hat \theta - \theta_0) \convdist \mathrm{N}(0,\bar V) \ \ \text{and}  \ \ \  \hat V - \bar V \convprob 0.$
\end{theorem}

Inspection of the proof shows that convergence results hold uniformly over $h_0$ and $q^L$ satisfying Assumptions 1--12 under a fixed set of implicit $O$ constants and $o$ sequences.  This paper does not explicitly expound this property for two reasons.  Foremost, the current framework is already broad enough to encompass situations in which model selection mistakes are generic (as $q^L$ depends on $L$).  This is the primary practical concern in post-model selection inference problems.  Second, dealing with sequences $h_{0,n}$ or infinumns $h_0 \in \mathscr H$ would complicate the already notationally heavy exposition.  Note that in contrast, \cite{BCH-PLM} do explicitly state uniformity over a class of functions $h_0 \in \mathscr H$.  The proof is not sufficiently precise to immediately extend to  uniformity statements over functions $g_0$.

Improvements (e.g. \cite{bell:chern:chet:kato:series} and \cite{chen:christiansen:series}) in theoretical analysis of (low-dimensional) series estimators since \cite{newey:series} have relaxed assumptions and given sharper convergence rates.  Whether the theoretical analysis here can be improved to be compatible with theoretical innovations in the previously two cited papers is unknown and an interesting avenue for future research.  

\section{Discussion of Tuning Parameters and Alternative Estimation Strategies}

\subsection{Tuning Parameters}The most important practical tasks involved in any nonparametric estimation problem is calculating sensible tuning parameters from data.  This section recaps and discusses all of the choices of tuning parameters required to implement the entire Post-Nonparametric Double Selection procedure:    
1) choice of $K$ and  $p^K(x)$,  
2) choice of $L$ and  $q^L(z)$,
3) choice of $\Phi_K$, and 
4) choice of Lasso parameters $\lambda_{\RF}$, $l_{\RF,j}$ and $\lambda_{\varphi}$, $l_{\varphi,j}$.  Specific implementation details are available in Appendix A in the Supplement.

This paper considers B-spline bases for $p^K$.  In the simulations below, this paper uses a data-dependent $\hat K$ which is set to $\hat K = \lfloor \hat K_0 \log_{10}(n) \rfloor$, and $\hat K_0$ is a minimizer of BIC.  Appendix A in the Supplement gives further motivation and implementation details for the exact procedure.
A choice of $q^L(z)$ in cases in which $z$ is intrinsically high-dimensional can be taken $q^L(z) = z$.

Choosing $\Phi_K$ optimally is an important problem, which is similar to the problem of dictionary selection, for which the answer may differ by application.  The Span option, $\Phi_{K,\text{ Span}}$, is used in the simulation study as well as in the empirical example.  As mentioned above, the set $\Phi_{K,\text{Span}}$ is invariant to different orthogonalizations of $p^K(x)$.  In addition, $\Phi_{K,\text{Span}}$ performed well in initial simulations.   Several alternative options are possible. These include $\Phi_{K,\text{ Simple}} = \{p_{1K}(x),...,p_{KK}(x)\}$, which corresponds to a direct application of \cite{BCH-PLM}, which formally forgets that $K \rightarrow \infty$.\footnote{An earlier version of this manuscript worked exclusively with $\Phi_{K, \text{Simple}}$.}  Another possible option is $\Phi_{K,\text{ Multiple}}=\{p_{1K}^{(1)}(x),...,p_{KK}^{(1)}(x)\} \cup  ... \cup  \{p_{1K}^{(m)} (x),...,p_{KK}^{(m)}(x)\} \}$ corresponds to using extending $\Phi_{K, \text{Simple}}$ to multiple dictionaries (which are not necessarily used in estimation of $g_0$), indexed $(1),...,(m)$.  For example $\Phi_{K, \text{ Multiple}}$ could include the union of B-splines, orthogonal polynomials, and trigonometric polynomials, all in the first stage selection.  

In order to set up a practical choice of penalty levels, the set proposed above, $\Phi_{K,\text{Span}}$ is considered as a union\footnote{Some dictionaries $p^K(x)$ may not contain a term $p_{kK}(x) = x$.  In this case, $\varphi(x) = x$ can be appended to $\Phi_K$.  In addition, after rescaling, $\PhiK_1 \subseteq \PhiK_2 \subseteq \PhiK_3$.  This causes no additional problems.}: 
$\Phi_{K,\text{Span}} = \PhiK_1 \cup \PhiK_2 \cup \PhiK_3$
where 
$\PhiK_1 = \{ x \}$, 
$\PhiK_2 = \{ p_{1K}(x),...,p_{KK}(x) \}$, 
$\PhiK_3 = \{ \varphi(x) \in \text{LinSpan}(p^K(x)) : \text{var}(\varphi(x) ) \leq 1 \}.$
The reason then for decomposing $\Phi_{K,\text{Span}}$ in this way is to allow the use of different penalty levels on each of the three sets $\PhiK_1, \PhiK_2, \PhiK_3$.  In particular, $\lambda_{\PhiK_1}$ is the penalty for a single heteroskedastic Lasso as described in \cite{BellChenChernHans:nonGauss}.  $\lambda_{\PhiK_2}$ is a penalty which adjusts for the presence of $K$ different Lasso regressions with $K \rightarrow \infty$.   The main proposed estimator sets $\lambda_{\Phi_{K3}} = \lambda_{\Phi_{K2}}$.  This is less conservative than the higher penalty level given in \cite{BCFH:Policy}, who also estimate a continuum of Lassos.\footnote{Note, the normalization that $\|\varphi (x)\|_2 \leq 1$ ensures that $\Phi_{K3}$ is indexed by a compact set and so $\lambda_{\PhiK_3}$ can chosen as described in \cite{BCFH:Policy} to account for a continuum of Lassos.} As a result, any corresponding Lasso performance bounds do not hold uniformly over $\Phi_{K3}$.  Rather, the implied bounds hold only uniformly over any prespecified $K$ element subsets of $\Phi_{K3}$.  This is nonetheless  sufficient for the assumptions in the last section. The simulation study does compare a more conservative (higher) choice for $\lambda_{\Phi_{K3}}$ and finds no substantial difference in inference quality in the data generating processes considered.


The results in this paper are not strong enough to classify all procedures that satisfy Theorem 2 under the above regularity conditions.  Not all procedures with two selection stages work.    For example, consider the two-step procedure that depends on an initial estimate\footnote{Such an estimate may be obtained, e.g., from a Lasso regression $y$ on $(p^K(x), q^L(z))$, setting the penalty loadings over the terms $p^K$ identically 0.} $\check g$ of $g_0$ and that proceeds as Algorithm 1 with the exception that $I_{\check g(x),L}$ replaces $I_{\PhiK}$.  This procedure is simpler and does not require $\PhiK$.  Heuristically,  asymptotic normality would fail in cases in which $j \in I_{p_{kK}(x),L}$ but $j \notin I_{\check g(x),L}$ with high probability for many $j \leq L$.  Such cases arise if $\check g(x) = p^K(x)' \check \beta_g$ with $\max_{k \leq K}[\check \beta_g]_k = o_p(1)$ sufficiently quickly.
\footnote{There may be potential approaches that refine an initial post-nonparametric double selection estimate from Algorithm1 to depend less on $\PhiK$.   
Let $ p_{a}^K$ be a change of basis $p^K$ such that the resulting $ \frac{\partial a( p_a^K(x) ' b) }{\partial b}([ \hat \beta_{y,(p^K_a, \tilde q)}]_{g}) = \textbf{e}_1$, a vector with 1 on the first entry and 0 elsewhere.   Estimating a univariate post-double selection on the first component of the new orthogonalized dictionary results in an estimate of $\theta_0$, which does not require $\PhiK$ in reestimation.}  

\section{Simulation study}


The following simulations illustrate implementation and investigate finite-sample properties of Post-Nonparametric Double Selection.  
The simulation is divided into two parts.  The first compares alternative estimators to Post-Nonparametric Double Selection.   
The second compares Post-Nonparametric Double Selection estimates using different choices for $\Phi_K$.  

\vspace{2mm}
The following process generates the data in each simulation.
{ \onehalfspacing
\begin{align*}\label{sim: true}
& \quad \quad \quad \quad \quad \quad \quad \quad y = g_0(x) + h_0(z) + \varepsilon \textcolor{white}{\Bigg |}\\ 
&g_0(x) = 10 \sin ( 0.1x) -0.5\sin(4\pi x 4^{-x^2})\\ 
&h_0(z) = z' \beta_{h_0,L},   \ \beta_{h_0,L,j} = -0.5 \cdot (-0.65)^{j-1} \textbf{1}_{j \leq s_0} \\
&z_j \sim \N(0,1), \  \ \text{corr}(z_{j_1},z_{j_2}) = 0.25^{|j_1-j_2|} \\
&\varepsilon \sim \N(0,1) \\
&x = 0.15v + 0.0375 - 3.75(\text{stair}(z' \gamma_{0} + v ) + 0.375)F_{\N(0,1)}(10z_{s_0})... \\ & \ \ \ \ \ \ + 3.75(\text{stair}(0.5(z' \gamma_{0} + v)|z'\gamma_0 + v|^{0.25} )( 1-F_{\N(0,1)}(10z_{s_0}) ) \\ 
& \gamma_{0,L,j} = -1.5  (-0.75)^{j-1} \textbf{1}_{j \leq s_0}  , \ v \sim \N(0,1) \\
&\text{stair}( \ \cdot \ ) = 0.25 \frac{ \tanh(12(\  \cdot \ )/2.5) - 12 \lfloor( \ \cdot \ ) /2.5 \rfloor - 6}{2 \tanh( 6) + 0.5 + \lfloor( \ \cdot \ ) /2.5 \rfloor }.
\end{align*}}

\vspace{-5mm}
The study performs 1000 simulation replications for sample size $n \in \{100,150,...,500 \}$ and dimensionality $L \in \{n/2,2n\}$.  Sparsity is set to $s_0 = 6$.   

The data generating process is quite complicated.  It is designed in order to create correlations between the covariates $z$ and various transformations of $x$.  This allows the data generating process to highlight many different statistical problems which can arise using Nonparametric-Post Double Selection and alternative estimation techniques all in one simulation study.  Despite the complicated formulas for the joint distribution of $x$ and $z$, their realizations appear natural. Scatter plots of one sample of $n=500$ showing the respective bivariate distributions between $z_1,...,z_6$ and $x$ are provided in Figure 6.  Figure 5 provides a picture of the graph of $g_0$.
The simulations evaluate estimation of $g_0$ and of $\theta_0$ defined by
 \vspace{-2mm}
 $$\theta_{0} =  \Ep_1[g_0'(x)].$$ 
\noindent The expectation $\Ep_1$ is defined over the uniform probability distribution on the finite set $\{F_{\mathsf N(0,1)}^{-1}(0.01) , \ F_{\mathsf N(0,1)}^{-1}(0.02) , \  ... \ , \  F_{\mathsf N(0,1)}^{-1}(0.99) \}$ where $F_{\mathsf N(0,1)}$ is the standard Gaussian cdf.\footnote{Relative to studying $\Ep[g_0'(x)]$ (with expectation with respect to the unknown distribution of $x$), this choice of functional abstracts away complications arising from the fact that the distribution of $x$ may be unknown.}    

\subsection{Simulation Study Part I:  Comparison of Alternative Estimators}  The first part of the simulation study considers the performances of five estimators\footnote{There are likely other sensible estimators beyond the 5 considered in the simulation section.  As pointed out by an anonymous reviewer, such estimators may include propensity score matching on a continuous variable.  Though such an approach may work well, the context here is not exactly the same as usually seen in propensity score matching.  In particular, the assumptions here do not require unconfoundedness conditions.  In addition, propensity score techniques are most commonly applied to discrete treatement variables.  There is some work on propensity score matching with a continuous treatment; for example, see \cite{imbens:prop:matching:continuous}, who require the estimation of the conditional density of treatment.  In the high-dimensional setting, estimating the conditional density of $x$ given $z$ would introduce complications beyond the scope of this paper.}     for $g_0$ and $\theta_0$.  Each estimator is a reduced series estimator based on initial dictionaries consisting of a cubic spline expansion $p^K(x)$ for $g_0(x)$ and a linear expansion $q^L(z) = z$ for $h_0(z)$.  Detailed implementation descriptions are provided in Appendix A in the Supplement.

\begin{itemize}
\item[1.] \textbf{Oracle.}  Sets $\tilde q(z) = (z_1,..,z_{s_0})$.  This estimator is an infeasible benchmark for comparison to estimates in which the correct support is known.  
\item[2.] \textbf{Span PND.}  Selects $\tilde q(z)$ using Post-Nonparametric Double Selection with $\Phi_K = \Phi_{K,\text{Span}}$ given by the Span option. 
\item[3.] \textbf{Naive.}  Selects $\tilde q(z)$ in one model selection step using Lasso of $y$ on $q^L(z)$.
\item[4.] \textbf{OLS.}  Sets $\tilde q(z) = z$; no dictionary reduction (only calculated provided $L<n.$)
\item[5.] \textbf{TU.}  Implements an alternative inferential procedure for dense functionals of high-dimensional parameters; Targeted Undersmoothing TU(1). 
\end{itemize}


Comparison of estimators 1--4 is standard in the post-model selection econometrics literature.  The oracle estimator should be seen as a benchmark which is known to provide good estimates if the true set, $S_0$, was known.  The Naive estimator is expected to perform poorly since it is not a uniformly valid estimator and susceptible to size-distortions arising from model selection mistakes.   OLS is expected to perform poorly due to potential problems related to overfitting.  
Estimator 5 is a procedure called \textit{Targeted Undersmoothing} which looks to correct distortions in inference from model selection mistakes.  Targeted Undersmoothing appends covariates which significantly affect the value of the functional $\hat \theta = a(\hat g)$ to an initially selected model (see \cite{TU}).  It is appropriate for functionals of high-dimensional models which depend on a growing number of parameters (dense functionals) and is therefore a potentially sensible procedure for inference for $\theta_0$.  

The simulation results report several quantities which measure the performance of each estimator.  The results report standard deviation of the estimates $\hat \theta$, bias of the estimates for $\theta_0$, confidence interval length for estimates for $\theta_0$, rejection frequencies under the null for $\theta_0$ at the 5\% level, mean number of series terms $K$ used, mean number of series terms selected from the original $L$, and integrated squared error for $g_0$.  The simulation results are reported in Figure 1 for $L=n/2$ and Figure 2 for $L = 2n$.  The figures display the above-mentioned simulation results for each $n=100,...,500$ with $n$ changing over the horizontal axis.\footnote{Note that since $s_0$,  the magnitude of coefficients $\beta_{h_0,L}$ and the joint distribution between relevant covariates are all fixed in the simulations as $n \rightarrow \infty$.  Therefore, for sufficiently large $n$, all relevant covariates would be identified with high probability, and all of the post-model selection estimators would perform similarly.  This simulation study therefore is identifying differences in finite sample performance.}     Note also that across some of the estimators, some of the reported quantities will be identical.  For example, the point estimates for TU are identical to the Naive point estimates.  The selected $K$ is identical for the Naive and Post-Nonparametric Double Selection estimates.

In all of the simulations, the Post-Nonparametric Double Selection estimates behave similarly to the Oracle estimates.  The OLS estimates have wide confidence intervals relative to the Post-Nonparametric Double Selection estimation, but have similar coverage properties.   The final estimator, TU, is conservative in terms of coverage, with substantially larger intervals in every case.

On the other hand, the Naive estimator has poor coverage properties.  For the Naive estimator, after failing to control for the correct covariates, the increase in $K$ leads to an increasing bias.  This highlights the fact that simply producing undersmoothed estimates of $g_0$ by increasing $K$ may not be adequate for reducing bias and making quality statistical inference possible in the high-dimensional setting.

\subsection{Simulation Study Part II:  Comparison of Alternative $\PhiK$}

The second part of the simulation study compares four Post-Nonparametric Double Selection estimators that use different specifications for $\PhiK$.

\begin{itemize}
\item[1.] \textbf{Span PND.}   Post-Nonparametric Double with $\Phi_{K,\text{Span}}$.
\item[2.] \textbf{Conservative Span PND.}  Uses $p^K$ and $\Phi_K$ as in the Span option, but in the decomposition $ \Phi_{K,\text{Span}} = \PhiK_1 \cup \PhiK_2 \cup \PhiK_3$, the penalty applied to $\PhiK_3$ is more conservative, explicitly aimed at achieve Lasso performance bounds that hold uniformly over all of $\PhiK$.
\item[3.] \textbf{Simple PND.}   Post-Nonparametric Double with $ \Phi_{K,\text{ Simple}}$.
\item[4.] \textbf{Alternative Spline Basis Simple PND.}   Post-Nonparametric Double with a different basis for selection.  A QR decomposition is applied to $P$ in order to obtain orthonormal columns.  Next, $\Phi_K = \Phi_{K,\text{ Simple}}$ is used on the new orthogonalized data.  Importantly, the new $P$ spans the same $K$-dimensional linear space in $\mathbb R^n$ as in the 3 previous estimators.
\end{itemize}

The estimates for the second part of the simulation are presented in Figures 3--4.  Note that all estimators are identical with regard to $K$, hence, only one curve is visible in the corresponding plots.  In addition, the Conservative Span and Span estimators have very similar performance in terms of standard deviation, bias, interval length, rejection frequency, and integrated squared error.  The two estimators are practically indistinguishable except in terms of the number of elements of $q^L$ they select.  They do not give numerically identical estimates or confidence intervals.  Their differences are too small to be seen in Figures 3--4.

There are noticeable differences in the performance of the  estimators.  The Span option is able to identify the highest number of relevant covariates, followed by the Conservative Span option, the Simple option, and the Alternative Spline Basis Simple option.  The Span, Conservative Span, and Simple Post-Nonparametric Double Selection procedures exhibit favorable finite sample properties for this data generating process.  In particular, for those estimators, the calculated rejection frequencies move toward 5\% as $n$ increases.  
By contrast, the Alternative Spline Basis Simple Post-Nonparametric Double Selection procedure has very poor finite sample performance.  It is unlikely that the projection of the new orthogonalized basis onto $q^L$ has a good sparse representation.  This causes increased model selection mistakes in the first stage.  Unlike in the partially linear model, these mistakes can accumulate to cause more severe bias as the number of first stage selection steps is growing with $K$.  Note that the Alternative Spline Basis estimator has similar performance to the Naive estimator in the first part of the simulation study.
The Span and the Conservative Span options offer an opportunity to potentially add robustness.  There is no evidence from this simulation study that using the Span option overselects conditioning variables to the extent that rejection frequencies become severely distorted or variability increases to an undesirable level relative to other options.





\begin{figure}[h]
\caption{Simulation study design details} 
\flushleft \singlespacing
\includegraphics[scale = .10]{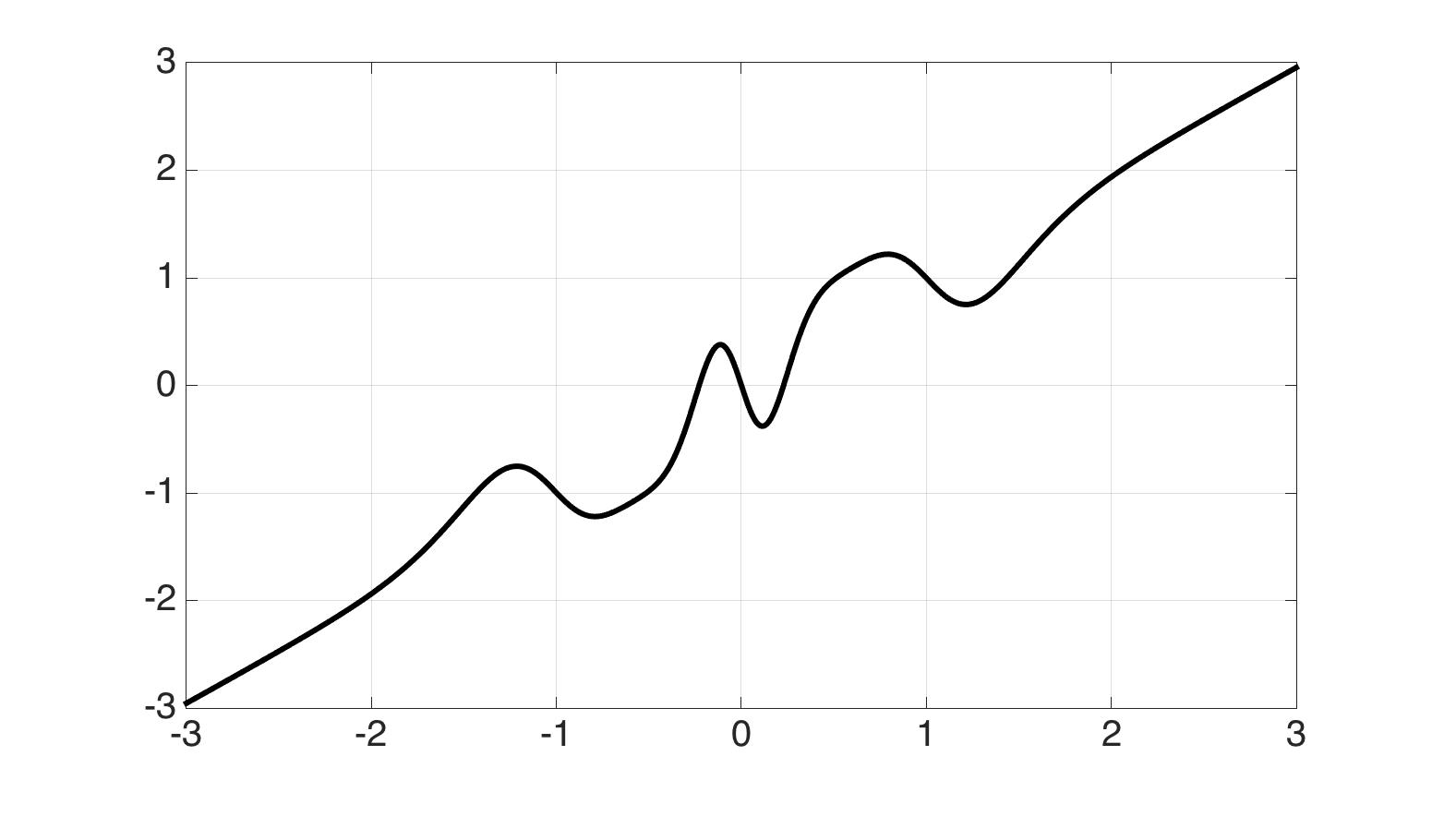}  \includegraphics[scale = .12]{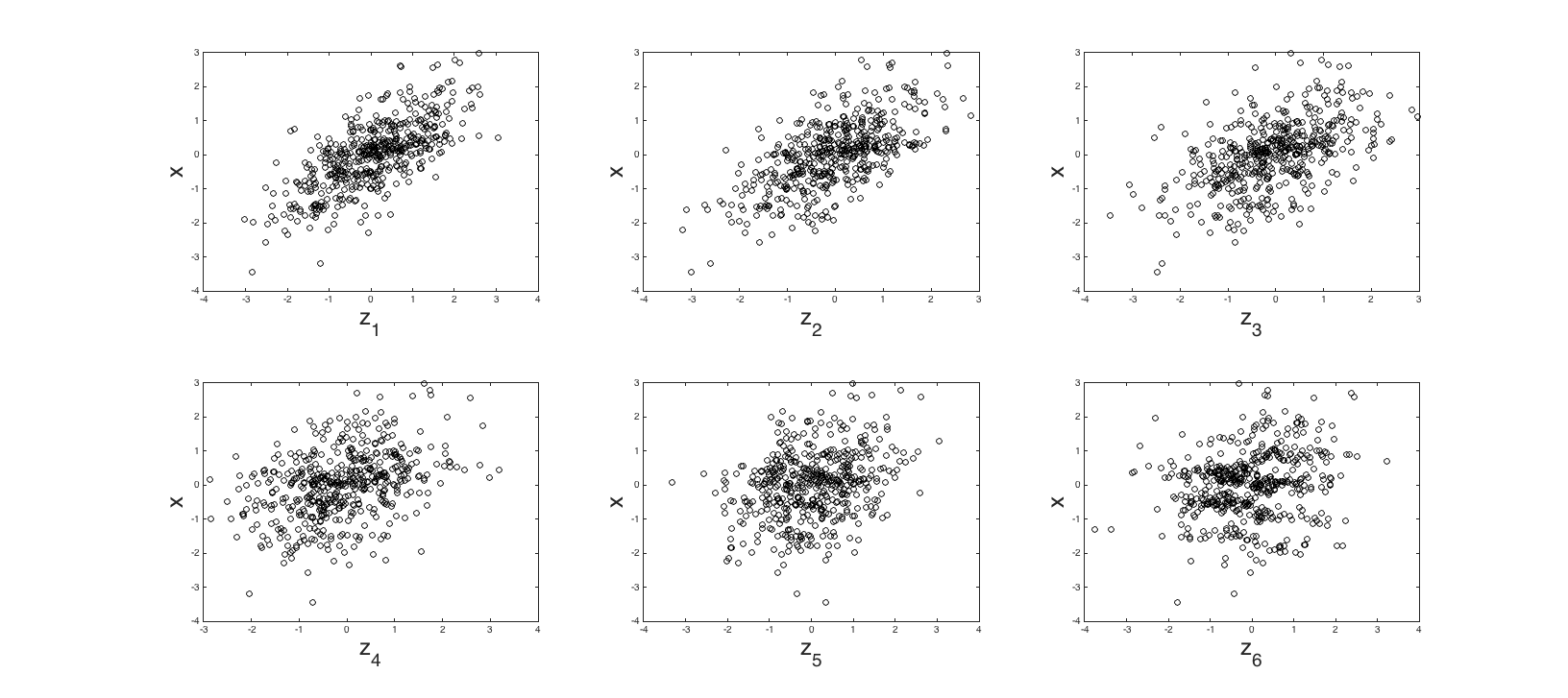}
{  \footnotesize 
\center{The left panel depicts the function $g_0$ used in simulation.  The right panel depicts the joint distribution between $x$ and the first $s_0 = 6$ covariatest.  The plots are generated by one sample of size $n=500$.}}
\end{figure}

\singlespacing

\begin{landscape}
\begin{multicols}{2}

\vspace{-1mm}
 \begin{figurehere}
 \caption{Simulation Results}

\


\includegraphics[scale=.29]{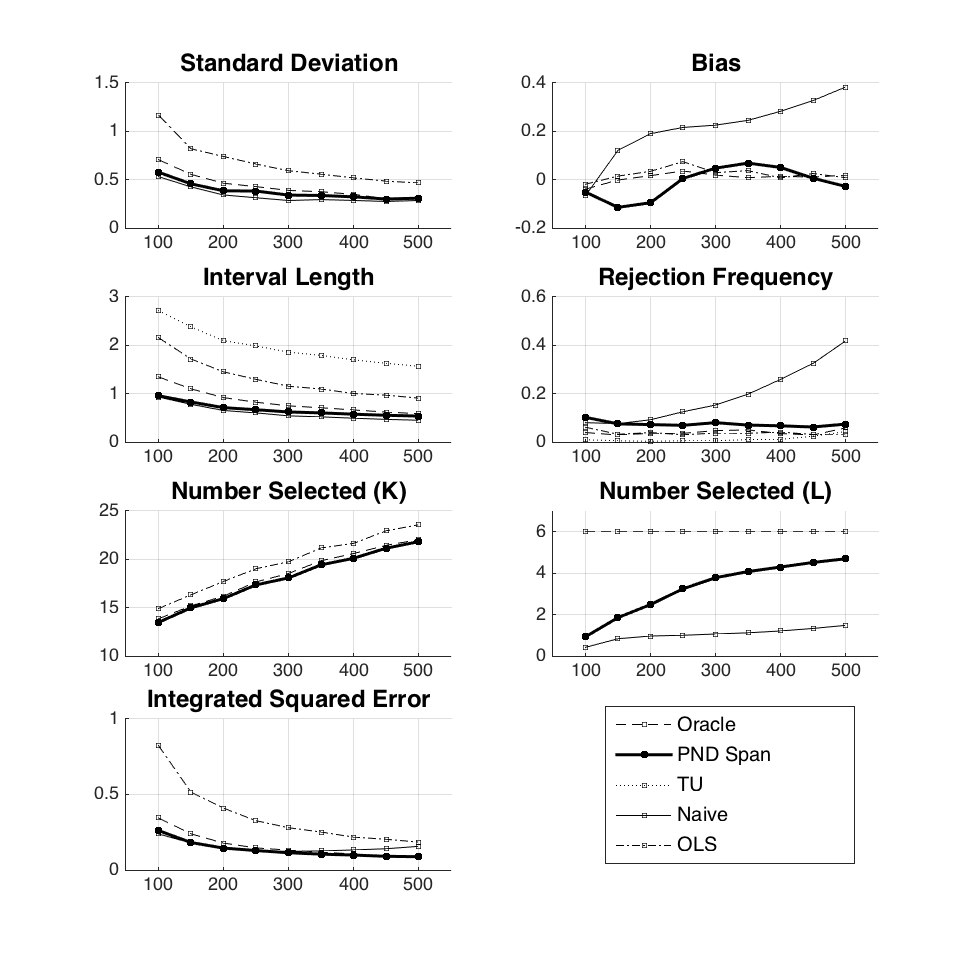} 
\scriptsize
Simulation results for the estimation of $g_0$ and $\theta_0$ with $n=100,150,...,500$ with $s_0 = 6$ and $L=n/2$.  Estimates presented for estimators: Oracle, Post-Nonparametric Double (PND Span), Naive, OLS, and Targeted Undersmoothing (TU).  Plot 1: standard deviation of the estimates for $\theta_0$.   Plot 2: bias of the estimates for $\theta_0$. Plot 3:  confidence interval length for estimates for $\theta_0$.  Plot 4: rejection frequencies under the null for $\theta_0$ for a 5\% level test.   Plot 5:  mean number of series terms $K$ used.  Plot 6:  mean number of series terms from $L$ selected.  Plot 7:  root mean integrated squared error for $g_0$.   Figures are based on 1000 simulation replications.  $n$ is indexed by the horizontal axis.

\


 \end{figurehere}

\

\begin{figurehere}
\caption{Simulation Results}

\


\includegraphics[scale=.28]{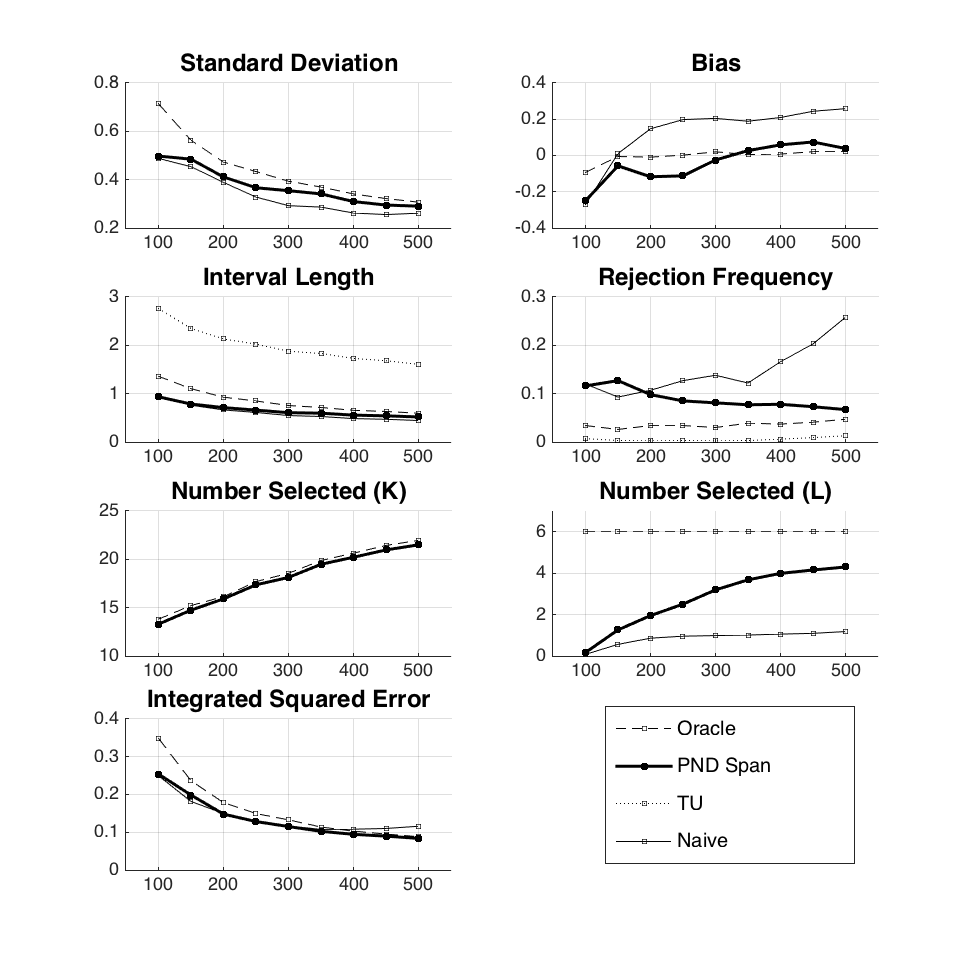}
\scriptsize
\flushleft
Simulation results for the estimation of $g_0$ and $\theta_0$ with $n=100,150,...,500$ with $s_0 = 6$ and $L=2n$.  Estimates presented for estimators: Oracle, Post-Nonparametric Double (PND Span), Naive, and Targeted Undersmoothing (TU).  Plot 1: standard deviation of the estimates for $\theta_0$.   Plot 2: bias of the estimates for $\theta_0$. Plot 3:  confidence interval length for estimates for $\theta_0$.  Plot 4: rejection frequencies under the null for $\theta_0$ for a 5\% level test.   Plot 5:  mean number of series terms $K$ used.  Plot 6:  mean number of series terms from $L$ selected.  Plot 7:  root mean integrated squared error for $g_0$.   Figures are based on 1000 simulation replications.  $n$ is indexed by the horizontal axis.

\


 \end{figurehere}

\

\

 \begin{figurehere}
\caption{Simulation Results}

\


 \includegraphics[scale=.28]{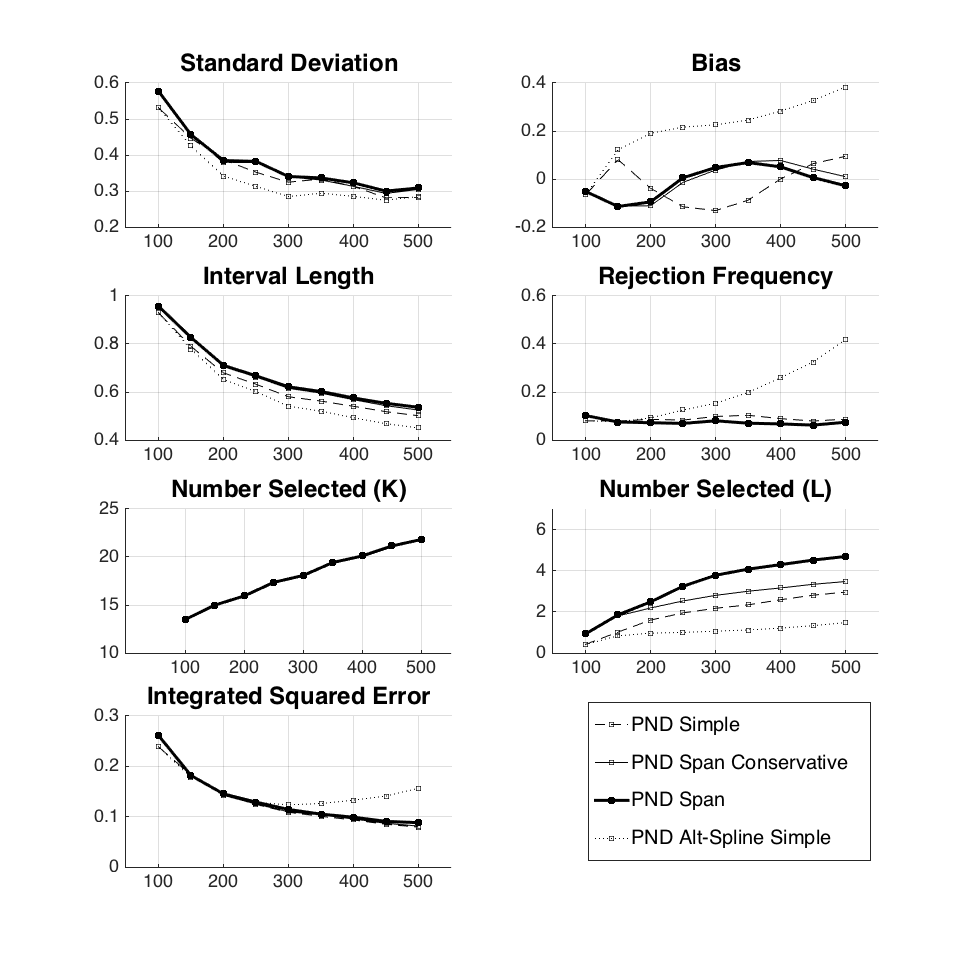}  
\scriptsize
\flushleft
Simulation results for the estimation of $g_0$ and $\theta_0$ with $n=100,150,...,500$ with $s_0 = 6$ and $L=n/2$.  Estimates are presented for four Post-Nonparametric Double Selection (PND) estimators, Simple, Span, Conservative Span, and Alternative Spline Simple as described in the text. Plot 1: standard deviation of the estimates for $\theta_0$.   Plot 2: bias of the estimates for $\theta_0$. Plot 3:  confidence interval length for estimates for $\theta_0$.  Plot 4: rejection frequencies under the null for $\theta_0$ for a 5\% level test.   Plot 5:  mean number of series terms $K$ used.  Plot 6:  mean number of series terms from $L$ selected.  Plot 7:  root mean integrated squared error for $g_0$.   Figures are based on 1000 simulation replications.  $n$ is indexed by the horizontal axis.


\


 \end{figurehere}

\begin{figurehere}
\caption{Simulation Results}

\


\includegraphics[scale=.28]{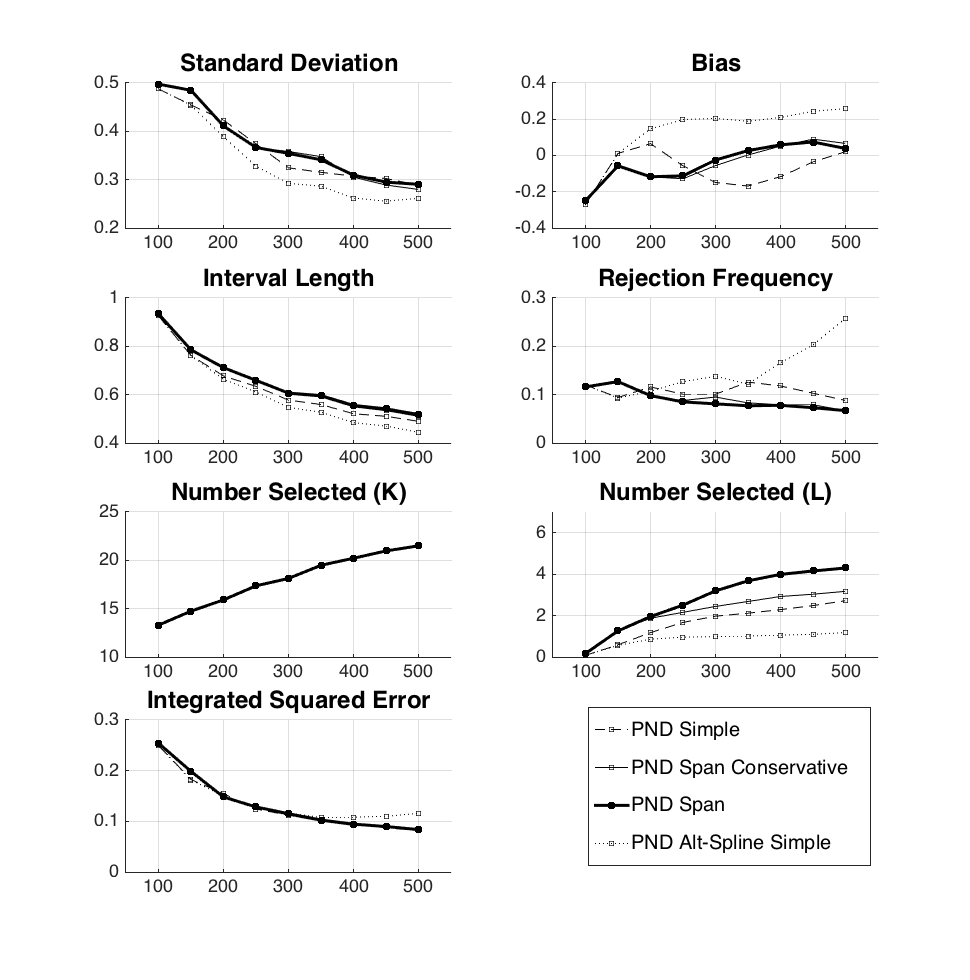}
\scriptsize
\flushleft
Simulation results for the estimation of $g_0$ and $\theta_0$ with $n=100,150,...,500$ with $s_0 = 6$ and $L=2n$.  Estimates are presented for four Post-Nonparametric Double Selection (PND) estimators, Simple, Span, Conservative Span, and Alternative Spline Simple as described in the text. Plot 1: standard deviation of the estimates for $\theta_0$.   Plot 2: bias of the estimates for $\theta_0$. Plot 3:  confidence interval length for estimates for $\theta_0$.  Plot 4: rejection frequencies under the null for $\theta_0$ for a 5\% level test.   Plot 5:  mean number of series terms $K$ used.  Plot 6:  mean number of series terms from $L$ selected.  Plot 7:  root mean integrated squared error for $g_0$.   Figures are based on 1000 simulation replications.  $n$ is indexed by the horizontal axis.

\


 \end{figurehere}

\end{multicols}
\end{landscape}





\doublespacing
\section{College performance and the ACT College Entrance Exam}

This section illustrates Post-Nonparametric Double Selection with an application to learning the relation between college performance and the ACT college entrance exam.  

Understanding factors that predict college performance in students is important for identifying opportunities for improvement in college outcomes.  As an example (see \cite{radford:etal:postsecondary}; see also \cite{BettingerEvansPope}, who also cite this example), 35\% of students who started a postsecondary program in the Fall of 2003 had dropped out and earned no degree six years later.  One key factor identified as contributing to this dropout rate is a mismatch problem: due to noise in the admissions process, capable students are under-placed into less-selective institutions where they are less likely to graduate (see \cite{bowen:chingos:mcpherson}).  This mismatching is costly to the students who drop out, students who could have enrolled in their place, and to the respective colleges.

In deciding which students to admit, many American universities use the ACT college entrance exam. 
The ACT covers four subjects: Mathematics, English, Reading, and Science. Scores from 1--36 on each of these components and a composite score are provided to colleges for their use in admission decisions. The composite score is the rounded average of the scores on the four individual sections. Most colleges report only using the ACT composite score in admissions decisions (\cite{BettingerEvansPope}).   

Correct use of the information in the ACT can improve screening of college students for admissions, and therefore lead potentially to more efficient educational outcomes.   At the same time, relying too much on the ACT exam can also miss important student characteristics that may indicate their potential for success in college.  The question of how to most appropriately use information from the ACT exam is therefore relevant to active admissions committees across the country.  As one recent example highlighting the uncertainty of the value in the ACT, The University of Chicago recently decided to drop the ACT requirement (see \cite{uchi_drops_act}).

In a 2009 paper, Bettinger, Evans, and Pope (\cite{BettingerEvansPope}) studied the predictive capabilities of components of the ACT exam on college performance outcomes.   Their findings are based on linear regression analysis and showed that the Math and English components of the ACT contain more predictive power on various student outcomes of interest including college GPA and dropout rates, conditioning on observables which would have been available to admissions committees.  Conversely, they find that the Science and Reading components add much less explanatory power.  A practical recommendation based on their findings is that college admissions committees can improve screening practices by focusing the ACT Math and English components rather than on the ACT composite.   A subset of their findings is replicated in Table 1 and described below.  

This paper considers the relation between college GPA, ACT scores, and observables that would have been available to admissions committees and allows nonlinear dependence between these variables.  Specifically, this paper estimates the following model for GPA, ACT, and conditioning variables $z$.
\begin{align*}  \text{GPA} & = g_0^{\text{math}}(\text{ACT}_{\text{math}}) + g_0^{\text{eng}}(\text{ACT}_{\text{eng}})   + g_0^{\text{read}}(\text{ACT}_{\text{read}})  + g_0^{\text{sci}}(\text{ACT}_{\text{sci}})  + h_0(z) + \eps
\end{align*}
 Here, GPA is the outcome variable and is defined as college grade point average at the end of the first year.  The regressor of interest ACT = $(\text{ACT}_{\text{math}},\text{ACT}_{\text{eng}},\text{ACT}_{\text{read}},\text{ACT}_{\text{sci}})$, consists of the Math, English, Reading, and Science components of the ACT exam.  Note that the additional additive separability imposed on $g_0$ is also covered under the general framework studied in the previous sections.  Finally, $z$ are control variables that include individual specific data recording (1) which college the student attended, (2) intended major, (3) race, (4) gender, and (5) information about the individual's high-school GPA.  The data used here were compiled by the authors of \cite{BettingerEvansPope} in collaboration with the Ohio Board of Regents (now known as the Ohio Board of Higher Education).
 The particular extract used here contains information on a 1999 cohort of college students who matriculated into an Ohio university.  More details are given below and are also available in \cite{BettingerEvansPope}.

\subsection{Results}

Table 1 replicates regression analyses in \cite{BettingerEvansPope} and presents two motivational specifications which include quadratic nonlinearities (one including and one excluding conditioning variables, Columns 4-5).
Column 3 replicates the results in \cite{BettingerEvansPope}, which is a prediction of college GPA with a linear specification in ACT, and controlling for intended college major, race, gender, highschool GPA, and college campus.  According to the estimates in Column 3, a 1 point increase in component scores for Math, English, Reading, Science are associated with .0291 (std err .01071),  .0388 (std err .0086), -.0057 (std err .0073) and .0405 (std err .0107) increases in GPA.
As was noted by  \cite{BettingerEvansPope}, the science and reading components of the ACT are weak predictors of first year GPA;  in the linear specification, both show small estimated coefficients relative to the Math and English components.  Both are statistically insignificant.  The results in column 5 which allow a quadratic-only nonlinearity suggest that nonlinearities are present.  For example, the quadratic for the ACT Math component is positive and significant with an estimate of .0008 (std err .0002).    Sensitivity of the coefficients to nonlinearity suggests nonparametric analysis may be helpful.

Table 2 presents Post-Nonparametric Double Selection estimates of several weighted average derivatives.  In particular, for a component $\text{ACT}_{\text{c}}$ of the ACT, and for a quartile Q determined by the ACT composite, estimates are presented for  $\theta_0 = \int_{Q} \frac{ \textcolor{white}{|_|} \partial g_0  \textcolor{white}{|_|} }{ \textcolor{white}{|^|}  \partial{\text{ACT}_{\text{c}} \textcolor{white}{|^|} }}dF$.  For simplicity (as in the simulation section), $F$ is taken as the empirical distribution of the ACT composite scores.
The dictionary expansion $q^L(z)$ is defined using the same baseline variables that \cite{BettingerEvansPope} used\footnote{ The present analysis uses a courser grouping 8 of distinct major, still coded in the original data, rather than the grouping of 45 distinct majors used by \cite{BettingerEvansPope}.}.  There are a total of 39 indicator variables characterizing the control variables mentioned above.  These are interacted using three different interaction expansions: the \textit{and} expansion, \textit{or} expansion, and \textit{Hadamard-Walsh} expansion\footnote{ Let $v_{i1},...,v_{ik}$ denote the original set of indicator variables. For each subset $A \subseteq \{1,...,k\}$ the corresponding \textit{and, or, and Hadamard-Walsh} interactions transformations of $(v_{i1},...,v_{ik})$ are given by $\psi_A^{and}(v_{i1},...,v_{ik}) = \prod_{j \in A} v_{ij}, \ \psi_A^{or}(v_{i1},...,v_{ik}) = \textbf{1}\{ \sum_{j \in A} v_{ij} \geq 1\},  \   \psi_A^{H\text{-}W}(v_{i1},...,v_{ik}) = (-1)^{|A \cap \{j:v_{ij}=1\} | }$.  Interactions of order $\leq 3$ correspond to $|A|\leq 3$.  See \cite{TU} for a discussion of the Hadamard-Walsh basis.}, all containing interactions of order $\leq 3$.  Terms with standard deviation $\leq .05$ were then excluded from the analysis. The resulting model is oversaturated.  Practically, such over-parameterization arising from the union of different interaction expansions may help give a better sparse approximation than a single interaction expansion.  The total number of terms is $L=21799$ compared to $n = 21757$.  More implementation and documentation details are available in Appendix A in the Supplement, replication files available from the author, and in \cite{BettingerEvansPope}.

The overall average derivative estimates for the Math and English components, .0256 (std err .0021) and .0159 (std err .0018) are both large and significant at the 95\% level.  Overall average derivative estimates for the Reading and Science areas are  .0027 (std err .0015) and -.0015 (std err .0021), which are smaller and insignificant at  the 95\% level.   

The analysis highlights nonlinearities in $g_0$.  All components of the ACT have statistically significant average derivative estimates in the top quartile (75--100).  The estimated average derivate for the ACT Math component in the highest percentiles (75--100) is .0374 (std err .0031).   This estimate is higher than in all of the lower percentile bins.  Average derivative estimates in  percentiles (50--75), (25--50), (0--25) are .0203 (std err .0020), 0173 (std err .0025), 0276 (std err .0047).    Correcting for 24 different possible comparisons (6 comparisons within each subject component) using a Bonferroni adjustment, the average derivatives is statistically different at the 95\% level in the ACT math component both between percentile ranges (75--100) and (50--75) and percentile ranges (75--100) and (25--50).  

In the percentile range (75--100), English component scores,  .0188 (std err .0032), and Reading component scores,  .0074 (std err .0027) are also significantly and positvely associated with GPA as measured by average derivatives, but with smaller magnitude as compared with Math component scores.  The average derivative estimate for the ACT Science component is negative at -.0067 (std err .0032) and significant at the 95\% level.  The general pattern of large average derivatives in the extremes as seen in Math component estimates persists in the English, Reading, and Science component estimates.  However, correcting for multiple testing, assuming 4 hypothesis tests for the 4 components of the ACT, this component is no longer significantly associated to GPA at the 95\% level.  After correcting for 24 different possible comparisons, there is no statistically significant difference in average derivatives across percentile ranges for English, Reading, and Science components.


\subsection{Limitations}
This section describes limitations in the model assumptions and statistical procedure as they apply to the ACT data.  Addressing these limitations in future work could lead to improved type I error properties and interpretability. 

(1) The above analysis has a potential sample selection bias because only students who matriculate into a four-year college are observed.  \cite{BettingerEvansPope} discuss this problem at length.  They note a large percentage of students in the Ohio public school education system remain in that system throughout their education.  They also perform a validation exercise using 2006 data.  
In addtion, $\Ep[ \text{GPA} | \text{ACT}, z, \text{Matriculated} ]$ is nevertheless of  interest in its own right, because it may help identify at risk students within the pool of students that matriculate. 

(2) An implication of additive separability is that the derivative of $g$ at any value $x$ cannot depend on $z$.   There currently exists no widely implemented post-model selection specification testing framework for the current additively separable setting.   Though, in principle, tests for this assumption could be potentially an interesting avenue for future research.  There are several difficulties in creating such a test.  For example, specification test via looking at quantities of the form $\frac{1}{n} \sum_{i=1}^n \hat \eps_i \mathsf v(x_i,z_i)$, in which $\mathsf v$ is an interaction term is problematic:  the convergence $\max_{i \leq n} | \hat \eps_i - \eps_i| \rightarrow_p 0$ may be at a slower rate than the convergence $\frac{1}{n} \sum_{i=1}^n \eps_i \mathsf v(x_i,z_i) \rightarrow_p \Ep[\eps \mathsf v(x,z)]$.  

(3)  The theory developed in the preceding sections of the paper assumes that the functional $a$ is a fixed, known functional $g \mapsto a(g) \in \mathbb R$.  In particular, the theory does not cover the following two important cases.  The current results do not immediately give asymptotic results in which the actual functional of interest $a$ is itself random and dependent on the data, taking the form $a = a_{\mathscr D_n}$.  This covers the case in which $a$ is the average derivative of $g$ against the empirical distribution of $x$, which may also be a quantity of interest.   The second case is the case in which the functional $a$ itself must be estimated.  This includes the case in which $a(g) = \Ep [ g'(x)]$ and the population distribution of $x$ is unknown and hence the operator $\Ep$ must be ``estimated.''  The reference \cite{newey:stoker:1993} discusses estimating population averages of $g'(x)$ using sample averages of $\hat g'(x_i)$ in low dimensional settings.  Under their regularity conditions, for certain weighting functions $w(x)$, $\Ep[w(x) \hat g(x)]$ and $\frac{1}{n} \sum_{i=1}^n w(x_i) \hat g '(x_i)$ will have different sampling distributions.   However, a key condition in that reference is that $w(x) \times  \mathsf{density}(x)$ vanishes on the boundary of the support of $x$.  This is to allow cancelation of certain boundary terms in the integration.  Otherwise, an average derivative is not $\sqrt{n}$ estimable.  In the ACT example, assuming the density vanishes on the boundary may not be the best approximation (the fraction of students who achieve a score of 36 is small but not vanishingly small compared to the 20,000 students in the sample).   Relatedly, in this example, interest in the average derivative against certain fixed (or independently estimated from a different data source) distribution would be of potential interest to an admissions body.  One example would be the distribution of ACT scores from a previous year's class.  Alternatively, the nation-wide ACT score distribution may be of interest.   In addition, it may also be of interest to calculate the average derivative against the uniform distribution.  This last possibility would be immediately applicable from the theory in the paper.

\subsection{Discussion}

The results in Table 2 add further support to the results found in \cite{BettingerEvansPope}.   The Math and English components of the ACT contain more predictive power than the Scient and Reading components on college GPA.  Expanding on the results in \cite{BettingerEvansPope}, the Post-Nonparametric Double Selection analysis highlights that this general trend holds at all regions of the ACT score distribution (as measured by the average derivatives reported).  The analysis finds that the association between ACT scores and GPA is stronger in the extremes of the ACT score distribution.


\begin{table}[h!]
{\footnotesize Table 1.  Baseline Regression Models
\begin{center}
\begin{tabular}{lcccccc}
\hline 
\hline \\
 GPA & (1)  & (2) & (3) & (4) & (5)   \\
\cline{2-6} \\
ACT Composite       & .0719   \\  
			       & (.0012)  \\ \\
ACT Math                &    & .0335 & .0204 & .0291 & -.0166  \\
       			       &   & (.0016)  & (.0017) & (.0107) & ( .0113)\\ \\
ACT English             &   & .0374 & .0182  & .0388 & .0205 \\
        			        &   & (.0017)  & (.0017) & (.0086) &(.0091)  \\ \\
ACT Reading            &   & .0046 & .0029 &  -.0057 &  -.0174 \\
                                 &   & (.0014)  & (.0013) &  (.0073) &(.0075)\\ \\
ACT Science             &   & -.0041 & -.0023 & .0405 & .0160  \\
                                    &     & (.0020) & (.0020) & ( .0107) & (.0112)  \\ \\
(ACT Math)$^2$       & & &     &  .0001 & .0008  \\
                                  & & &   & ( .0000)  & ( .0002) \\ \\
(ACT English)$^2$   & &    &   &  -.0000 &  -.0001 \\
                                       & &   & &  (.0019)  & (.0002)  &  \\ \\
(ACT Reading)$^2$    & &  &  &  .0002& .0004 \\
                                   &  & & & (.0002)  & (.0002)  &   \\ \\
 (ACT Science)$^2$    & &  & &  -.0010 &.0160 \\
                                    & & &  & (.0002)  & (.0112) \\ \\
$^\sharp$\hspace{.5mm}Linear Controls& & & X & & X  \\
R$^2$ & .1279 & .1407 & .2157 & .1414 & .2165 \\
$n$ &25243 &  25243 & 21757 & 25243 &  21757     \\
\hline
\end{tabular}
\end{center}}
\scriptsize
 \flushleft This table displays estimates from 5 parametric specifications for individual level regressions of College GPA on ACT composite score, ACT subject-area scores, and controls.  Standard errors in parenthesis.

 $^\sharp$ The control variables used here contain indicators for race, gender, high school GPA, college campus, and college major.  There are 39 distinct indicators which enter all above specifications linearly.  The indicators for race, gender, and high school GPA are constructed identically to \cite{BettingerEvansPope}.  The construction of the major indicators differs slightly by grouping majors into 8 categories instead of 45 categories.  
Note: Specifications (1) and (2) replicate results in Table 2 of \cite{BettingerEvansPope} exactly.  Specification (3)  approximately 
replicates the results of \cite{BettingerEvansPope} with small discrepancy due to differing grouping of college majors. Specifications (4) and (5) are not reported in that reference. \ \ \ \ \ \ \ \ \ \ \ \ \ \ \ \ \ \ \ \ \ \ \ \

\

\

\

\

\

\

\

\

\

\

\end{table}

\begin{table}[h!]
{\footnotesize Table 2. Post-Nonparametric Double Selection Analysis \\ Average Derivative Estimates
\begin{center}
\begin{tabular}{lccccc}
\hline 
\hline \\
  & (1)  & (2) & (3) & (4) & (5)   \\
 GPA   & Overall  & Percentiles  & Percentiles   & Percentiles   & Percentiles   \\
           &		   &  0 - 25 & 25-50 & 50-75 & 75 - 100\\
\cline{2-6} \\
ACT Math                &  .0256    & .0276 & .0173 &  .0203  &  .0374  \\
       			       &  (.0021)  & (.0047)  & (.0025)  & (.0020) & (.0031)\\ \\
ACT English              &  .0159    & .0156 & .0137 &  .0156  &  .0188  \\
        			         &  (.0018)  & (.0032)  & (.0028)  & (.0023) & (.0032)\\ \\
ACT Reading             &  .0027    & -.0023 & .0013 &  .0044  &  .0074  \\
                                  &  (.0015)  & (.0031)  & (.0021)  & (.0019) & (.0027)\\ \\
ACT Science              &  -.0015    & .0046 & -.0019 &  -.0021  &  -.0067  \\
                                     &  (.0021)  & (.0038)  & (.0031)  & (.0030) & (.0032)\\ \\
                                    \\
$n:$ 21757      \ \ 
$K$: 24 \ \
$L$: 21799 \\
Selected: 614 \ \ R$^2$: 0.2428 \\
\hline
\end{tabular}
\end{center}}
\scriptsize
\flushleft
Post-Nonparametric Double Selection estimates for average derivatives of GPA with respect to ACT component scores and a profile of conditioning variables generated from information on race, gender, campus, high school GPA and college major.  An overall average derivative is estimated for each ACT component in panel (1). Quartile (percentile range) memberships are defined according to ACT composite scores.  Average derivatives are computed within each quartile in panels (2) - (5).  
\end{table}

\section{Conclusion}

This paper considers the problem of selecting a conditioning set in the context of nonparametric regression.    Convergence rates and inference results are provided for series estimators of a primary component of interest in additively separable models with high-dimensional conditioning information.  The finite sample performance of several Post-Nonparametric Double Selection estimators are evaluated in a simulation study.  Overall, the proposed Span option has good estimation and inferential properties in the data generating processes considered.

\vspace{-1mm}

\pagebreak 
\singlespace

\bibliographystyle{plain} 
{
\bibliography{dkbib1}

\begin{thebibliography}{10}

\bibitem{andrews:hang:additiveinteraction}
Donald~D.W. Andrews and Yoon-Jae Whang.
\newblock Additive interactive regression models: Circumvention of the curse of
  dimensionality.
\newblock {\em Econometric Theory}, 6(4):466–479, 12 1990.

\bibitem{andrews1991}
Donald W.~K. Andrews.
\newblock Asymptotic normality of series estimators for nonparametric and
  semiparametric regression models.
\newblock {\em Econometrica}, 59(2):307--345, 1991.

\bibitem{BaiNg2008}
J.~Bai and S.~Ng.
\newblock Forecasting economic time series using targeted predictors.
\newblock {\em Journal of Econometrics}, 146:304--317, 2008.

\bibitem{BaiNg2009b}
J.~Bai and S.~Ng.
\newblock Boosting diffusion indices.
\newblock {\em Journal of Applied Econometrics}, 24, 2009.

\bibitem{BellChenChernHans:nonGauss}
A.~Belloni, D.~Chen, V.~Chernozhukov, and C.~Hansen.
\newblock Sparse models and methods for optimal instruments with an application
  to eminent domain.
\newblock {\em Econometrica}, 80:2369--2429, 2012.
\newblock Arxiv, 2010.

\bibitem{BC-PostLASSO}
A.~Belloni and V.~Chernozhukov.
\newblock Least squares after model selection in high-dimensional sparse
  models.
\newblock {\em Bernoulli}, 19(2):521--547, 2013.
\newblock ArXiv, 2009.

\bibitem{BCFH:Policy}
A.~Belloni, V.~Chernozhukov, I.~Fernández-Val, and C.~Hansen.
\newblock Program evaluation and causal inference with high-dimensional data.
\newblock {\em Econometrica}, 85(1):233--298, 2017.

\bibitem{BellChernHans:Gauss}
A.~Belloni, V.~Chernozhukov, and C.~Hansen.
\newblock Lasso methods for gaussian instrumental variables models.
\newblock 2010 arXiv:[math.ST], http://arxiv.org/abs/1012.1297, 2010.

\bibitem{BCH2011:InferenceGauss}
A.~Belloni, V.~Chernozhukov, and C.~Hansen.
\newblock Inference for high-dimensional sparse econometric models.
\newblock {\em Advances in Economics and Econometrics. 10th World Congress of
  Econometric Society. August 2010}, III:245--295, 2013.

\bibitem{victor:newseries}
Alexandre Belloni, Victor Chernozhukov, Denis Chetverikov, and Kengo Kato.
\newblock Some new asymptotic theory for least squares series: Pointwise and
  uniform results.
\newblock {\em Journal of Econometrics}, 186(2):345 -- 366, 2015.
\newblock High Dimensional Problems in Econometrics.

\bibitem{bell:chern:chet:kato:series}
Alexandre Belloni, Victor Chernozhukov, Denis Chetverikov, and Kengo Kato.
\newblock Some new asymptotic theory for least squares series: Pointwise and
  uniform results.
\newblock {\em Journal of Econometrics}, 186(2):345 -- 366, 2015.
\newblock High Dimensional Problems in Econometrics.

\bibitem{BCH-PLM}
Alexandre Belloni, Victor Chernozhukov, and Christian Hansen.
\newblock Inference on treatment effects after selection amongst
  high-dimensional controls with an application to abortion on crime.
\newblock {\em Review of Economic Studies}, 81(2):608--650, 2014.

\bibitem{BCHK:Panel}
Alexandre Belloni, Victor Chernozhukov, Christian Hansen, and Damian Kozbur.
\newblock Inference in high-dimensional panel models with an application to gun
  control.
\newblock {\em Journal of Business \& Economic Statistics}, 34(4):590--605,
  2016.

\bibitem{BettingerEvansPope}
Eric~P. Bettinger, Brent~J. Evans, and Devin~G. Pope.
\newblock Improving college performance and retention the easy way: Unpacking
  the act exam.
\newblock {\em American Economic Journal: Economic Policy}, 5(2):26--52, 2013.

\bibitem{BickelRitovTsybakov2009}
P.~J. Bickel, Y.~Ritov, and A.~B. Tsybakov.
\newblock Simultaneous analysis of {L}asso and {D}antzig selector.
\newblock {\em Annals of Statistics}, 37(4):1705--1732, 2009.

\bibitem{bowen:chingos:mcpherson}
William~G. Bowen, Matthew~M. Chingos, and Michael~S. McPherson.
\newblock {\em Crossing the Finish Line - Completing College at America's
  Public Universities}.
\newblock Princeton University Press, 2011.

\bibitem{BuhlmannGeer2011}
P.~B\"{u}hlmann and S.~van~de Geer.
\newblock {\em Statistics for High-Dimensional Data: Methods, Theory and
  Applications}.
\newblock Springer, 2011.

\bibitem{buja1989}
Andreas Buja, Trevor Hastie, and Robert Tibshirani.
\newblock Linear smoothers and additive models.
\newblock {\em Ann. Statist.}, 17(2):453--510, 06 1989.

\bibitem{BuneaTsybakovWegkamp2007b}
F.~Bunea, A.~Tsybakov, and M.~H. Wegkamp.
\newblock Sparsity oracle inequalities for the lasso.
\newblock {\em Electronic Journal of Statistics}, 1:169--�194, 2007.

\bibitem{BuneaTsybakovWegkamp2006}
F.~Bunea, A.~B. Tsybakov, , and M.~H. Wegkamp.
\newblock Aggregation and sparsity via $\ell_1$ penalized least squares.
\newblock In {\em Proceedings of 19th Annual Conference on Learning Theory
  (COLT 2006) (G. Lugosi and H. U. Simon, eds.)}, pages 379--�391, 2006.

\bibitem{BuneaTsybakovWegkamp2007}
F.~Bunea, A.~B. Tsybakov, and M.~H. Wegkamp.
\newblock Aggregation for {G}aussian regression.
\newblock {\em The Annals of Statistics}, 35(4):1674--1697, 2007.

\bibitem{CandesTao2007}
E.~Cand\`{e}s and T.~Tao.
\newblock The {D}antzig selector: statistical estimation when $p$ is much
  larger than $n$.
\newblock {\em Ann. Statist.}, 35(6):2313--2351, 2007.

\bibitem{ChenHardleLintonSeveranceLossin}
R.~Chen, W.~H{\"a}rdle, O.~B. Linton, and E.~Severance-Lossin.
\newblock Nonparametric estimation of additive separable regression models.
\newblock In Wolfgang H{\"a}rdle and Michael~G. Schimek, editors, {\em
  Statistical Theory and Computational Aspects of Smoothing}, pages 247--265,
  Heidelberg, 1996. Physica-Verlag HD.

\bibitem{chen:christiansen:series}
Xiaohong Chen and Timothy~M. Christensen.
\newblock Optimal uniform convergence rates and asymptotic normality for series
  estimators under weak dependence and weak conditions.
\newblock {\em Journal of Econometrics}, 188(2):447 -- 465, 2015.
\newblock Heterogeneity in Panel Data and in Nonparametric Analysis in honor of
  Professor Cheng Hsiao.

\bibitem{LocalPartitionedReg}
Norbert Christopeit and Stefan G.~N. Hoderlein.
\newblock Local partitioned regression.
\newblock {\em Econometrica}, 74(3):787--817, 2006.

\bibitem{cox1988}
Dennis~D. Cox.
\newblock Approximation of least squares regression on nested subspaces.
\newblock {\em Ann. Statist.}, 16(2):713--732, 06 1988.

\bibitem{eastwoodgallant}
Brian~J. Eastwood and A.~Ronald Gallant.
\newblock Adaptive rules for seminonparametric estimators that achieve
  asymptotic normality.
\newblock {\em Econometric Theory}, 7(3):307--340, 1991.

\bibitem{FF:1993}
Ildiko~E. Frank and Jerome~H. Friedman.
\newblock A statistical view of some chemometrics regression tools.
\newblock {\em Technometrics}, 35(2):109--135, 1993.

\bibitem{TU}
C.~{Hansen}, D.~{Kozbur}, and S.~{Misra}.
\newblock {Targeted Undersmoothing}.
\newblock {\em ArXiv e-prints}, June 2017.

\bibitem{hastie1986}
Trevor Hastie and Robert Tibshirani.
\newblock [generalized additive models]: Rejoinder.
\newblock {\em Statist. Sci.}, 1(3):314--318, 08 1986.

\bibitem{elements:book}
Trevor Hastie, Robert Tibshirani, and Jerome Friedman.
\newblock {\em Elements of Statistical Learning: Data Mining, Inference, and
  Prediction}.
\newblock Springer, New York, NY, 2009.

\bibitem{huang2010}
Jian Huang, Joel~L. Horowitz, and Fengrong Wei.
\newblock Variable selection in nonparametric additive models.
\newblock {\em Ann. Statist.}, 38(4):2282--2313, 08 2010.

\bibitem{horowitz:lasso}
Jian Huang, Joel~L. Horowitz, and Fengrong Wei.
\newblock Variable selection in nonparametric additive models.
\newblock {\em Ann. Statist.}, 38(4):2282--2313, 2010.

\bibitem{imbens:prop:matching:continuous}
Guido Imbens and Keisuke Hirano.
\newblock The propensity score with continuous treatments.
\newblock 2004.

\bibitem{JM:ConfidenceIntervals}
Adel Javanmard and Andrea Montanari.
\newblock Confidence intervals and hypothesis testing for high-dimensional
  regression.
\newblock {\em Journal of Machine Learning Research}, 15:2869--2909, 2014.

\bibitem{knight:shrinkage}
Keith Knight.
\newblock Shrinkage estimation for nearly singular designs.
\newblock {\em Econometric Theory}, 24:323--337, 2008.

\bibitem{Koltchinskii2009}
V.~Koltchinskii.
\newblock Sparsity in penalized empirical risk minimization.
\newblock {\em Ann. Inst. H. Poincar� Probab. Statist.}, 45(1):7--57, 2009.

\bibitem{leeb:potscher:pms}
Hannes Leeb and Benedikt~M. P{\"o}tscher.
\newblock Can one estimate the unconditional distribution of
  post-model-selection estimators?
\newblock {\em Econometric Theory}, 24(2):338--376, 2008.

\bibitem{li:racine:book}
Qi~Li and Jeffrey~Scott Racine.
\newblock {\em Nonparametric Econometrics: Theory and Practice}.
\newblock Princeton University Press: Princeton, NJ, 2006.

\bibitem{Lounici2008}
K.~Lounici.
\newblock Sup-norm convergence rate and sign concentration property of lasso
  and dantzig estimators.
\newblock {\em Electron. J. Statist.}, 2:90--102, 2008.

\bibitem{LouniciPontilTsybakovvandeGeer2009}
K.~Lounici, M.~Pontil, A.~B. Tsybakov, and S.~van~de Geer.
\newblock Taking advantage of sparsity in multi-task learning.
\newblock {\em arXiv:0903.1468v1 [stat.ML]}, 2010.

\bibitem{MY2007}
N.~Meinshausen and B.~Yu.
\newblock Lasso-type recovery of sparse representations for high-dimensional
  data.
\newblock {\em Annals of Statistics}, 37(1):2246--2270, 2009.

\bibitem{newey:series}
Whitney~K. Newey.
\newblock Convergence rates and asymptotic normality for series estimators.
\newblock {\em Journal of Econometrics}, 79:147--168, 1997.

\bibitem{newey:stoker:1993}
Whitney~K. Newey and Thomas~M. Stoker.
\newblock Efficiency of weighted average derivative estimators and index
  models.
\newblock {\em Econometrica}, 61(5):1199--1223, 1993.

\bibitem{potscher}
Benedikt~M. P{\"o}tscher.
\newblock Confidence sets based on sparse estimators are necessarily large.
\newblock {\em Sankhy\=a}, 71(1, Ser. A):1--18, 2009.

\bibitem{radford:etal:postsecondary}
Alexandria~Walton Radford, Lutz Berkner, Sara Wheeless, and Bryan Shepherd.
\newblock Persistence and attainment of 2003-04 beginning postsecondary
  students: After 6 years. first look. nces 2011-151.
\newblock 2010.

\bibitem{RosenbaumTsybakov2008}
Mathieu Rosenbaum and Alexandre~B. Tsybakov.
\newblock Sparse recovery under matrix uncertainty.
\newblock {\em The Annals of Statistics}, 38(5):2620--2651, 2010.

\bibitem{additive:derivatives}
Eric Severance-Lossin and Stefan Sperlich.
\newblock Estimation of derivatives for additive separable models.
\newblock {\em Statistics}, 33(3):241--265, 1999.

\bibitem{stone1985}
Charles~J. Stone.
\newblock Additive regression and other nonparametric models.
\newblock {\em The Annals of Statistics}, 13(2):689--705, 1985.

\bibitem{T1996}
R.~Tibshirani.
\newblock Regression shrinkage and selection via the lasso.
\newblock {\em J. Roy. Statist. Soc. Ser. B}, 58:267--288, 1996.

\bibitem{uchi_drops_act}
uchicago news.
\newblock Uchicago launches test-optional admissions process with expanded
  financial aid, scholarships.
\newblock 2018.

\bibitem{vdGeer}
S.~A. van~de Geer.
\newblock High-dimensional generalized linear models and the lasso.
\newblock {\em Annals of Statistics}, 36(2):614--645, 2008.

\bibitem{vdGBRD:AsymptoticConfidenceSets}
Sara van~de Geer, Peter Bühlmann, Ya’acov Ritov, and Ruben Dezeure.
\newblock On asymptotically optimal confidence regions and tests for
  high-dimensional models.
\newblock {\em Ann. Statist.}, 42(3):1166--1202, 06 2014.

\bibitem{Wainright2006}
M.~Wainwright.
\newblock Sharp thresholds for noisy and high-dimensional recovery of sparsity
  using $\ell_1$-constrained quadratic programming (lasso).
\newblock {\em IEEE Transactions on Information Theory}, 55:2183--2202, May
  2009.

\bibitem{YANG2003521}
Lijian Yang, Stefan Sperlich, and Wolfgang Härdle.
\newblock Derivative estimation and testing in generalized additive models.
\newblock {\em Journal of Statistical Planning and Inference}, 115(2):521 --
  542, 2003.

\bibitem{ZhangHuang2006}
C.-H. Zhang and J.~Huang.
\newblock The sparsity and bias of the lasso selection in high-dimensional
  linear regression.
\newblock {\em Ann. Statist.}, 36(4):1567--1594, 2008.

\bibitem{ZhangZhang:CI}
Cun-Hui Zhang and Stephanie~S. Zhang.
\newblock Confidence intervals for low dimensional parameters in high
  dimensional linear models.
\newblock {\em Journal of the Royal Statistical Society: Series B (Statistical
  Methodology)}, 76(1):217--242, 2014.

\bibitem{zhao:yu:2006}
Peng Zhao and Bin Yu.
\newblock On model selection consistency of lasso.
\newblock {\em J. Mach. Learn. Res.}, 7:2541--2563, December 2006.

\end{thebibliography}
}

\pagebreak



\appendix

\noindent \textbf{SUPPLEMENT TO ``Inference in Additively Separable Models with a High-Dimensional Set of Conditioning Variables''}

\

\onehalfspacing

\section{Implementation Details}

\subsection{Lasso Implementation Details}

\

\subsubsection{Lasso implementation given penalty $\lambda$.}
In every case, penalty loadings $\ell_j$ are chosen as described in \cite{BellChenChernHans:nonGauss} with one small modification.  The procedure suggested in \cite{BellChenChernHans:nonGauss} requires an initial penalty loadings which are constructed using initial estimates of regression residuals.  Their suggestion is to use $\hat \varepsilon_i^{\text{initial}} = y_i$ followed by an iterative procedure.  Here, instead, $\hat \varepsilon_i^{\text{initial}}$ are taken as the linear regression residuals after regressing the outcome $v$ on the 5 most marginally correlated $q_{jL}$, ie, the 5 which have the highest $| \hat {\text{corr}}(v, q_{jL}(z)|$.   Such modification was also used in \cite{TU}.    

\subsubsection{Penalty level choice for single outcome.}  In every case when a single outcome variable is considered in isolation (this includes the reduced form selection step and the selection step corresponding to $\PhiK_1$), Lasso is implemented with penalty $\lambda$ as described in \cite{BellChenChernHans:nonGauss}.  For ease of reference, note that \cite{BellChenChernHans:nonGauss} suggest $\lambda$ given by $2c_{\text{Lasso}} F^{-1}_{\N(0,1)}(1-\alpha_{\text{Lasso}} /L)$ where $c_{\text{Lasso}}>1, \alpha_{\text{Lasso}} \rightarrow 0$ are tuning parameters.  In every instance in this paper, $c_{\text{Lasso}}=1.01$ and $\alpha_{\text{Lasso}} = .05$ are used.

\subsubsection{Penalty level choice for ${\Phi_{K,\text{\em Simple}}}$.} 

In this case, $K$ Lasso regressions are run simultaneously.   In this case, for all $\varphi \in \Phi_K$, $\lambda$ is given by $2c_{\text{Lasso}} F^{-1}_{\N(0,1)}(1-\alpha_{\text{Lasso}} /L)$ where $c_{\text{Lasso}}=1.01$ and $\alpha_{\text{Lasso}} = .05/K$ are used.

\subsubsection{Penalty level choice and implementation for ${\Phi_{K,\text{\em Span}}}$.} 
When the Span option is used, ${\Phi_{K,\text{\em Span}}}$ is decomposed 
${\Phi_{K,\text{ Span}}} = \PhiK_1 \cup \PhiK_2 \cup \PhiK_3$.  Each component has a corresponding penalty level applied to all $\varphi$ within that component.  On the first component, $\lambda_{ \PhiK_1 } = 2c_{\text{Lasso}} F^{-1}_{\N(0,1)}(1-\alpha_{\text{Lasso}} /L)$ where $c_{\text{Lasso}}=1.01$ and $\alpha_{\text{Lasso}} = .05$.  On the second component, $\lambda_{ \PhiK_2 } = 2c_{\text{Lasso}} F^{-1}_{\N(0,1)}(1-\alpha_{\text{Lasso}} /L)$ where $c_{\text{Lasso}}=1.01$ and $\alpha_{\text{Lasso}} = .05/K$.  On the third component, $\lambda_{ \PhiK_3 } = 2c_{\text{Lasso}} F^{-1}_{\N(0,1)}(1-\alpha_{\text{Lasso}} /L)$ where $c_{\text{Lasso}}=1.01$ and $\alpha_{\text{Lasso}} = .05/K$.

The following procedure is used for approximating $I_{\PhiK}$ in the case that a component of $\Phi_K$ contains a continuum of test functions.  For each $j \leq L$, a Lasso regression $\check \varphi_j \in \Phi_{K3}$ which is more likely to select $q_{jL(z)}$ than other $\varphi \in \Phi_K$.  Specifically, for each $j$, $\check \varphi_j$ is set to the linear combination of $p_{1K},...,p_{KK}$ with highest marginal correlation to $q_{jL}$.  Then the approximation to the first stage model selection step proceeds by using $\check I_{\PhiK_3} =  \bigcup_{j\leq L} I_{ \check \varphi_{j}(x)}$ in place of $I_{\PhiK_3}$.

\subsubsection{Penalty level choice for ${\Phi_{K,\text{\em Span-Conservative}}}$.} 

When the Conservative Span option is used, ${\Phi_{K,\text{ Span-Conservative}}}$ is decomposed 
${\Phi_{K,\text{ Span-Conservative}}} = \PhiK_1 \cup \PhiK_2 \cup \PhiK_3$.  Each component again has a corresponding penalty level applied to all $\varphi$ within that component.  On the first component, $\lambda_{ \PhiK_1 } = 2c_{\text{Lasso}} F^{-1}_{\N(0,1)}(1-\alpha_{\text{Lasso}} /L)$ where $c_{\text{Lasso}}=1.01$ and $\alpha_{\text{Lasso}} = .05$.  On the second component, $\lambda_{ \PhiK_2 } = 2c_{\text{Lasso}} F^{-1}_{\N(0,1)}(1-\alpha_{\text{Lasso}} /L)$ where $c_{\text{Lasso}}=1.01$ and $\alpha_{\text{Lasso}} = .05/K$.  On the third component, $\lambda_{ \PhiK_3 } = 2c_{\text{Lasso}} F^{-1}_{\N(0,1)}(1-\alpha_{\text{Lasso}} /L)$ where $c_{\text{Lasso}}=1.01K^{1/2}$ and $\alpha_{\text{Lasso}} = .05$.  

In order to approximate the variables selected on the continuum of Lasso estimates indexed by $\PhiK_3$, the identical procedure with the Span option above is used.  Note that the only difference between the Conservative Span option and the Span option is in $\lambda_{\PhiK_3}$.

\

\subsection{$p^K$ Implementation Details}

In every simulation, $p^K$ is constructed using a cubic B-spline expansion. For fixed $K$, the approximating dictionary is chosen according to the following procedure.   
Knots points $t_1,...,t_{K-3}$ are chosen according to the following rule.    
Set $$t_{\max} = \text{quantile}_{0.95}(|x_1|,...,|x_n|) \ \text{and } t_{\min}= - t_{\max}.$$    Let $\Delta_k = t_k - t_{k-1}$.  For constants $c_1,c_2 \geq 0$ set $$ \Delta_k = c_1 + c_2| (K-2)/2 - k| $$ for $k=2,...,K-3$.

The constants $c_1,c_2$ serve to insert more knot points where the density of $x$ is higher.   The choices for $c_1,c_2$ are determined uniquely by the condition that  $c_1 = 2c_2$ and that the endpoints satisfy $t_1 = t_{\min}$ and $t_{K-3} = t_{\max}$.
Next, the B-spline formulation used here is given by the recursive formulation.   Set $$B_{k,0}(x) = \textbf{1}_{t_k \leq x < t_{k+1}}.$$
Set $B_{k,0} = 0$ for $k$ outside of $1,...,K-3$.  In addition, for spline order $o>0$,
$$ B_{k,o}(x) = \frac{x - t_k}{t_{k + o} - t_k}B_{k,o-1} + \frac{ t_{k+o+1} - x}{t_{k + o+1} - t_{k+1}}B_{k+1,o-1} .$$
 Set $(p_{1,K}(x),...,p_{K-3,K}(x)) = (B_{1,3}(x) , ... B_{K-3,3}(x))$.  The dictionary is completed by adding the additional terms $
p_{K-2, K}(x) = x, \  p_{ K-1, K}(x) = x^2, \  p_{ K, K}(x) = x^3 $.

$\hat K$ is chosen according to the following procedure.  First, an initial set of terms $ q^{\text{initial}}(z) \subseteq q^L(z)$ is selected.  In each case, $q^{\text{initial}}(z)$ contains the terms $I_{RF}$.  That is, the terms selected in a Lasso regression $y$ on $q^L(z)$.  Next, an initial value $\hat K_0 \leq 2\lfloor n^{1/3} \rfloor  $ is chosen to minimize BIC using $(p^K(x),  q^{\text{initial}}(z))$.  In the simulation study, $\hat K_0$ is constrained to be $\geq 5$.  Finally, in order to ensure undersmoothing, $\hat K$ is set to $\hat K = \lfloor ( \log_{10}(n)) \hat K_0 \rfloor $ in the simulation studies and $\hat K = \hat K_0+1$ in the empirical example.  

In the empirical example, separate components of $g_0$, given by
$g_0^{\text{math}}(\text{ACT}_{\text{math}}),   g_0^{\text{eng}}(\text{ACT}_{\text{eng}}), g_0^{\text{read}}(\text{ACT}_{\text{read}}  g_0^{\text{sci}}(\text{ACT}_{\text{sci}}))$ are each approximated with separate dictionaries $p_{\text{math}}^{K_{\text{math}}}$,  $p_{\text{eng}}^{K_{\text{eng}}}$,  $p_{\text{read}}^{K_{\text{read}}}$,  $p_{\text{sci}}^{K_{\text{sci}}}$.  The restriction $ K_{\text{math}} =  K_{\text{eng}} =  K_{\text{read}} =  K_{\text{sci}}$ is enforced.  Set $$ K =  K_{\text{math}} +  K_{\text{eng}} +  K_{\text{read}} + K_{\text{sci}} .$$Each of the above four dictionaries is a B-spline basis defined exactly as in the simulations.  $\hat K$ is chosen to minimize BIC in the same way as in simulations.  Then $\hat K$ is set to $$\hat K = 4 + \hat K_0.$$
The choice $\hat K = 4 + \hat K_0$ in the empirical example, instead of $\hat K = \lfloor ( \log_{10}(n)) \rfloor$, is made to avoid $\hat K > 4 \times 36 $ which corresponds to the size of the support of the data.  

\subsection{Targeted Undersmoothing Implementation Details}

The following procedure is used to estimate the Targeted Undersmoothing (TU; specifically TU(1); see \cite{TU}) confidence intervals for $\theta_0$.  For each $I \subseteq \{1,...,p\}$ let $\hat {\text{CI}}_{K,I}(\theta_0)$ be the corresponding confidence interval for $\theta_0$ using $K$ terms and the components of $q^L$ corresponding to $I$. 
Then the full TU confidence interval is defined by the convex hull of $\cup_{j \leq p} \hat {\text{CI}}_{\hat K, I_{\RF} \cup \{j\} }(\theta_0)$.  In this implementation, a truncated TU confidence interval is calculated instead: $\cup_{j \leq 2s_0} \hat {\text{CI}}_{\hat K, I_{\RF} \cup \{j\} }(\theta_0)$.  This is done because the simulation run time reduces to the order of a day (from the order of a month), and therefore helps facilitate easier replicability.  Changing the code to calculate the full TU confidence intervals is trivial.  This also highlights that computing speed is another advantage of the Post-Nonparametric Double procedure relative to TU in certain settings.  In terms of approximation error, the full TU estimator was implemented for the case $n=100$, $p=50$ for 1000 replications.  The full TU confidence intervals as well as the truncated TU confidence intervals each made 9 false rejections.  In addition, the average interval length for the full TU intervals was 1.740 while the average interval length for the truncated TU intervals was 1.722.  Therefore, the truncated and full TU confidence intervals show very similar performance in this instance.

\section{Proofs}

\subsection{Preliminary Setup and Additional Notation}

Throughout the course of the proof, as much reference as possible is made to results in \cite{newey:series},\cite{BCH-PLM}.  This is done in order to maximize clarity and to present a better picture of the overall argument.  In many cases, appealing directly to arguments in \cite{newey:series} is possible because many of the bounds required for deriving asymptotic normality for series estimators depend only on properties of $\hat g$, $g_0$, $p^K$ and $D$.  Less direct appeal to bounds in the original Post-Double Selection argument is possible, since those arguments do not track $K$, and do not have notions of quantities stemming from $\PhiK$ like $\alpha_\rho, \alpha_{\Phi}$.  However, the main idea of decomposing $p^K$ into components in the span of, and orthogonal to $q^L$, remains as a theme throughout the proofs. 

 For any function $\varphi$, let $\varphi(X)$ denote the vector $[\varphi(x_1),\varphi(x_2),...,\varphi(x_n)]'$.  Similarly, let $\phi_{q^L}\varphi(Z) = [\pi_{q^L} \varphi(z_1),\pi_{q^L} \varphi(z_2),...,\pi_{q^L} \varphi(z_n)]'$.  In addition, define the following quantities.

\begin{itemize}
\item[1.] Let $m$ be the $n \times K$ matrix $m = \pi_{q^L}p^K(Z) = [\pi_{q^L}p_{1K}(Z),...,\pi_{q^L}p_{KK}(Z)]$
\item[2.] Let $W = P -m$ 
\item[3.] Let $\hat \Omega = n^{-1}P' \mathscr M P$  
\item[4.] Let $\Omega = n^{-1}\Ep[W'W]$  
\item[5.] Let $\bar \Omega = n^{-1} W'W$
\item[6.] Let $m$ be partitioned $m = [m_1,...,m_K]$
\item[7.] Let $W$ be partitioned $m = [W_1,...,W_K]$
\item[8.] For any $\varphi \in \PhiK$, let $R_\varphi = Q ( \beta_{\varphi,L} - \beta_{\varphi,L,s_0} )$  
\item[9.] Let $R_y=Q ( \beta_{y,L} - \beta_{y,L,s_0} )$
\item[10.] For any $\varphi$, let $U_\varphi = \varphi(X) - Q \beta_{\varphi,L}$ 
\item[11.] Let $U_y = Y - Q \beta_{y,L}$ 
\item[12.] Let $F = V^{-1/2}$
\item[13.] Let $\varphi_a(x)$ be the function such that $\pi_{q^L}\varphi_a(Z) =  FA'm$  
\item[14.] Let $m_a = FA'm$ 
 \item[15.] Let $W_a = \varphi_a(X) - m_a$.
 \item[16.] For $g \notin \PhiK$, let $R_g = \pi_{q^Lg}(Z) - \eta_1Q ( \beta_{\varphi_1,L} - \beta_{\varphi_1,L,s_0} ) - ... - \eta_{k_g}Q ( \beta_{\varphi_{k_g},L} - \beta_{\varphi_{k_g},L,s_0} )$ for some $(\varphi_1,...\varphi_{k_g}),(\eta_1,...,\eta_{k_g})$ within a fixed constant factor of achieving the infinum in the density assumption (Assumption 8.)
 \item[17.] Let $R_m = [R_{m_1},...,R_{m_K}]$.
\end{itemize}

Assume without loss of generality that $B_K = \text{Id}_K$, the identity matrix of order $K$.  The reason this is without loss of generality is that dictionary $p^K$ is used only in the post-selection estimation, while $\PhiK$ is used for first stage model selection. In addition, assume without loss of generality that $\Omega = \text{Id}_K$.

Throughout the exposition, there is a common naming convention for various regression coefficients.  Quantities of the form $\hat \beta_{v, I}$ always denotes the sample regression coefficients from regressing the variable $v$ on the components specified by $I$.  This implies that the quantities $\hat \beta_{\varphi, I_{\varphi,L}} = \hat \beta_{\varphi,L,\text{Post-Lasso}}$ are equivalent, since the specified components being regressed on are the same. In addition, $\hat \beta_{\varphi, I_{\PhiK + \RF}} = \hat \beta_{\varphi,\tilde q} = \hat \beta_{\varphi(X), I_{\Phi_K + \RF}}$ are equivalent.  
Next, quantities of the form $\beta_{v,L} $ and $ \beta_{v,L,s_0}$ without a hat accent are population quantities and are defined in the text above.

\subsection{Preliminary Lemmas}

\begin{lemma} Under the assumptions of Theorem 1,

\begin{itemize}
\item[1.] $J_1:=\max_{k \leq K } n^{-1/2}\| Q ' W_k  \|_{\infty}  = O_p(\Jone)$\vspace{1.5mm}
\item[2.] $J_2:=n^{-1/2}  \| Q ' \mathscr E \|_{\infty} = O_p(\Jtwo)$ \vspace{1.5mm}
\item[3.] $J_{3}:=n^{-1/2} \|R_m'\mathscr E\|_2 = O_p(\Jthree)$\vspace{1.5mm}
\item[4.] $J_4:=n^{-1/2} \| R_{h_0}'W\|_{2} = O_p(\Jfour)$\vspace{1.5mm}
\item[5.] $J_5:=\max_{ k \leq K}  n^{-1/2} \| \MI m_k\|_{2} $ $$= O_p(\Jfive)$$
\item[6.] $J_6:=n^{-1/2} \| \MI h_0(Z) \|_2= O_p( \Jsix)$\vspace{1.5mm}
\item[7.]  $J_7:=\max_{k\leq K} \|\hat \beta_{m_k,I_{\PhiK + \text{RF}} } - \beta_{p_{kK}, L, s_0} \|_1 $ $$=O_p( \Jseven)$$
\item[8.] $J_8:=\| \hat \beta_{h_0, I_{\PhiK + \text{RF}}} - \beta_{h_0,L,s_0} \|_1 $ $$= O_p( \Jeight)$$
\item[9.] $J_9:=\max_{k \leq K} \| \hat \beta_{W_k,  I_{\PhiK + \text{RF}}  }  \|_1 = O_p( \Jnine )$\vspace{1.5mm}
\item[10.] $J_{10}:= \| \hat \beta_{\mathscr E,  I_{\PhiK + \text{RF}}  }  \|_1 =  O_p(\Jten)$\vspace{1.5mm}
\item[11.] $J_{11} := n^{-1/2}\| Q ' W_a  \|_{\infty}  = O_p(\Jone)$\vspace{1.5mm}
\item[12.] $J_{12}:=n^{-1/2} \|R_{m_a}'\mathscr E\|_2 = O_p(\Jthree)$\vspace{1.5mm}
\item[13.] $J_{13}:= n^{-1/2} \| \MI m_a\|_{2} = O_p(\Jfive)$
\item[14.] $J_{14}:= \|\hat \beta_{m_a,I_{\PhiK + \text{RF}} } - \beta_{\varphi_a, L, s_0} \|_1 $ $$= O_p( \Jseven)$$
\item[15.] $J_{15} :=\| \hat \beta_{W_a,  I_{\PhiK + \text{RF}}  }  \|_1 = O_p( \Jnine )$\vspace{1.5mm}
\item[16.] $J_{16} :=n^{-1}\| R_m'W \|_\F = O_p( \Jsixteen ).$\vspace{1.5mm}

\end{itemize}
\end{lemma}

\begin{proof}  

\

\

\noindent \textbf{Statement 1.}  By Lemma 5 of \cite{BCH-PLM}, two conditions which together are sufficient for $\max_{k \leq K j\leq L } \frac{| Q_{j} ' W_k |}{  \sqrt{\sum_{i=1}^n q_{jL}(z_i)^2 W_{ki}^2} } = O_p( (\log KL)^{1/2} )$
are that $\max_{k \leq K, j \leq L} \frac{ \Ep[|q_{jL}(z)|^3|W_{ik}|^3]^{1/3} }{\Ep[q_{jL}(z)^2W_{ik}^2]^{1/2}   }=O(\zeta_0(K))$ and the rate condition $\log KL = o(\zeta_0(K)^{-1} n^{1/3})$.   Note that $\Ep[q_{jL}(z)^2W_{ik}^2]^{1/2}  $ is bounded away from zero by assumption. In addition, by H\"older's inequality, $\Ep[|q_{jL}(z)|^3|W_{ik}|^3] \leq \Ep[|q_{jL}(z)|^3 ] \zeta_0(K)^3$. This implies that the first condition holds.  The second condition is given in the assumptions.

\

\noindent \textbf{Statement 2.}   Follows similarly as Statement 1.

\

\noindent \textbf{Statement 3.}  This statement follows directly from the fact that $\Ep[\varepsilon | x,z]=0$, $\Ep[\varepsilon^2 | x,z ]$ bounded, along with $\text{dim}(R_m'\mathscr E) = K$ and $\|R_{m_k}\|_\infty = O(K^{\alpha_{\rho}}L^{-\alpha_{\mathscr Z}})$ by the density assumption, allowing the use of the $K$-dimensional Chebyshev Inequality.

\

\noindent \textbf{Statement 4.}    $\| R_{h_0}'W \|_2 = \| \sum_{i} R_{h_0,i}W_i \|_2 \leq O(L^{-\alpha_{\mathscr Z}}) \zeta_0(K)$ by the facts that $\|R_{h_0} \|_{\infty}= O(L^{-\alpha_{\mathscr Z}})$ and $\| W_i \|_2 \leq \zeta_0(K)$.

\

\noindent \textbf{Statement 5.} 

\

First note that the following two hold.  
\begin{itemize}
\item[1.] For any $\varphi \in \PhiK$, $ \MI \pi_{q^L}\varphi(Z)  =   \MI R_{\varphi}  + \MI (Q\beta_{\varphi,L,s_0 } - Q \hat \beta_{\varphi,I_{\varphi,L}} ).$  
\item[2.] For any $g \in \text{LinSpan}(p^K)$, and any corresponding expansion $g = \eta_1\varphi_1 + ... + \eta_{k_g} \varphi_{k_g} + r_g$ with $\eta_1,...\eta_{k_g} \in \mathbb R, \varphi_1,...\varphi_{k_g} \in \PhiK$, \ $$\|  \MI \pi_{q^L}g(Z) \|_2   \leq  \| \eta \|_1  \max_{\varphi \in \{ \varphi_1,...,\varphi_{k_g} \}} ({ \| Q \beta_{\varphi,L,s_0}} - Q \hat \beta_{\varphi,I_{\varphi,L}} \|_2 + \| R_{\varphi} \|_2 ) +\| r_g(Z) \|_2. $$
\end{itemize}

To show the first of the above two statements, for each $\varphi \in \PhiK$, note that 
\begin{align*}
  \MI \pi_{q^L}\varphi(Z)  & = \MI  \MI \pi_{q^L}\varphi(Z)   \\ & =  \MI(\pi_{q^L}\varphi(Z) -  \PI \pi_{q^L}\varphi(Z) ) \\
  & =   \MI  ( Q\beta_{\varphi,L}-  \PI (\varphi(X) - U_\varphi) ) \\
  &   =  \MI  (    Q\beta_{\varphi,L}  - Q \hat \beta_{\varphi,I_{\PhiK + \RF}} +  \PI U_{\varphi}    )\\
  &   =    \MI  R_{\varphi}  + \MI (Q\beta_{\varphi,L,s_0 } - Q \hat \beta_{\varphi,I_{\PhiK + \RF}} )+ \MI  \PI U_{\varphi}    \\
& = \MI R_{\varphi} + \MI (Q\beta_{\varphi,L,s_0 } - Q \hat \beta_{\varphi,I_{\PhiK + \RF}} ) \\
    &   =   \MI R_{\varphi}  + \MI (Q\beta_{\varphi,L,s_0 } - Q \hat \beta_{\varphi,I_{\varphi,L}} ) + \underbrace{\MI(Q \hat \beta_{\varphi,I_{\varphi,L}} - Q \hat \beta_{\varphi,I_{\PhiK + \RF}}  )  }   \\
    \end{align*}
\begin{align*}
 &\hspace{5.5cm} =\MI( \PI_{I_{\varphi,L}} \varphi(X) - \PI \varphi(X)) \\
 &\hspace{5.5cm}=\MI \PI(\PI_{I_{\varphi,L}} \varphi(X) - \varphi(X) )  \\
 &\hspace{5.5cm}= 0 \\
 \Rightarrow   \MI \pi_{q^L}\varphi(Z)   & =   \MI R_{\varphi}  + \MI (Q\beta_{\varphi,L,s_0 } - Q \hat \beta_{\varphi,I_{\varphi,L}} ). 
 \end{align*}

This establishes the first claim.
Now turn to the second claim.  Note that using the density assumption, there are $\varphi_1,...,\varphi_{k_{g}}$ and a vector $\eta = (\eta_1,...,\eta_{k_g})$ such that 
$g = \eta_1\varphi_1 + ... + \eta_{k_g}\varphi_{k_g} + r_g$
for some remainder $r_g$, sufficiently small.  Then

$$\| \MI  \pi_{q^L} g(Z)  \|_2 = \| \eta_1  \MI  \pi_{q^L} \varphi_1(X) + ... +\eta_{k_g}  \MI  \pi_{q^L} \varphi_{k_g}(Z)  + \MI  \pi_{q^L} r_g(Z) \|_2$$

Next, looking at each $\varphi$ in the above expansion (ie each $\varphi \in \{ \varphi_1,...,\varphi_{k_g} \}$) and combining the above expression gives 
 
 $$\| \MI  \pi_{q^L} g(Z)  \|_2 = \| \eta_1 \MI R_{\varphi_1}  + \eta_1 \MI(Q\beta_{\varphi_1,L,s_0 } - Q \hat \beta_{\varphi_1,I_{\varphi_1}} ) + ... $$ $$... +\eta_{k_g} \MI R_{\varphi_{k_g}}  + \eta_{k_g} \MI( Q\beta_{\varphi_{k_g},L,s_0 } - Q \hat \beta_{\varphi_{k_g},I_{\varphi_{k_g}}} )+  \MI  \pi_{q^L} r_g(Z) \|_2.$$
 
 \
 
Applying H\"older's inequality and the fact that $\MI$ is a projection (and hence non-expansive) gives the bound

\

$$ \leq  \| \eta \|_1  \max_{\varphi \in \{ \varphi_1,...,\varphi_{k_g} \}}( { \| Q \beta_{\varphi,L,s_0}} - Q \hat \beta_{\varphi,I_{\varphi,L}} \|_2 +\| R_\varphi \|_2)  +\| r_g(Z) \|_2.$$

These can then be applied directly to $n^{-1/2} \| \MI m_k \|_2$. The corresponding $\eta$ and $R_{m_k}$ satisfy $L^{-\alpha_{\mathscr Z}} \| \eta \|_1 \leq O(K^{-\alpha_{\rho}})$ and $\| R_{m_k} \|_2 \leq n^{1/2} O(K^{-\alpha_{\rho}} )$.  Then we have the bound

$$  \| \MI  \pi_{q^L} g(Z)  \|_2 =O_p( K^{\alpha_{\rho}} K^{\alpha_{\Phi}/2}s_0^{1/2} \log(L)^{1/2} + n^{1/2} K^{- \alpha_{\rho }} ).$$

\

Under Assumption 10, note that for each $m_k$, taking $\eta = 1$ and $R_{m_{k}} = 0$ are feasible by assumption.  The result follows.  

\

\noindent \textbf{Statement 6.}
\begin{align*}
n^{-1/2} \| \MI h_0(Z) \|_2 & = n^{-1/2}\| \MI (Q\beta_{h_0,L,s_0} + R_{h_0} )   \|_2)\\
& \leq n^{-1/2} (\| \MI Q\beta_{h_0,L,s_0}   \|_2 + \| \MI R_{h_0} \|_2 )\\
& \leq n^{-1/2} (\| \MI Q(\beta_{g_0,L,s_0} +\beta_{h_0,L,s_0} - Q\beta_{g_0,L,s_0}  )  \|_2 + \| \MI R_{h_0} )\|_2\\
& =n^{-1/2} (\| \MI (Q\beta_{y,L,s_0} - Q\beta_{g_0,L,s_0}  )  \|_2 + \| \MI R_{h_0} \|_2)\\
& \leq n^{-1/2} (\| \MI Q\beta_{y,L,s_0} \|_2 + \| \MI Q\beta_{g_0,L,s_0}    \|_2 + \| \MI R_{h_0} \|_2)\\
& = n^{-1/2} (\| \MI \pi_{q^L}y(Z) \|_2 + \| \MI \pi_{q^L}g_0(Z)    \|_2 + \| \MI R_{h_0} \|_2)\\
\end{align*}

The first two terms above, $n^{-1/2} (\| \MI Q\pi_{q^L}y(Z) \|_2 + n^{-1/2} \| \MI \pi_{q^L}g_0(Z)    \|_2 )$, are $O_p( K^{\alpha_{\rho}} K^{\alpha_{\Phi}/2}n^{-1/2}s_0^{1/2} \log(L)^{1/2} + L^{- \alpha_{\mathscr Z}}K^{\alpha_{\rho}} )$ by the same reasoning as Statement 5.  In addition $n^{-1/2} \| \MI R_{h_0} \|_2 \leq n^{-1/2} \| R_{h_0} \|_2= O(L^{-\alpha_{\mathscr Z}}) $ by assumption.  This gives 

$$n^{-1/2} \| \MI h_0(Z) \|_2 = O_p( K^{\alpha_{\rho}} K^{\alpha_{\Phi}/2}n^{-1/2} s_0^{1/2} \log(L)^{1/2} +  L^{- \alpha_{\mathscr Z}}K^{\alpha_{\rho}}).$$

\

\noindent \textbf{Statement 7.}
\begin{align*}
& \| \hat \beta_{m_k, I_{\PhiK + \RF}} - \beta_{p_{kK},L,s_0} \|_1 \\
& \leq | I_{\PhiK + \RF} |^{1/2} \| \hat \beta_{m_k, I_{\PhiK + \RF}} - \beta_{p_{kK},L,s_0} \|_2  \\
& = | I_{\PhiK + \RF} |^{1/2} \( (\hat \beta_{m_k, I_{\PhiK + \RF}} - \beta_{p_{kK},L,s_0}) '(\hat \beta_{m_k, I_{\PhiK + \RF}} - \beta_{p_{kK},L,s_0}) \)^{1/2}  \\
& \leq  | I_{\PhiK + \RF} |^{1/2} O_p(1)\( (\hat \beta_{m_k, I_{\PhiK + \RF}} - \beta_{p_{kK},L,s_0}) ' (Q_{I_{\PhiK + \RF}} ' Q_{I_{\PhiK + \RF}} / n)(\hat \beta_{g, I_{\PhiK + \RF}} - \beta_{p_{kK},L,s_0}) \)^{1/2}  \\
& = | I_{\PhiK + \RF} |^{1/2} O_p(1) n^{-1/2} \| \PI m_k - Q \beta_{p_{kK}, L, s_0} \|_2 \\
& = | I_{\PhiK + \RF} |^{1/2} O_p(1)n^{-1/2} \| m_k-\MI m_k - Q \beta_{p_{kK}, L, s_0} \|_2 \\
& = | I_{\PhiK + \RF} |^{1/2} O_p(1)n^{-1/2} \| -\MI m_k + R_{m_k} \|_2 \\
& \leq | I_{\PhiK + \RF} |^{1/2} O_p(1) ( J_5 + O(L^{-\alpha_{\mathscr Z}})) \\
& = O_p( s_0^{ 1/2} K^{\alpha_{I_\Phi}/2})(J_5 + O(L^{-\alpha_{\mathscr Z}}))\\
& =O_p( s_0^{\alpha_{I_\Phi}/2 + 1/2} ) O_p( K^{\alpha_{\rho}} K^{\alpha_{\Phi}/2}s_0^{1/2} \log(L)^{1/2} + n^{1/2} L^{- \alpha_{\mathscr Z}}K^{\alpha_{\rho}} )\\
& =O_p( K^{\alpha_{\rho}} K^{\alpha_{\Phi}/2}s_0^{1}K^{\alpha_{I_\Phi}/2} \log(L)^{1/2} + n^{1/2} L^{- \alpha_{\mathscr Z}}K^{\alpha_{\rho}} ).\\
\end{align*}

\noindent \textbf{Statement 8.}  Proven analogously to Statement 7.

\

\noindent \textbf{Statement 9.} 
\begin{align*}
&\max_{k \leq K} \| \hat \beta_{W_k, I_{\PhiK + \RF} } \|_1 \\
& = \max_{k \leq K} \| ( \tilde Q' \tilde Q)^{-1} \tilde Q' W_k \|_1 \\
& \leq| |I_{\PhiK + \RF}|^{1/2} \max_{k \leq K} \| ( \tilde Q' \tilde Q)^{-1} \tilde Q' W_k \|_2 \\
& \leq  |I_{\PhiK + \RF}|^{1/2}\kappa_{\min}^{-1/2}(I_{\PhiK + \RF}|) \max_{k \leq K} \| n^{-1} Q'W_k \|_{\infty}\\
 &= O_p( s_0^{1/2} K^{ \alpha_{I_\Phi}/2} \cdot 1 \cdot n^{-1/2} \log(KL)^{1/2} ).
\end{align*}

\noindent \textbf{Statement 10.} \ \ Proven analogously to Statement 9.

\

\noindent \textbf{Statements 11-15.} \ \ Proven analogously to Statements 1,3,5,7,9.

\

\noindent \textbf{Statement 16.} \ \ 

$$ \n^{-1} \Big  \| \sum_{i=1}^n W_i ' R_{m,i} \Big \|_\F = n^{-1} \(\sum_{k} \| W' R_{m_k} \|_2^2 \)^{1/2} $$
$$\leq  n^{-1} \( \sum_{k} n^2 \zeta_0(K)^2 \| R_{m_k}^2\|_\infty  \)^{1/2}. $$
By the density assumption, $\| R_{m_k}\|_\infty \leq K^{\alpha_{\rho}}L^{-\alpha_{\mathscr Z}}$.  This then implies that $$ \n^{-1} \Big  \| \sum_{i=1}^n W_i ' R_{m,i} \Big \|_\F \leq K^{1/2}K^{-\alpha_{\rho}}.$$
\end{proof}

\begin{lemma}  

\

\begin{itemize}
\item[1.] $\Xi_1: =  n^{-1}\| W' \PI W \|_{\mathscr F} \leq  n^{-1/2} K J_9 J_1 $
\item[2.] $\Xi_2 : = n^{-1}\| m' \MI m\|_{\mathscr F} \leq K J_5^2 $
\item[3.] $\Xi_3:=n^{-1}\| m' \MI W \|_{\mathscr F} \leq J_{16} + n^{-1/2} K J_7J_1 $
\item[4.] $ \Xi_4:=n^{-1/2} \| m' \MI h_0(Z) \|_2 \leq n^{1/2}K^{1/2} J_5J_6 $
\item[5.] $ \Xi_5:= n^{-1/2} \| W' \MI h_0(Z) \|_2 \leq J_4 +  {K}^{1/2}J_{8} J_1$
\item[6.] $ \Xi_6: = n^{-1/2} \| W' \PI \mathscr E  \|_2 \leq K^{1/2}J_9J_2 $
\item[7.] $ \Xi_7:= n^{-1/2} \| m' \MI \mathscr E \|_2 \leq J_4 + K^{1/2}J_7J_2$
\item[8.] $ \Xi_8:=n^{-1/2} | m_a' \MI h_0(Z) | \leq n^{1/2}J_5J_{13} $
\item[9.] $ \Xi_9:= n^{-1/2} | W_a' \MI h_0(Z) | \leq  J_{12} +  J_{14} J_1$
\item[10.] $ \Xi_{10}: = n^{-1/2} | W_a' \PI \mathscr E  | \leq J_9J_{11} $
\item[11.] $ \Xi_{11}:= n^{-1/2} | m_a' \MI \mathscr E | \leq J_{12} + J_7J_{11}$.
\end{itemize}
\end{lemma}


\begin{proof} 

\

\

\noindent \textbf{Statement 1.}  \begin{align*}
\(n^{-1} \| W' \PI W \|_\F\)^2 & =\sum_{k,\bar k \leq K} (n^{-1} W_k' \PI W_{\bar k} )^2  = 
\\& = \sum_{k,\bar k  \leq K} (n^{-1} \hat \beta_{W_k , I_{\PhiK + RF}}'Q ' W_{\bar k})^2 \\ 
&\leq \sum_{k,\bar k \leq K} \| n^{-1/2} \hat \beta_{W_{ k}  , I_{\PhiK + RF}} \|_1 ^2 \|n^{-1/2}   Q'  W_{\bar k} \|_\infty ^2 \\
&=\(\sum_{k\leq K}   \|  n^{-1/2} \hat \beta_{W_k , I_{\PhiK + RF}} \|_1^2\)  \(\sum_{\bar k\leq K} \|n^{-1/2}  Q' W_{\bar k} \|_{\infty}^2\) \\
&\leq K \cdot n^{-1} J_9^2 \cdot K \cdot  J_1^2\\ \ \\
& \Rightarrow n^{-1}\| W' \PI W \|_\F \leq  n^{-1/2}KJ_1J_9.\\
\end{align*}

\

\noindent \textbf{Statement 2.}
\begin{align*}
\(n^{-1} \|  m' \MI m \|_\F\)^2 & = \sum_{k,\bar k \leq K} (n^{-1} m_k' \MI m_{\bar k} )^2   \leq  \sum_{k,\bar k \leq K} \|n^{-1/2} \MI m_k  \|_2^2 \|n^{-1/2} \MI m_{\bar k} \|_2^2 \\
&= \(\sum_{k\leq K} \| n^{-1/2} \MI m_k \|_2^2\)^2 \leq K^2 J_5^4 \\ \ \\
& \Rightarrow  n^{-1}\|m' \MI m \|_\F \leq KJ_5^2. 
\end{align*}

\noindent \textbf{Statement 3.}

\begin{align*}
n^{-1} \|m' \MI W\|_{\F} &=n^{-1} \| m' W/n - m' \PI W \|_\F \\
&= n^{-1} \| R_m'W + (Q\beta_{p^K,L,s_0})'W - m' \PI W\|_\F \\
&= n^{-1} \| R_m'W + (Q\beta_{p^K,L,s_0})'W - (Q\hat \beta_{m,I_{\PhiK + \RF} })'W \|_\F \\
&= n^{-1} \| R_m'W + (\beta_{p^K,L,s_0}-\hat \beta_{m,I_{\PhiK + \RF} })'Q'W \|_\F \\
&\leq n^{-1} \| R_m'W\|_\F + n^{-1}\| (\beta_{p^K,L,s_0}-\hat \beta_{m,I_{\PhiK + \RF} })'Q'W \|_\F. 
\end{align*}
Then the first term in the last line is bounded above as $n^{-1} \| R_m'W\|_\F = J_{16}$ while the second term has 
\begin{align*} &\(n^{-1} \| (\beta_{p^K,L,s_0}-\hat \beta_{m,I_{\PhiK + \RF} })'Q' W\|_{\F}\) ^2 \\
& =n^{-2}  \sum_{k,\bar k \leq K }( (\beta_{p_k,L,s_0}-\hat \beta_{m_k,I_{\PhiK + \RF} })'Q'W_{\bar k})^2 \\
&\leq  n^{-2}\sum_{k,\bar k \leq K } \|\beta_{p_k,L,s_0}-\hat \beta_{m_k,I_{\PhiK + \RF} } \|_1^2 \| Q'W_{\bar k} \|_{\infty}^2 \\
&=n^{-1}\(\sum_{k \leq K } \|\beta_{p_k,L,s_0}-\hat \beta_{m_k,I_{\PhiK + \RF} } \|_1^2 \) \(\sum_{\bar k \leq K } \| n^{-1/2} Q'W_{\bar k} \|_{\infty}^2 \) \\
&\leq n^{-1}K \cdot J_7^2 \cdot K \cdot J_1^2. 
\end{align*}
Therefore, $n^{-1} \|m' \MI W\|_{\F} \leq J_{16} + n^{-1/2}KJ_7J_1.$ 

\

\noindent \textbf{Statement 4.}
 \begin{align*}n^{-1/2} \|m' \MI h_0(Z) \|_2  & \leq n^{1/2}    \| n^{-1/2} \MI h_0(Z) \|_2 K^{1/2}  \max_{k \leq K} n^{-1/2} \|m' \MI \|_2   \\ &\leq n^{1/2}K^{1/2} J_5J_6. \end{align*}

\noindent \textbf{Statement 5.}
\begin{align*}
n^{-1/2} \| W' \MI h_0(Z) \|_2 & = n^{-1/2} \| W'  h_0(Z) - W' \PI h_0(Z) \|_2 \\
& = n^{-1/2} \| W'  h_0(Z) - W ' Q \hat \beta_{h_0(Z),I_{\PhiK + \RF}}  \|_2 \\
& =    n^{-1/2} \| W' R_{h_0} + W'  Q \beta_{h_0,L,s_0} - W ' Q \hat \beta_{h_0(Z),I_{\PhiK + \RF}}  \|_2 \\
& = n^{-1/2}\( \| W' h_0(Z) \|_2 +  \| (\hat \beta_{h_0(Z),I_{\PhiK + \RF}}  - \beta)' Q'W\|_2 \) \\
& \leq  J_4 +  {K}^{1/2} \max_{k \leq K} \| \betahat_{h_0(Z), I_{\PhiK + \RF}} - \beta_{h_0,L,s_0} \|_1 \|n^{-1/2}Q' W_k\|_{\infty} \\ 
& \leq J_4 +  {K}^{1/2}J_{8} J_1. 
\end{align*}

\noindent \textbf{Statement 6.}  \begin{align*}
\(n^{-1/2} \| W' \PI W \|_2\)^2 & = n^{-1} \sum_{k  \leq K} ( W_k' \PI \EC )^2   
\\& = n^{-1} \sum_{k  \leq K} ( \hat \beta_{W_k , I_{\PhiK + RF}}'Q ' \EC)^2 \\
& \leq \sum_{k \leq K} \|  \hat \beta_{W_{ k}  , I_{\PhiK + RF}} \|_1 ^2 \|n^{-1/2}   Q' \EC  \|_\infty ^2 \\
&\leq K \cdot  J_9^2 \cdot  J_2^2\\
& \Rightarrow n^{-1/2}\| W' \PI \EC \|_\F \leq  K^{1/2}J_9J_2.
\end{align*}

\noindent \textbf{Statement 7.}
\begin{align*}
n^{-1/2} \|m' \MI \EC\|_{2} &=n^{-1/2} \| m' \EC /n - m' \PI \EC \|_\F \\
&= n^{-1/2} \| R_m' \EC + (Q\beta_{p^K,L,s_0})'W - m' \PI \EC \|_\F \\
&= n^{-1/2} \| R_m' \EC + (Q\beta_{p^K,L,s_0})'W - (Q\hat \beta_{m,I_{\PhiK + \RF} })'\EC \|_2 \\
&= n^{-1/2} \| R_m' \EC + (\beta_{p^K,L,s_0}-\hat \beta_{m,I_{\PhiK + \RF} })'Q'\EC \|_2 \\
&\leq n^{-1/2} \| R_m' \EC \|_2 + n^{-1/2}\| (\beta_{p^K,L,s_0}-\hat \beta_{m,I_{\PhiK + \RF} })'Q'\EC \|_2. 
\end{align*}
Then the first term in the last line is bounded above as $n^{-1/2} \| R_m'\EC \|_2 = J_4$.  Turning to the second term, 
\begin{align*} & \(n^{-1/2} \| (\beta_{p^K,L,s_0}-\hat \beta_{m,I_{\PhiK + \RF} })'Q' \EC\|_{2}\)^2 \\
& =n^{-1}  \sum_{k\leq K }( (\beta_{p_k,L,s_0}-\hat \beta_{m_k,I_{\PhiK + \RF} })'Q'\EC)^2 \\
&\leq  \sum_{k \leq K } \|\beta_{p_k,L,s_0}-\hat \beta_{m_k,I_{\PhiK + \RF} } \|_1^2 \|n^{-1/2}  Q'\EC \|_{\infty}^2 \\
&\leq K \cdot J_7^2  \cdot J_2^2. 
\end{align*}
Therefore, $n^{-1/2} \|m' \MI \EC\|_{2} \leq J_4 + K^{1/2}J_7J_2$.

\


\noindent \textbf{Statements 8-11.} 

\

The argument is identical to the argument for Statements 4-7, adjusting appropriately for the fact that $m_a$ is $1$-dimensional rather than $K$-dimensional.
\end{proof}

The following corollaries follow directly from assumed rate conditions and the above bounds.  These are used in the proof of Theorems 1 and 2.
 
 \
 
\begin{corollary}

Under the assumptions of Theorem 1,
\begin{itemize}  
\item[1.] $\Xi_1 + \Xi_2 + \Xi_3 = O_p( n^{-1/2} \zeta_0(K)K^{1/2} )$
\item[2.] $n^{-1/2} (  \Xi_4 + \Xi_5 +  \Xi_6 + \Xi_7)  = O_p(n^{-1/2}K^{1/2} + K^{-\alpha_{g_0}}).$
\end{itemize}
\end{corollary}

\begin{corollary}

Under the assumptions of Theorem 2,
\begin{itemize}
\item[1.] $ n^{-1/2} \zeta_0(K)K^{1/2} (\Xi_4 + \Xi_5 + \Xi_6 + \Xi_7) = o_p(1).$
\item[2.] $\Xi_8 + \Xi_9 + \Xi_{10} + \Xi_{11} = o_p(1).$
\end{itemize}
\end{corollary}

 \begin{proof} Corollaries 1 and 2 follow from the rate conditions stated in Assumptions 2,9,11.  The also follow from the following more general rate conditions.
 
 {Sufficient rate conditions for Corollary 1:
\begin{itemize}
\item[1.] $s_0K^{\alpha_{I_\Phi}} = o(s_\kappa)$ 
\item[2.]$\log(KL) = o(\zeta_0(K)^{-1}n^{1/3})$
\item[3.]$ L^{-\alpha_{\mathscr Z}} n^{1/2} K^{-1/2} \zeta_0(K) = O(1)$ 
\item[4.]$ L^{-2\alpha_{\mathscr Z}} K^{2\alpha_{\rho}}( K^{1/2}n^{1/2} \zeta_0(K)^{-1} + n^{1/2} + K \log(L)^{1/2} \zeta_0(K)^{-2} )= O(1)$
\item[5.] $n^{-1/2}K^{1/2}s_0 \log(L) \zeta_0(K)^{-1} (K^{ 2 \alpha_{\rho} + \alpha_{\Phi}} + K^{\alpha_{\rho} + \alpha_{\Phi}/2 + \alpha_{I_\Phi}/2}) =O(1)$ 
\item[6.] $n^{-1/2} s_0^{1/2} \log(L) (K^{2 \alpha_{\rho} + \alpha_{\Phi} }s_0^{1/2} + K^{\alpha_{I_\Phi}/2} )= O(1)$.
\end{itemize}
Additional rate conditions sufficient (along with the above conditions) for Corollary 2:
\begin{itemize}  
\item[1.] $L^{-2\alpha_{\mathscr Z}}K^{2\alpha_{\rho}}(\zeta_0(K)K + \zeta_0(K)^4K^{1 - 2\alpha_{\rho}}+ K \log(L)+ n^{1/2} ) = o(1)$\item[2.] $n^{-1}s_0 \zeta_0(K)\log(L) (K^{1 + 2\alpha_{\rho} + \alpha_\Phi} + K^{1 + \alpha_\rho + \alpha_{\Phi}/2 + \alpha_{I_\Phi}/2})=o(1)$\item[3.] $n^{-1}s_0^2 \log(L)^2 (K^{4\alpha_{\rho} + 2\alpha_\Phi} + K^{2\alpha_\rho +  \alpha_{\Phi} + \alpha_{I_\Phi}})=o(1)$\item[4.] $s_0 K^{\alpha_{I_{\Phi}}} \(n^{-1/2} \zeta_0(K) K^{1/2} + K^{-\alpha_{g_0}}\) = o(1)$ 
\item[5.]  $ n^{2/(4+\delta)}\zeta_0(K)n^{-1/2}K^{1/2} =o(1).$  
\end{itemize}
}

 \end{proof}

 \
 
 \
 
 \
 
\subsection{Proof of Theorem 1}

\begin{lemma}

\

 \begin{itemize}
\item[1.]$\| \hat \Omega - \Omega \|_{\mathscr F} \leq O_p(\zeta_0(K)  K^{1/2} n^{-1/2}) + \Xi_1 + \Xi_2 + \Xi_3  = o_p(1)$   
\item[2.] $\|\hat \Omega^{-1} - \Omega^{-1}\|_{2 \rightarrow 2} = O_p(\zeta_0(K)  K^{1/2} n^{-1/2}) + O(\Xi_1 + \Xi_2 + \Xi_3).$
\end{itemize}
\end{lemma}

 \begin{proof} The argument in Theorem 1 of \cite{newey:series} 
 gives the bound $\| \bar \Omega - \Omega\|_{\mathscr F} = O_\Pr(\zeta_0(K)  K^{1/2} n^{-1/2})$.   
Next, using the decomposition, $P = m + W$,  write $ \hat \Omega = (m+W)' \MI (m+W)/n = W'W/n - W'(\text{Id}_n - \MI )W/n + m' \MI m/n + 2 m' \MI W/n$.  By triangle inquality, $\| \bar \Omega - \hat \Omega \|_{\mathscr F}  \leq \|W'  \PI W/n \|_{\mathscr F}  + \| m' \MI m/n \|_{\mathscr F}  +  \|  2m' \MI W/n \|_{\mathscr F}  = \Xi_1 + \Xi_2 + \Xi_3$.  Bounds for each of the three above terms are established above along with the assumed rate conditions give $\| \bar \Omega - \hat \Omega \|_{\mathscr F}  = o_p(1)$.  
The last statement holds by applying an expansion of the matrix inversion function around $\text{Id}_K$.
\begin{align*}
\hat \Omega^{-1}  &= (\text{Id}_K - (\text{Id}_K - \hat \Omega))^{-1} = \text{Id}_K+  (\text{Id}_K - \hat \Omega ) +  (\text{Id}_K - \hat \Omega )^2 + ...
\end{align*} 
The sum given above is with probability $\rightarrow 1$ absolutely convergent relative to the Frobenius norm $\F$.  In addition, by the bound $\|  \cdot  \|_{2 \rightarrow 2} \leq \|  \cdot  \|_{\F}$, we have  $
\| \hat \Omega^{-1} - \text{Id}_K \|_{2 \rightarrow 2}  \leq  \| \hat \Omega - \text{Id}_K \|_{\F}
 \leq \| \text{Id}_K - \hat \Omega \|_\F  + \| \text{Id}_K - \hat \Omega \|_\F^2 +...
=O_p(\zeta_0(K)  K^{1/2} n^{-1/2}) + O(\Xi_1 + \Xi_2 + \Xi_3)$.

 \end{proof}

\noindent Note that since $\Omega$ has minimal eigenvalues bounded from below by assumption, it follows that $\hat \Omega$ and $\bar \Omega$ are invertible with probability approaching 1. The reference \cite{newey:series} works on the event $1_n:=\{\lambda_{\min} (\hat \Omega) > 1/2 \}$ and later uses the fact that this event has probability $\rightarrow 1$.  This fact is used several times, however its use is only implicitly in reference to arguments in \cite{newey:series}.

\begin{lemma}
$ \| \hat \Omega^{-1} n^{-1} P' \MI \mathscr E \|_2 = O_p(n^{-1/2}K^{1/2})$.\end{lemma}

\begin{proof} 
\begin{align*} \| \hat \Omega^{-1} n^{-1} P' \MI \mathscr E \|_2 & \leq \| \hat \Omega^{-1} \|_{2 \rightarrow 2} n^{-1} \|P' \MI \mathscr E \|_2 \\ & \leq  \| \hat \Omega^{-1} \|_{2 \rightarrow 2} (n^{-1} \|W' \EC\|_2  + n^{-1/2} \Xi_6 + n^{-1/2}\Xi_7 )\end{align*}
$\|W'\EC\|_2 = O_p(n^{-1/2}K^{1/2} = O_p(n^{-1/2}) $ by arguments in \cite{newey:series}.  Bounds for $n^{-1/2} \Xi_6 + n^{-1/2}\Xi_7$ follows from the previous Lemmas and from the assumed rate conditions.
\end{proof}

\begin{lemma} $ \| \hat \Omega^{-1} P' \MI (g_0(X) - P \beta_{g_0,K})/n \|_{2} = O_p(K^{-\alpha_{g_0}})$.
\end{lemma}
\begin{proof}

\begin{align*}
 \| \hat \Omega^{-1} P' \MI (g_0(X) - P \beta_{g_0,K})/n \|_2 &= [(g_0(X)-P\beta)'\MI P \hat \Omega^{-1} P' \MI (g_0(X)-P\beta)/n ]^{1/2}\\ 
& = O_p(1) [(g_0(X)-P\beta_{g_0,K})'(g_0(Z)-P\beta_{g_0,K})/n]^{1/2}  \\
& = O_p(K^{-\alpha_{g_0}}) 
\end{align*}
by assumption on $(g_0(X)-P\beta_{g_0,K})$ and idempotency of $\MI P \hat \Omega^{-1} P' \MI = \MI P (P' \MI \MI P)^{-1} \MI $.
\end{proof}

\begin{lemma}
$ \|\hat \Omega^{-1} P' \mathscr \MI h_0(Z)/n \|_2 = o_p(n^{-1/2})$.
\end{lemma}
\begin{proof} $\hat \Omega$ has eigenvalues bounded below and above with probability approaching 1.  Then, 
\begin{align*} 
\| \hat \Omega^{-1} P' \MI h_0(Z) /n\|_2 & \leq O_p(1) \| P' \MI h_0(Z) /n\|_2 \\
 & = O_p(1) \| (m+W)' \MI h_0(Z) /n\|_2 \\
& = n^{-1/2} O_p(1) n^{-1/2}  \| (m+W)' \MI h_0(Z) /n\|_2 \\
& \leq n^{-1/2} O_p(1) (n^{-1/2}  \| m' \MI h_0(Z) \|_2  + n^{-1/2}\| W' \MI h_0(Z) \|_2) \\
& = n^{-1/2} O_p(1) ( \Xi_4 + \Xi_5)\\
& = n^{-1/2} O_p(1)o_p(1).
\end{align*}
\end{proof}

\begin{lemma} $\| \hat \beta_g-\beta_{g_0,K} \|_2= O_p(n^{-1/2} K^{1/2}+ K^{-\alpha_{g_0}}).$ 
\end{lemma}  
\begin{proof}Note that $ ([\hat \beta_{y,(\tilde p, \tilde q)}]_g - \beta_{g_0,K} ) = n^{-1}  \hat \Omega^{-1} P' \MI \mathscr E  + n^{-1}\hat \Omega^{-1} P' M_{\hat I} ( g_0(X) - P \beta_{g_0,K}) +  n^{-1}\hat \Omega^{-1} P' \MI h_0(Z) $.  Triangle inequality in conjuction with the bounds described in the previous three lemmas give the result.  
\end{proof}

The final statement of Theorem 1 follows from the bound on $\| \hat \beta_g-\beta_{g_0,K} \|_2$ using the arguments in \cite{newey:series}. \QED

\subsection{Proof of Theorem 2}
Recall that $F = V^{-1/2}$.  Let $\bar g = p^K(x)'\beta_{g_0,K}$ and decompose the quantity $n^{1/2} F [a(\hat g) - a( g_0)]$ by 
\begin{align*}
n^{1/2}  F [a(\hat g) - a( g_0)] &= n^{1/2}  F [  a(\hat g) - a( g_0) + D(\hat g) - D( g_0) \\
& + D(\bar g) - D(\hat g) \\
&+ D(g_0) - D(\bar g)].
\end{align*}

\begin{lemma} $ n^{1/2} F [D(\bar g) - D(g_0)]  =O( n^{1/2} K^{-\alpha_{g_0}}).$
\end{lemma} \begin{proof} This follows from arguments given in the proof of Theorem 2 in \cite{newey:series}.  Note that the statement does not contain any reference to random quantities. \end{proof}

\begin{lemma} $|n^{1/2} F [a(\hat g ) - a(g) - D(\hat g) + D(g) | = o_p(1).$
\end{lemma}

\begin{proof} Bounds on $| \hat g - g |_d $ given by Theorem 1 imply that $ |n^{1/2} F [a(\hat g ) - a(g_0) - D(\hat g) + D(g_0) |  \leq C n^{1/2} | \hat g - g_0 |_d ^2 
= O_p(n^{1/2} (n^{-1/2}\zeta_d(K)K^{1/2} + K^{-\alpha_{g_0}} )^2) = o_p(1).$  This is again identical to the reasoning given in Theorem 2 in \cite{newey:series}, since that references uses only a bound on $| \hat g - g |_d$ to prove the analogous result.
\end{proof}

The last step is to show that $n^{1/2} F [D(\hat g) - D(\bar g)] \rightarrow_d \N(0,1)$.

\begin{lemma} $n^{1/2} F [D(\hat g) - D(\bar g)] \rightarrow_d \N(0,1).$
\end{lemma} 

\begin{proof}
Note that $D(\hat g)$ can be expanded 
\begin{align*}
 D(\hat g) &= D( p^K(x) ' [\hat \beta_{y,(\tilde p, \tilde q)}]_g)  =  D( p^K(x) ' \hat \Omega^{-1} n^{-1}  P' \MI Y) \\
&=   D( p^K(x) '  \hat \Omega^{-1} n^{-1} P' \MI (g_0(X) + h_0(Z) + \mathscr E)) \\
& = D( p^K(x) )' \hat \Omega^{-1} n^{-1} (g_0(X) + h_0(Z) + \mathscr E) \\
 &=  A'  \hat \Omega^{-1} n^{-1} P' \MI (g_0(X) + h_0(Z) + \mathscr E) \\
 & =  A' \hat \Omega^{-1} n^{-1} P' \MI g_0(X) +  A' \hat \Omega^{-1} n^{-1} P' \MI h_0(Z) + A' \hat \Omega^{-1} n^{-1} P' \MI \mathscr E.
\end{align*}
In addition, $D(\bar g) = D( p^K(x)' \beta_{g_0,K}) = A' \beta_{g_0,K}$ gives 
\begin{align*}
n^{1/2} F [D(\hat g) - D(\bar g)] & = n^{1/2} F A' [ \hat \Omega^{-1}n^{-1}  P' \MI g_0(X) -  \beta_{g_0,K} ]\\ 
& + n^{1/2} FA'[ \hat \Omega^{-1} n^{-1} P' \MI (h_0(Z)+ \mathscr E )].
\end{align*}
The above equation gives a decomposition of the right hand side into two terms, which are next bounded separately.  Before proceeding, note that the following bounds $\| FA \|_2=O(1),$ $ \|FA' \hat \Omega^{-1}\|_2 =O_p(1)$, $\|FA' \hat \Omega^{-1/2}\|_2 =O_p(1)$, $ \|FA' \Omega^{-1}\|_2 =O_p(1)$, $\|FA'\Omega^{-1/2}\|_2 =O(1)$ all hold by arguments in \cite{newey:series}.  Consider the first term.
\begin{align*}
& | n^{1/2} FA'  [ n^{-1} \hat \Omega^{-1} P' \MI g_0(X) -  \beta_{g_0,K} ]  |\\
& =  | \sqrt n FA' [ (P' \MI P/n)^{-1} P' \MI (G - P  \beta)/n] |\\
& \leq  \|FA '  \hat \Omega^{-1} P' \MI / \sqrt n \|_2 \|g_0(X) - P  \beta_{g_0,K} \|_2 \\
& \leq  \|FA' \hat \Omega^{-1} P'\MI /\sqrt{n} \|_2 \sqrt n \max_{i \leq n} |g(x_i) - \bar g(x_i)| \\
& =  \|FA' \hat \Omega^{-1/2} \|_2  \sqrt{n}\max_{i \leq n} |g(x_i) - \bar g(x_i)| \\
& \leq \|FA' \hat \Omega^{-1/2}   \|_2 \sqrt n |g - \bar g |_0 \\ 
&= O_p(1) O_p(\sqrt n K ^ {-\alpha}) \\& = o_p(1). 
\end{align*}

Next, consider $n^{1/2}  FA' \hat \Omega^{-1} n^{-1}  P' \MI (h_0(Z)+\EC) $.  
To handle this term, first bound
\begin{align*}
& |n^{-1/2} FA' (\hat \Omega^{-1} - \Omega^{-1} ) P' \MI (h_0(Z) + \EC) | \\
&\leq n^{-1/2}   \| FA' (\hat \Omega^{-1} - \Omega^{-1} ) \|_2 \| P' \MI (h_0(Z) + \EC) \|_2 \\
&=  \| FA' (\hat \Omega^{-1} - \Omega^{-1} ) \|_2 (n^{-1/2} \| P' \MI (h_0(Z) + \EC) \|_2) \\
&=  \| FA' (\hat \Omega^{-1} - \Omega^{-1} ) \|_2 (n^{-1/2} \| (m + W)' \MI (h_0(Z) + \EC) \|_2) \\
&\leq   \| FA' (\hat \Omega^{-1} - \Omega^{-1} ) \|_2 ( \Xi_4 + \Xi_5 + \Xi_6 + \Xi_7) \\
& \leq \| \hat \Omega^{-1} - \Omega^{-1}  \|_{2 \rightarrow 2} \| FA' \|_2  ( \Xi_4 + \Xi_5 + \Xi_6 + \Xi_7)\\
& = o_p(1).
\end{align*}
Next consider the last remaining term for which a central limit result will be shown.  
\begin{align*}
&\sqrt n F A' \Omega ^{-1}P'\MI (h_0(Z) + \mathscr E)/n \\
&= \sqrt n F A' \Omega^{-1} (W+m)'\MI(h_0(Z) + \EC)/n  \\
&= \sqrt n (F A' \Omega^{-1}W + m_a)'\MI(h_0(Z) + \EC)/n  \\
&= \sqrt n F A' \Omega^{-1}W\MI \EC + \sqrt{n} m_a'\MI(h_0(Z) + \EC)/n + \sqrt{n} W_a'\MI h_0(Z)/n\\
&=  \sqrt n F A' \Omega^{-1}W \EC - \sqrt n F A' \Omega^{-1}W \PI \EC + \sqrt{n} m_a'\MI(h_0(Z) + \EC)/n + \sqrt{n} W_a'\MI h_0(Z)/n\\
&= \sqrt n FA'\Omega^{-1}W'\EC/n + o_p(1). \\
\end{align*}

Note that the last $o_p(1)$ bound in the equation array above holds by the fact that $|\sqrt n F A' \Omega^{-1}W \PI \EC + \sqrt{n} m_a'\MI(h_0(Z) + \EC)/n + \sqrt{n} W_a'\MI h_0(Z)/n| \leq \Xi_8 + \Xi_9 + \Xi_{10} + \Xi_{11}$.   The term $\sqrt n FA'\Omega^{-1}W'\EC/n$ satisfies the  conditions Lindbergh-Feller Central Limit Theorem, by arguments given in \cite{newey:series}.  
\end{proof}
The previous three lemmas prove that  $n^{1/2} F [a(\hat g) - a( g_0)] \rightarrow \N(0,1)$.

\

The next set of arguments bound $\hat V - V$.  For $\nu$ as in the statement of Assumption 14,   Define the event $ \mathscr A_g = \{|\hat g - g_0 |_d < \nu/2 \}$.  
 Define $\hat u = 1_{\mathscr A_g} \hat \Omega^{-1} \hat A F$ and $u = 1_{\mathscr A_g}\Omega^{-1} AF$.   In addition, define $\bar \Sigma = \sum_i W_i W_i' \varepsilon_i^2 /n$, an infeasible sample analogue of $\Sigma$.

\begin{lemma}
\
\begin{itemize}
\item[1.] $\| \hat A - A\|_2 = o_p(1)$
\item[2.] $\| \hat u - u \|_2 = o_p(1)$
\item[3.] $\| \bar  \Sigma - \Sigma\|_{\mathscr F} =o_p(1)$
\item[4.] $|\hat u \bar \Sigma \hat u - \hat u' \Sigma \hat u| =o_p(1)$.
\end{itemize}
\end{lemma}

\begin{proof}  

\noindent \textbf{Statement 1.  } In the case that $a(g)$ is linear in $g$, then $a(p' \beta) = A' \beta \implies \hat A = A$.  Therefore, consider the case that $a(g)$ is not linear in $g$. 
Using arguments identical to those in \cite{newey:series}, $1_{\mathscr A_g} =1 $ with probability $\rightarrow 1$, and $$1_{\mathscr A_g} \|\hat A - A\|_2 \leq C \cdot \zeta_d(K) |\hat g - g|_d.$$

\

\noindent \textbf{Statement 2.}  This follows from arguments in \cite{newey:series}.
  
  \
  
\noindent \textbf{Statement 3.}  This follows from arguments in \cite{newey:series}.

\

\noindent \textbf{Statement 4.}  An immediate implication of Statement 3 is that $ 1_{\mathscr A_g} |\hat u '\bar \Sigma \hat u - \hat u' \Sigma \hat u| = |\hat u ' (\bar \Sigma - \Sigma)\hat u| \leq \| \hat u \|_2^2 \|\bar \Sigma - \Sigma\|_{2\rightarrow 2}^2 = O_p(1) o_p(1).$

\end{proof}

\begin{lemma}
$\max_{i \leq n} |h_0(z_i) - \hat h(z_i)|  = o_p(1)$.
\end{lemma}

\begin{proof}
First note that 
  $$\max_i |h_0(z_i) - \hat h(z_i)| \leq \max_i |h_0(z_i) - q^L(z_i)' \beta_{h_0,L,s_0}| $$ $$ \hspace{5cm}+ \max_i |q^L(z_i)'[\hat \beta_{y,(\tilde p, \tilde q)}]_h - q^L(z_i)' \beta_{h_0,L,s_0}|.$$  The first term has the bound $ \max_i |h_0(z_i) - q(x_i)' \eta|  = O_p(L^{-\alpha_{\mathscr Z}} )$ by assumption. Next,   
  \begin{align*}
\max_i |q^L(z_i)'[\hat \beta_{y,(\tilde p, \tilde q)}]_h - q^L(z_i)' \beta_{h_0,L,s_0}| &= \max_i |q^L(z_i)'([\hat \beta_{y,(\tilde p, \tilde q)}]_h - \beta_{h_0,L,s_0})| \\
& \leq \max_i \| q^L(z_i)\|_{\infty} \| [\hat \beta_{y,(\tilde p, \tilde q)}]_h - \beta_{h_0,L,s_0}\|_1 \\
\end{align*}
Then,
\begin{align*}
\|[\hat \beta_{y,(\tilde p, \tilde q)}]_h - \beta_{h_0,L,s_0} \|_1 &= \| \hat \beta_{y - \hat g, I_{\PhiK + \RF}} -\beta_{h_0,L,s_0}  \|_1\\
& = \|  \hat \beta_{g_0, I_{\PhiK + \RF}} + \hat \beta_{h_0, I_{\PhiK + \RF}}  + \hat \beta_{\varepsilon, I_{\PhiK + \RF}}   - \hat \beta_{ \hat g, I_{\PhiK + \RF}}  -\beta_{h_0,L,s_0} \|_1. \\
&\leq \| \hat \beta_{h_0, I_{\PhiK + \RF}} - \beta_{h_0,L,s_0}\|_1 + \|  \hat \beta_{\varepsilon, I_{\PhiK + \RF}} \|_1 + \| \hat \beta_{g_0 - \hat g, I_{\PhiK + \RF}} \|_1\\
& = J_8 + J_{10} +\| \hat \beta_{g_0 - \hat g, I_{\PhiK + \RF}} \|_1.
\end{align*}

Next, \begin{align*}
& \| \hat \beta_{g_0 - \hat g, I_{\PhiK + \RF}} \|_1  \\& \leq | I_{\PhiK + \RF}|^{1/2}  \| \hat \beta_{g_0 - \hat g, I_{\PhiK+ \RF}}\|_2 \\
& = |I_{\PhiK + \RF}|^{1/2} \|(Q_{I_{\PhiK + \RF}}'Q_{I_{\PhiK + \RF}}/n)^{-1} Q_{I_{\PhiK + \RF}}'  (g_0(X) - \hat g(X))/n\|_2 \\
& \leq |I_{\PhiK + \RF}|^{1/2}  \kappa_{\min}(| I_{\PhiK + \RF} |)^{-1} \| Q_{I_{\PhiK + \RF}}'  (g_0(X) - \hat g(X)) /n\|_2\\
& \leq |I_{\PhiK + \RF}|^{1/2}  \kappa_{\min}(| I_{\PhiK + \RF} |)^{-1} |I_{\PhiK + \RF}|^{1/2}  \| Q_{I_{\PhiK + \RF}}'  (g_0(X) - \hat g(X))/n \|_\infty\\
& \leq |I_{\PhiK + \RF}|  \kappa_{\min}(| I_{\PhiK + \RF} |)^{-1} \| Q_{I_{\PhiK + \RF}}'  (g_0(X) - \hat g(X)) /n\|_\infty\\
& =  |I_{\PhiK + \RF}|  \kappa_{\min}(| I_{\PhiK + \RF} |)^{-1}   \(\max_{j}  n^{-1} \sum_{i=1}^n| q_{jL}(z_i) | \) \|  g_0(X) - \hat g(X)  \|_\infty \\
&= O_p(s_0^{1}K^{\alpha_{I_\Phi}}) O_p(1) o_p(n^{-1/2} \zeta_0(K) K^{1/2} + K^{-\alpha_{g_0}}).
\end{align*}
 
 Putting these together, it follows from the assumed rate conditions that $$\max_{i} | h_0(z_i) - \hat h(z_i)| = o_p(1).$$

\end{proof}

Next, let $\Delta_{g_0i} = g_0(x_i) - \hat g(x_i)$ and $\Delta_{h_0i} = h_0(z_i) - \hat h(z_i)$.  Then above lemma states $\max_{i \leq n} \Delta_{h_0i} = o_p(1)$.  In addition $\max_{i \leq n} |\Delta_{g_0i}| \leq |\hat g - g|_0 = o_p(1) $.  
Let $\omega_i^2 =  u' W_i W_i'  u$ and $\hat \omega_i^2 = \hat u' W_i W_i' \hat u$.  

\begin{lemma} $|F \hat V F - \hat u '\bar \Sigma \hat u| =o_p(1)$.
\end{lemma}

\begin{proof}
\begin{align*}
1_{\mathscr A_g} |F \hat V F - \hat u' \bar \Sigma \hat u| &=| \hat u' ( \hat \Sigma - \bar \Sigma) \hat u | = \left | \sum_{i=1}^n \hat u' \hat W_i \hat W_i' \hat \varepsilon_i^2  \hat u/n -  \sum_{i=1}^n  \hat u' W_i W_i' \varepsilon_i^2  \hat u/n \right |  \\ 
& \leq \left | \sum_{i=1}^n  \omega_i^2 (\hat \varepsilon_i^2 - \varepsilon_i^2)  /n \right |+
  \left | \sum_{i=1}^n (\hat \omega_i^2- \omega_i^2) \hat \varepsilon_i^2 /n  \right |.
\end{align*}
Both terms on the right hand side will be bounded.  Consider the first term.  Expanding $ (\hat \varepsilon_i^2 - \varepsilon_i^2) $ gives
$$ \left | \sum_{i=1}^n  \omega_i^2 (\hat \varepsilon_i^2 - \varepsilon_i^2)  /n \right | \leq  \left | \sum_{i=1}^n \omega_i^2 \Delta_{1i}^2/n \right | + \left | \sum_{i=1}^n \omega_i^2  \Delta_{2i}^2/n\right | + \left | \sum_{i=1}^n\omega_i^2\Delta_{1i}\Delta{2i}/n \right | $$ $$
+ 2 \left | \sum_{i=1}^n \omega_i^2  \Delta_{1i}\varepsilon_i/n \right | + 2 \left | \sum_{i=1}^n \omega_i^2 \Delta_{2i}\varepsilon_i/n \right |.$$
Note that $\sum_{i=1}^n \omega_i^2/n, \sum_{i=1}^n \omega_i^2|\varepsilon_i| = O_p(1)$ by arguments is in \cite{newey:series}.  The five terms above are then bounded in order of their appearence by
\begin{align*}
&\sum_{i=1}^n \omega_i^2  \Delta_{1i}^2/n \leq \max_{i\leq n} |\Delta_{1i}| \sum_{i=1}^n \omega_i^2/n = o_p(1)O_p(1)  \\
& \sum_{i=1}^n \omega_i^2  \Delta_{2i}^2/n \leq \max_{i \leq n}  |\Delta_{2i} | \sum_{i=1}^n \omega_i^2 |\varepsilon_i|/n = o_p(1) O_p(1) \\
&\sum_{i=1}^n \omega_i^2  \Delta_{1i}\Delta_{2i}/n \leq \max_{i\leq n} |\Delta_{1i}|\max_{i\leq n}|\Delta_{2i}| \sum_{i=1}^n \omega_i^2  /n = o_p(1)O_p(1)  \\
&\sum_{i=1}^n \omega_i^2  \Delta_{1i}\varepsilon_i/n \leq \max_{i\leq n} |\Delta_{1i}| \sum_{i=1}^n \omega_i^2 |\varepsilon_i| /n = o_p(1)O_p(1)  \\
&\sum_{i=1}^n \omega_i^2  \Delta_{2i}\varepsilon_i/n \leq \max_{i\leq n} |\Delta_{2i}| \sum_{i=1}^n \omega_i^2 |\varepsilon_i| /n = o_p(1)O_p(1).  \\
\end{align*}

The second term is bounded by
\begin{align*}
&\left | \sum_{i=1}^n \hat u '(\hat W_i \hat W_i' - W_i W_i') \hat \varepsilon_i^2 \hat u/n \right | \leq \max_{i \leq n} |\varepsilon_i^2 | \left | \sum_{i=1}^n \hat u'(\hat W_i \hat W_i' - W_i W_i') \hat u/n \right | \\
&\leq \max_{i \leq n} |\hat \varepsilon_i^2 | \| \hat u \|^2_2 \| \sum_{i=1}^n (\hat W_i \hat W_i' - W_i W_i') /n \|_{2 \rightarrow 2}   = \max_{i \leq n} |\hat \varepsilon_i^2 | \| \hat u \|^2_2 \| \|\hat \Omega - \bar \Omega\|_{2 \rightarrow 2}  \\
& \leq \left ( \max_{i \leq n} | \varepsilon_i^2 | + \max_{i \leq n} |\hat \varepsilon_i^2 - \varepsilon_i^2| \right ) \| \hat u \|_2^2 \| \|\hat \Omega - \bar \Omega\|_{2 \rightarrow 2} \\
&= \( O_p(n^{2/\delta}) + o_p(1) \) O_p(1)( O_p( \zeta_0(K) n^{-1/2}K^{1/2} ) + \Xi_{1} + \Xi_2 + \Xi_3) \\
&= o_p(1).
\end{align*}
where the last bounds come from the rate condition in Assumption 9 and $\max_{i \leq n} |\hat \varepsilon_i^2 - \varepsilon_i^2| = o_p(1)$ by $\max_{i \leq n} |\Delta_{1i}| + |\Delta_{2i}| =o_p(1)$.
\end{proof}

These results give the conclusion that 
$$n^{1/2} \hat V^{-1/2} (\hat \theta - \theta) =  n^{1/2} (F \hat V F)^{-1/2} (\hat \theta - \theta) \convdist N(0,1).$$
Calculations which give the rates of convergence in each of the cases of Assumption 17 or of Assumption 18, as well as the proof of the second statement of Theorem 2, use the same arguments as in \cite{newey:series}.  This concludes the proof.
\QED

\end{document}